\documentclass[12pt]{amsart}
\usepackage[margin=1.25in]{geometry}
\usepackage{amsmath,amsthm,amscd,amsfonts,wasysym,amssymb,epic,eepic,bbm,tikz-cd,mathtools,array}
\usepackage{fullpage}
\newcommand{\WL}[1]{{\color{orange} {#1}}}

\usepackage{amsthm,amsfonts,amssymb,amsmath,amsxtra}

\usepackage{xr-hyper}
\usepackage[pagebackref,colorlinks=
   citecolor=Black,
   linkcolor=Red,
   urlcolor=Blue]{hyperref}
\usepackage{verbatim}

\usepackage[all,arc]{xy}
\SelectTips{cm}{}
\usepackage{enumerate}
\usepackage{mathrsfs}
\usepackage{color}   
\usepackage{hyperref}
\hypersetup{
    colorlinks=true, 
    linktoc=all,     
    linkcolor=blue,  
}
\usepackage{amsthm}
\usepackage{amsmath}
\usepackage{graphicx}
\usepackage{wasysym}
\usepackage{pgf,tikz}
\usetikzlibrary{positioning}
\usepackage{ marvosym }
\usepackage{tikz-cd}
\usepackage{ textcomp }
\usepackage{bbm}
\usepackage{ stmaryrd }

\numberwithin{equation}{section}

\usepackage{mathrsfs}

\RequirePackage{xspace}
\RequirePackage{etoolbox}
\RequirePackage{varwidth}
\RequirePackage{enumitem}
\RequirePackage{tensor}
\RequirePackage{mathtools}
\RequirePackage{longtable}
\RequirePackage{multirow}
\RequirePackage{bigstrut}
\RequirePackage{bigints}

\newcommand{\sP}{\ensuremath{\mathscr{P}}\xspace}

\newcommand{\fkh}{\ensuremath{\mathfrak{h}}\xspace}

\newcommand{\W}{\mathbf{W}}
\newcommand{\calw}{\mathcal{W}}

\newcommand{\cc}{\mathbb C}
\newcommand{\ff}{\mathbb F}

\newcommand{\zz}{\mathbb Z}

\newcommand{\rr}{\mathbb R}
\newcommand{\A}{\mathbb A}

\newcommand{\la}{\langle}
\newcommand{\ra}{\rangle}
\newcommand{\lra}{\longrightarrow}
\newcommand{\hra}{\hookrightarrow}

\newcommand{\al}{\alpha}

\newcommand{\ga}{\gamma}
\newcommand{\de}{\delta}
\newcommand{\De}{\Delta}
\newcommand{\Ga}{\Gamma}
\newcommand{\lam}{\lambda}

\newcommand{\vp}{\varpi}
\newcommand{\sig}{\sigma}
\newcommand{\ka}{\kappa}

\DeclareMathOperator{\G}{G}

\DeclareMathOperator{\T}{T}

\DeclareMathOperator{\sspan}{span}

\DeclareMathOperator{\coker}{coker}

\DeclareMathOperator{\I}{I}
\DeclareMathOperator{\Q}{Q}
\DeclareMathOperator{\supp}{supp}

\newcommand{\fg}{\mathfrak g}

\newcommand{\fc}{\mathfrak c}

\newcommand{\fh}{\mathfrak h}

\newcommand{\fp}{\mathfrak p}

\newcommand{\fgl}{\mathfrak{gl}}
\newcommand{\fX}{\mathfrak X}

\newcommand{\calh}{\mathcal{H}}

\newcommand{\calo}{\mathcal{O}}

\newcommand{{\rP}}{\mathrm{P}}
\newcommand{{\rM}}{\mathrm{M}}
\newcommand{{\X}}{\mathrm{X}}
\newcommand{\calv}{\mathcal{V}}

\newcommand{\Gm}{\mathbb{G}_m}

\newcommand{\iso}{\overset{\sim}{\longrightarrow}}
\newcommand{\car}{\textbf{car}}

\newcommand{\Fbar}{\overline{F}}
\newcommand{\bfun}{\mathbf{1}}
\newcommand{\rH}{\mathrm{H}}

\makeatletter
\def\Ddots{\mathinner{\mkern1mu\raise\p@
\vbox{\kern7\p@\hbox{.}}\mkern2mu
\raise4\p@\hbox{.}\mkern2mu\raise7\p@\hbox{.}\mkern1mu}}
\makeatother

\newenvironment{psmatrix}
  {\left(\begin{smallmatrix}}
  {\end{smallmatrix}\right)}

\newtheorem{Thm}{Theorem}[section]
\newtheorem{Prop}[Thm]{Proposition}
\newtheorem{LemDef}[Thm]{Lemma/Definition}
\newtheorem{Lem}[Thm]{Lemma}
\newtheorem{Cor}[Thm]{Corollary}

\newtheorem{Conj}[Thm]{Conjecture}

\newtheorem{Assumption}[Thm]{Assumption}

\theoremstyle{definition}
\newtheorem{Def}[Thm]{Definition}

\theoremstyle{remark}
\newtheorem{Rem}[Thm]{Remark}

\theoremstyle{definition}

\newcommand{\quash}[1]{}

\newcommand{\BA}{\ensuremath{\mathbb {A}}\xspace}

\newcommand{\BC}{\ensuremath{\mathbb {C}}\xspace}

\newcommand{\BE}{\ensuremath{\mathbb {E}}\xspace}
\newcommand{\BF}{\ensuremath{\mathbb {F}}\xspace}
\newcommand{\BG}{\ensuremath{\mathbb {G}}\xspace}
\newcommand{\BH}{\ensuremath{\mathbb {H}}\xspace}

\newcommand{\BR}{\ensuremath{\mathbb {R}}\xspace}

\newcommand{\BV}{\ensuremath{\mathbb {V}}\xspace}
\newcommand{\BW}{\ensuremath{\mathbb {W}}\xspace}

\newcommand{\BZ}{\ensuremath{\mathbb {Z}}\xspace}

\newcommand{\CO}{\ensuremath{\mathcal {O}}\xspace}

\newcommand{\CS}{\ensuremath{\mathcal {S}}\xspace}

\newcommand{\Y}{\ensuremath{\mathrm {Y}}\xspace}

\DeclareMathOperator{\End}{End}

\DeclareMathOperator{\Gal}{Gal}
\newcommand{\GL}{\mathrm{GL}}

\DeclareMathOperator{\Hom}{Hom}

\newcommand{\inv}{{\mathrm{inv}}}

\DeclareMathOperator{\Lie}{Lie}

\DeclareMathOperator{\Nm}{Nm}

\DeclareMathOperator{\Orb}{Orb}

\renewcommand{\Re}{{\mathrm{Re}}}

\DeclareMathOperator{\Res}{Res}

\DeclareMathOperator{\Spec}{Spec}

\newcommand{\SO}{{\mathrm{SO}}}

\DeclareMathOperator{\Sym}{Sym}
\DeclareMathOperator{\sgn}{sgn}

\newcommand{\U}{\mathrm{U}}

\DeclareMathOperator{\vol}{vol}

\newcommand{\Herm}{\mathrm{Herm}}

\newcommand{\wt}{\widetilde}
\newcommand{\wh}{\widehat}

\newcommand{\ov}{\overline}
\newcommand{\ul}{\underline}

\newcommand{\bs}{\backslash}

\newcommand{\ep}{\varepsilon}

\newtheorem{theorem}{Theorem}

\theoremstyle{definition}

\newtheorem{remark}[theorem]{Remark}

\numberwithin{equation}{section}
\numberwithin{theorem}{section}

\setitemize[0]{leftmargin=*,itemsep=\the\smallskipamount}
\setenumerate[0]{leftmargin=*,itemsep=\the\smallskipamount}

\let\shortmapsto\mapsto
\renewcommand{\mapsto}{%
   \ifbool{@display}{\longmapsto}{\shortmapsto}%
   }
\newlength{\olen}
\newlength{\ulen}
\newlength{\xlen}
\newcommand{\xra}[2][]{%
   \ifbool{@display}%
      {\settowidth{\olen}{$\overset{#2}{\longrightarrow}$}%
       \settowidth{\ulen}{$\underset{#1}{\longrightarrow}$}%
       \settowidth{\xlen}{$\xrightarrow[#1]{#2}$}%
       \ifdimgreater{\olen}{\xlen}%
          {\underset{#1}{\overset{#2}{\longrightarrow}}}%
          {\ifdimgreater{\ulen}{\xlen}%
             {\underset{#1}{\overset{#2}{\longrightarrow}}}
             {\xrightarrow[#1]{#2}}}}%
      {\xrightarrow[#1]{#2}}
   }
\makeatother
\newcommand{\xyra}[2][]{%
   \settowidth{\xlen}{$\xrightarrow[#1]{#2}$}%
   \ifbool{@display}%
      {\settowidth{\olen}{$\overset{#2}{\longrightarrow}$}%
       \settowidth{\ulen}{$\underset{#1}{\longrightarrow}$}%
       \ifdimgreater{\olen}{\xlen}%
          {\mathrel{\xymatrix@M=.12ex@C=3.2ex{\ar[r]^-{#2}_-{#1} &}}}%
          {\ifdimgreater{\ulen}{\xlen}%
             {\mathrel{\xymatrix@M=.12ex@C=3.2ex{\ar[r]^-{#2}_-{#1} &}}}
             {\mathrel{\xymatrix@M=.12ex@C=\the\xlen{\ar[r]^-{#2}_-{#1} &}}}}}%
      {\mathrel{\xymatrix@M=.12ex@C=\the\xlen{\ar[r]^-{#2}_-{#1} &}}}%
   }
\makeatletter
\newcommand{\xla}[2][]{%
   \ifbool{@display}%
      {\settowidth{\olen}{$\overset{#2}{\longleftarrow}$}%
       \settowidth{\ulen}{$\underset{#1}{\longleftarrow}$}%
       \settowidth{\xlen}{$\xleftarrow[#1]{#2}$}%
       \ifdimgreater{\olen}{\xlen}%
          {\underset{#1}{\overset{#2}{\longleftarrow}}}%
          {\ifdimgreater{\ulen}{\xlen}%
             {\underset{#1}{\overset{#2}{\longleftarrow}}}
             {\xleftarrow[#1]{#2}}}}%
      {\xleftarrow[#1]{#2}}
   }
\newcommand{\isoarrow}{%
   \ifbool{@display}{\overset{\sim}{\longrightarrow}}{\xrightarrow\sim}%
   }

\begin{document}

\title[]{Unitary Friedberg--Jacquet periods and their twists:\\
Relative trace formulas
}
\author{Spencer Leslie}
\address{Boston College, Department of Mathematics, Chestnut Hill, MA 02467 USA}
\email{spencer.leslie@bc.edu}

\author{Jingwei Xiao}
\address{Beijing, China}
\email{jwxiao922@gmail.com}

\author{Wei Zhang}
\address{Massachusetts Institute of Technology, Department of Mathematics, 77 Massachusetts Avenue, Cambridge, MA 02139, USA}
\email{weizhang@mit.edu}

\date{\today}

\begin{abstract}In \cite{LXZfund}, we formulated a global conjecture for the automorphic period integral associated to the symmetric pairs defined by unitary groups over number fields, generalizing a theorem of Waldspurger's toric period for $\GL(2)$. In this paper, we introduce a new relative trace formula to prove our global conjecture under some local hypotheses. A new feature is the presence of the relative endoscopy in the comparison. We also establish several local results on relative characters. 
\end{abstract}

\maketitle

\setcounter{tocdepth}{1}
\tableofcontents

\section{Introduction}

In \cite{LXZfund}, we formulated a global conjecture for the automorphic period integral associated to the symmetric pairs defined by unitary groups over number fields, generalizing a theorem of Waldspurger's toric period for $\GL(2)$.  In this paper, as a sequel to \cite{LXZfund}, we consider certain relative trace formulas and complete the proof of the main global theorems stated in {\em loc. cit.}.

\subsection{Main global results}

We refer to \cite[\S1.1]{LXZfund} for the more general conjecture. In this paper we will focus on the theorems announced in {\it loc. cit.}. 
Let $E/F$ be a quadratic extension of number fields. Let $V$ be an $E/F$-Hermitian space with $\dim_E V=2n$, and 
let $\G = \U(V)$ be the associated unitary group, a reductive group over $F$. In this paper we consider the following two cases of unitary symmetric pairs $(\G,\rH)$.

\begin{enumerate}
\item (The {\em split-inert} case) There is an orthogonal decomposition 
\[
V=W_1\oplus W_2,\qquad \dim(W_1) = \dim(W_2).
\]
and let $\rH=\U(W_1)\times \U(W_2)$; this case is referred to as unitary Friedberg--Jacquet periods. 

\item (The {\em inert-inert} case) There exists a decomposition 
\[
V=U\oplus U^\perp,\qquad \la-,-\ra|_{U}\equiv 0
\]
and let $\rH = \Res_{F/F_0}\GL(U)$, as a maximal Levi subgroup of $\G$; this case is referred to as twisted unitary Friedberg--Jacquet periods. 
\end{enumerate}

Fix such a pair $(\G,\rH)$. Let $\pi$ be a cuspidal automorphic representation of $\G(\BA_F)$, where $\BA_F$ the ring of adeles of $F$. In this paper we are interested in the automorphic $\rH$-period integral, defined by
$$
 \sP_{\rH}(\varphi)=\int_{[\rH]}\varphi(h)\,dh,
$$
where $\varphi\in\pi$ and $[\rH]=\rH(F)\bs \rH(\BA_F)$. We will say that $\pi$ is $\rH$-distinguished if $ \sP_{\rH}(\varphi)\neq0$ for some $\varphi\in\pi$.

 Let ${\rm BC}_E(\pi)$ denote the (weak) base change from $\G$ to $\Res_{E/F}\GL_{2n}$ \cite[Section 4.3]{BPLZZ}. Now let $\G'=\GL_{2n,F}$, and let $\pi_0$ be a cuspidal automorphic representation of $\G'(\BA)$. To abuse notation, we also denote by ${\rm BC}_E(\pi_0)$ the base change from $\GL_{2n,F}$ to $\GL_{2n,E}$ \cite{ArthurClozel}.


To state the main global theorems, we impose the following constraints on our number fields. We assume that $E/F$ is a quadratic extension of number fields such that
\begin{enumerate}
    \item$E/F$ is everywhere unramified,
    \item $E/F$ splits over every finite place $v$ of $F$ such that $$p\leq \max\{e(v/p)+1,2\},$$ where  $v|p$ and $e(v/p)$ denote the ramification index of $v$ in $p$.
\end{enumerate}
We refer to the beginning of \S\ref{Section: global proofs} for some comments on these assumptions.
We note that the unramified hypothesis implies that the number of non-split archimedean places is necessarily even (by computing $\eta_{E/F}(-1)=1$ locally).

We first state our main result in the  split-inert case.

\begin{Thm}\label{Thm: si}
Let $\pi_{0}$ be a cuspidal automorphic representation of $\G'(\A_{F})$ satisfying 
\begin{enumerate}
    \item there is a split non-archimedean place $v_1$ such that $\pi_{v_1}$ is supercuspidal,
    \item there is a split non-archimedean place $v_2$ such that $\pi_{v_2}$ is $\rH'$-elliptic,
        \item for each non-split non-archimedean place $v$, $\pi_v$ is unramified,
        \item for each non-split archimedean  place $v$, $\pi_v$ is the descent of the base change of a representation of the compact unitary group $\U_{2n}(\BR)$ that is distinguished by the compact  $\U_{n}(\BR)\times \U_{n}(\BR)$.
\end{enumerate} Assume that $\pi_0$ is of {\em symplectic type}. Then the following two assertions are equivalent:
\begin{enumerate}
\item $L(1/2,{\rm BC}_E(\pi_{0}) )\neq 0$, and 
\item
There exist (non-degenerate) Hermitian spaces $W_1,W_2$ of dimension $n$, and a cuspidal automorphic representation $\pi$ of $\G(\A_{F})$ for $\G=\U(W_1\oplus W_2)$ satisfying
\begin{enumerate}
    \item ${\rm BC}_E(\pi) ={\rm BC}_E(\pi_{0})$,
    \item $\pi$ is $\rH$-distinguished for $\rH=\U(W_1)\times\U(W_2)$ .
\end{enumerate}
\end{enumerate}
\end{Thm}
\begin{remark}
The notion of being $\rH'$-elliptic is defined in terms of local relative characters in \S \ref{Section: elliptic}; by Theorem \ref{Thm: FJ elliptic}, is suffices to assume $\pi_{v_2}$ is supercuspidal and $\rH'$-distinguished.
\end{remark}

\begin{remark}The condition for $\pi_v$ at a  non-split archimedean  place $v$ is satisfied if  $\pi_{v}$ is a unitary representation of $\GL_{2n}(\BR)$ with the trivial infinitesimal character; its base change to $\GL_{2n}(\BC)$ is be the base change of the trivial representation of the compact unitary group $\U_{2n}(\BR)$.
We refer to Lemma \ref{lem: branc compact} regarding the branching law for the compact unitary groups $(\U_{2n}(\BR),\U_{n}(\BR)\times \U_{n}(\BR))$ and for a precise description on the L-parameters of such representations in its proof. 
\end{remark}

We also have an analogous theorem in the inert-inert case.
\begin{Thm}\label{Thm: ii}
We further assume that every archimedean place $v$ of $F$ splits in $E$.
   Let $\pi_{0}$ be as above. Assume that $\pi_0$ is of {\em symplectic type}. Then the following two assertions are equivalent:
\begin{enumerate}
\item\label{Thm: ii 1} At least one of $L(1/2,\pi_{0})$ or $L(1/2,\pi_{0}\otimes \eta_{E/F})$ is non-zero, and 
\item There exists a cuspidal automorphic representation $\pi$ of $\G(\A_{F})$ for the quasi-split unitary group $\G=\U(V)$ satisfying
\begin{enumerate}
    \item   ${\rm BC}_E(\pi) ={\rm BC}_E(\pi_{0})$, and
    \item $\pi$ is $\rH$-distinguished, where $\rH$ is the Levi subgroup of the Siegel parabolic subgroup associated 
to a Lagrangian subspace of $V$.
\end{enumerate}
\end{enumerate}
\end{Thm}

The two theorems are proved in Sections \ref{Section: proof P to L} and \ref{Section: proof L to P}. We also have slightly stronger results on the direction ``period non-vanishing implies $L$-value non-vanishing" where we may drop the assumption that $\pi_v$ is unramified for each non-split place $v$; see \S\ref{Section: period implies Lvalue}.

\subsection{The relative trace formulas}
To prove our theorems, we follow the general strategy of comparing relative trace formulas.
As is routine now, for the automorphic period integrals we consider the RTF associated to the triple
$$
( \G, \rH,\rH)
$$
where $\G\supset \rH$ are defined in the previous subsection.
For $f\in C^\infty_c(\G(\A_{F}))$, we define
\begin{equation}\label{eqn: u RTF}
J(f):= \int_{[\rH_1]}\int_{[\rH_2]}K_f(h_1,h_2)\,dh_1\,dh_2,
\end{equation}
where $\rH_1,\rH_2\subset \G$ are pure inner forms of $\rH$.

We will compare this RTF to the following relative trace formula. We consider the triple of $F_0$-groups $(\G',\rH'_1,\rH'_2)$ where $\rH'_1=\rH'_2=\rH'$
\begin{align*}
   \G' &:=\GL_{2n},\\
	\rH' &:= \GL_{n}\times\GL_n.
\end{align*}
We need to insert two weight factors. On the first factor $\rH'$ we need the twisted Godement--Jacquet Eisenstein series on the second $\GL_n$-factor in $\rH'$ (see \S\ref{ss:Eis})
$$
E(h, \Phi,s,\eta),\quad \Phi\in \CS(\A_F^{n}), \quad h\in\GL_n(\A_F) 
$$
where $\eta=\eta_{E/F}$.
On the second factor $\rH_2'(\BA_F)$, we need another quadratic character $\eta'$ (it will be either $\eta_{E/F}$ or trivial in this paper). We consider the distribution defined by the integral, for $\wt f=f'\otimes\Phi \in \CS(\G'(\A_F)\times \A_F^{n})$,
 \begin{equation}\label{eqn:linear RTF}
I^{(\eta,\eta')}(\wt f):=\displaystyle\int_{[\rH']}\int_{[\rH']}K_{f'}(h_1,h_2) E(h_{1}^{(2)},\Phi,0,\eta_{E/F})\,\eta'(h_2)\, dh_1\, dh_2,
 \end{equation}
 where $[\rH'] = Z_{\G'}(\A_F)\rH'({F})\backslash\rH'(\A_F)$. See \S\ref{Section: RTF comparison main} for more general formulation and Proposition \ref{Prop: factorize the linear rel char} for the relationship between its spectral expansion and central $L$-values.

For ``elliptic nice" test functions (Definitions \ref{nice test functions} and \ref{elliptic test functions}), we obtain a ``simple trace formula" for each of the two distributions \eqref{eqn: u RTF} and \eqref{eqn:linear RTF}: an equality between the spectral decomposition into a sum over global relative characters of automorphic representations and the geometric expansions into a sum of orbital integrals. See Proposition \ref{Prop: simple RTF linear} (resp. Proposition \ref{Prop: Simple RTF unitary}) for the general linear (resp. the unitary) case. 

To compare the geometric expansions of the two simple trace formulas, we require a stabilized version of the unitary side. Subject to the appropriate transfer and fundamental lemma statements, such a stabilization was formulated in great generality in \cite{LeslieEndoscopy} (see also \cite{Lesliedescent,LeslieUFJFL} for the fundamental lemma and partial results toward smooth transfer in the (pre-)stabilization of the RTF in the split-inert case). In the split-inert case, the ellipticity condition for the test functions causes the geometric side to completely simplify into a sum of {\em stable} orbital integrals, see Proposition \ref{Prop: prestab si}. In the inert-inert case, there remains an ``unstable'' term of $\ep$-orbital integrals in addition to the stable one even in the elliptic regime, see Proposition \ref{Prop: prestab ii}.


We state the necessary comparisons of stable orbits and orbital integrals in \S\ref{Section: orbital integrals} by reformulating the notions from \cite{LXZfund}. Together with the relative fundamental lemmas proved in \emph{loc. cit.} (recalled in \S\ref{Section: fundamental lemmas on variety}),  we arrive at a trace formula identity for matching test functions of the form
$$J(f)=
2I^{\eta,\eta}(\wt{f})
$$
in the split-inert case (cf. Proposition \ref{Prop: weak trace comparison si}), and 
$$J(f)=
I^{(\eta,1)}(\wt f)+
I^{(\eta,\eta)}(\wt f_\ep)
$$
in the inert-inert case (cf. Proposition \ref{Prop: weak trace comparison ii}). Here, the second term arises via the endoscopic comparison to the split-inert case followed by the stable comparison in that case. 
To our knowledge, our theorem in the inert-inert case is the first instance where relative endoscopy is applied to establish a generalization of Waldspurger theorem relating period integrals and special values of $L$-functions. 


 Another novelty of our main theorems is the ambiguity of which $L$-value is non-zero (cf. the general conjecture \cite[Conj. 1.1]{LXZfund}); this feature does not appear in previous related works such as the Jacquet--Rallis RTFs and adds additional arithmetic complexity. Indeed, while assumption \eqref{Thm: ii 1} of Theorem \ref{Thm: ii} asserts that at most one central $L$-value is non-zero, we show under our hypotheses that the non-vanishing of the appropriate global relative character is equivalent to the non-vanishing of the sum
\[
L(1/2,\pi_{0})+L(1/2,\pi_{0}\otimes \eta_{E/F}).
\]
The difference is resolved by appealing to a deep result of Lapid--Rallis \cite{LapidRallis} on the non-negativity of central values of symplectic $L$-functions, which implies that this non-vanishing is equivalent to assumption \eqref{Thm: ii 1} of Theorem \ref{Thm: ii} and allows us to prove the result.

\subsection{The organization of the paper} After setting conventions, we recall the necessary local geometric results from \cite{LXZfund} in \S \ref{Section: orbital integrals}, reformulated in terms of the groups (rather than symmetric varieties) for use in the trace formulas. In \S \ref{Section: spectral prelim FJ}, we recall the spectral preliminaries on various period integrals required to formulate the spectral properties of the RTFs. Due to the use of simple relative trace formulas and our limited knowledge towards the existence of smooth transfers, we need to prove non-vanishing results for local relative characters. We are able to do so in certain special cases (Theorem \ref{Thm: FJ elliptic} and Theorem \ref{thm:arch nonzero lin char}),  resulting the local conditions in the two main theorems. The proofs of these results are given in the two appendices. In \S\ref{Section: RTF}, we establish the necessary properties of the RTFs, including the relationship to central $L$-values (Proposition \ref{Prop: factorize the linear rel char}) and the (pre-)stabilizations (Propositions \ref{Prop: prestab si} and \ref{Prop: prestab ii}). Finally, we compare the RTFs and prove the main theorems in \S\ref{Section: global proofs}.

\subsection*{Acknowledgments}

We would like to thank Sol Friedberg, Jayce Getz, and Tasho Kaletha for helpful conversations regarding this work.

S. Leslie was partially supported by NSF grant DMS-2200852. W. Zhang was partially supported by NSF grant DMS 1901642 and 2401548, and the Simons foundation.

\section{Preliminaries}

All our conventions align with those of \cite[Section 2]{LXZfund}. We recall the important notations here.

\subsubsection{Local fields} Throughout the paper, $F$ is a local or global field of characteristic zero. When $F$ is a non-archimedean field, we set $|\cdot|_F$ to be the normalized valuation so that if $\vp$ is a uniformizer, then
\[
|\vp|^{-1}_F= \#(\calo_F/\fp) = :q
\]
is the cardinality of the residue field. Here $\fp$ denotes the unique maximal ideal of $\calo_F.$ 

For any quadratic \'{e}tale algebra $E$ of a local field $F$, we set $\eta=\eta_{E/F}$ for the character associated to the extension by local class field theory. In particular, if $E$ is not a field, then $\eta_{E/F}$ is the trivial character.

Suppose $\G$ a reductive group over a non-archimedean local field $F$. We let $C^\infty_c(\G(F))$ denote the usual space of compactly-supported locally constant functions on $\G(F)$. If $\Omega$ is a finite union of Bernstein components of $\G(F)$, we denote by $C^\infty_c(\G(F))_{\Omega}$ the corresponding summand of $C^\infty_c(\G(F))$ (for the action by left translation).

Throughout the article, all tensor products are over $\cc$ unless otherwise indicated.

{\subsection{Invariant theory and orbital integrals}\label{Section: orbital integrals conventions}
Let $F$ be a field of characteristic zero, with a fixed algebraic closure $\Fbar$. We consider several pairs $(\rH,\X)$ consisting of a reductive group $\rH$  over $F$ and an affine $F$-variety $X$ equipped with a \emph{left} $\rH$-action. We denote by $\X\sslash \rH=\Spec(k[\X]^{\rH})$ the categorical quotient, and let $\pi_{\X}$ denote the natural quotient map. Recall that every fiber of $\pi_\X$ contains a unique Zariski-closed orbit.

Recall that an element $x\in \X(\Fbar)$ is $\rH$-regular semi-simple if $\pi_\X^{-1}(\pi_\X(x))$ consists of a single $\rH(\Fbar)$-orbit. For any field $F\subset E\subset \Fbar$, $x\in \X(E)$ is $\rH$-regular semi-simple if $x$ is $\rH$-regular semi-simple in $X(\ov{F})$. We denote the set of such elements by $\X^{rss}(E)$.	We will write $\rH_x$ for the stabilizer of $x$; for all the varieties considered in this article, the stabilizer of a regular semi-simple point $x\in \X(F)$ will be connected.
\begin{Def} \label{Def: elliptic}
		We say $x\in \X(F)$ is \emph{$\rH$-elliptic} if $x$ is $\rH$-regular semi-simple and $\rH_x/Z(\rH)$ is $F$-anisotropic.
\end{Def}

For all the varieties considered in this article,  $x\in \X(F)$ $\rH$-regular semi-simple will force $\rH_x\subset \rH$ to be a torus. 
For $x,x'\in \X^{rss}(F)$, $x$ and $x'$ are said to lie in the same \emph{stable orbit} if $\pi_\X(x)=\pi_{\X}(x')$. By definition, there exists $h \in \rH(\ov{F})$ such that $h\cdot x=x'$; this implies that the cocycle
\[
(\sig\mapsto h^{-1}h^\sig)\in Z^1(F,\rH)
\]
lies in $Z^1(F,\rH_x)$. Let $\calo_{st}(x)$ denote the set of rational orbits in the stable orbit of $x$.   If $x'\in\calo_{st}(x)$, then $\rH_x$ and $\rH_{x'}$ are naturally inner forms via conjugation by $g$. In particular, when $\rH_x$ is abelian, $\rH_x\simeq\rH_{x'}$. 
 
 A standard computation shows that $\calo_{st}(x)$ is in natural bijection with
\[
\ker^1(\rH_x,\rH;F):=\ker\left[H^1(F,\rH_x)\to H^1(F,\rH)\right],
\]
where $H^1(F,\rH)$ denotes the first Galois cohomology set for $\rH$. We will have need of abelianized cohomology as well.  When $\rH$ is a reductive group over $F$, $H^1_{ab}(F,\rH)$ denotes the \emph{abelianized} cohomology group of $\rH$ \cite{Borovoi}. Note that $H^1(F,\rH)\simeq H^1_{ab}(F,\rH)$ when $\rH$ is a torus.

}
Now assume that $F$ is a local field. Let $f \in C_c^\infty(\X(F))$ and $x \in \X^{rss}(F)$. We denote the orbital integral of $f$ at $x$ by
\[\Orb^{\rH}(f,x)=\int_{\rH_x(F) \bs \rH(F)}f(h^{-1}\cdot x)dh,\]
which is convergent since the orbit of $x$ is closed, and is well-defined once a choice of Haar measures on $\rH(F)$ and $\rH_x(F)$ is given (see \S \ref{measures} below). With this, we define the \emph{stable orbital integral}
\begin{equation}\label{eqn: stable orbital integral}
    \SO^{\rH}(f,x) := \sum_{y\in \calo_{st}(x)}\Orb^{\rH}(f,y).
\end{equation}
Note that this sum is finite since $F$ is local.

 \subsection{Unitary groups and Hermitian spaces}\label{Section: groups and hermitian} 
For a field $F$ and for $n\geq1$, we consider the algebraic group $\GL_n$ of invertible $n\times n$ matrices. We fix the element
\[
w_n:=\left(\begin{array}{ccc}&&1\\&\Ddots&\\1&&\end{array}\right)\in\GL_n(F).
\]
For any $F$-algebra $R$ and $g\in \GL_n(R)$, we define the involution
\[
g^\theta = w_n{}^tg^{-1}w_n,
\]
 where ${}^tg$ denotes the transpose.
 
 Now suppose that $E/F$ is a quadratic \'{e}tale algebra and consider the restriction of scalars $\Res_{E/F}(\GL_n)$. Then for any $F$-algebra $R$ and $g\in \Res_{E/F}(\GL_n)(R)$, we set $g^\sig$ to denote the induced action of the Galois involution $\sig\in \Gal(E/F)$  associated to $E/F$. We denote by ${T}_n\subset \GL_n$ the diagonal maximal split torus, ${B_n}={T_nN_n}$ the Borel subgroup of upper triangular matrices with unipotent radical ${N_n}$.

Note that $w_n$ is Hermitian in sense that that ${}^tw_n^\sig =w_n$. 
For harmony with later notation, we set $\tau_n=w_n$
and denote by $V_n$ the associated Hermitian space; this distinguishes the quasi-split unitary group $\U_n:=\U(V_n)$.
For spectral reasons, we normalize the space of Hermitian matrices with respect to $w_n$ by setting
\[
\Herm_n(F)=\{\tau\in \fgl_n(E): \tau^\ast:=\tau_n{}^t{\tau}^\sig \tau_n^{-1}=\tau\}.
\]
and denote $\Herm_{n}^{\circ}:=\Herm_n\cap \Res_{E/F}(\GL_n)$. Note that $\GL_n(E)$ acts on $\Herm_{n}^{\circ}(F)$ on the left via 
\[
 g \ast \tau=g \tau{g}^\ast,\quad \tau\in\Herm_{n}^{\circ}(F),\: g\in \GL_n(E).
\]
 We let $\calv_n(E/F)$ be a fixed set of $\GL_n(E)$-orbit representatives. For any $\tau\in\Herm_{n}^{\circ}(F),$  then $\tau\tau_n$ is Hermitian and we denote by $V_\tau$ the associated Hermitian space and $\U(V_\tau)$ the corresponding unitary group. In particular, $V_n=V_{I_n}$. We set
 \[
 U(V_\tau) = \U(V_\tau)(F).
 \]Note that $\mathcal{V}_{n}(E/F)$ gives a complete set of representatives $\{V_\tau: \tau\in \calv_n(E/F)\}$ of the isometry classes of Hermitian vector spaces of dimension $n$ over $E$. When convenient, we will abuse notation and identify this set with $\calv_n(E/F)$. 
We say that $\tau\in \calv_n(E/F)$ is \emph{split} when $\tau= gg^\ast$ for some $g\in \GL_n(E)$, and \emph{non-split} otherwise.

\subsection{Measure conventions}\label{measures} We adopt the measure conventions of \cite{LXZfund}. In particular, when $F$ is $p$-adic and $\psi$ of conductor $\calo_F,$ this choice gives $K:=\GL_n(\calo_F)$ volume $1$. When $E/F$ is also unramified, the same holds for $K_E=\GL_n(\calo_E)\subset \GL_n(E)$ and $\U_n(\calo_F)\subset\U_n(F)$, as well as $\Herm_{n}^{\circ}(\calo_F):=\GL_n(\calo_E)\ast I_n\simeq \GL_n(\calo_E)/\U_n(\calo_F)$. 

When $E/F$ is a quadratic extension of number fields and $\G$ is a reductive group over $F$, we endow adelic group $\G(\A_F)$ with the product measures with respect to our local choices. As always, discrete groups are equipped with the counting measure. 

Finally, we recall our normalization of the measures on various tori which occur in the paper. When $F$ is a $p$-adic field, we endow $\T(F)$ with the Haar measure $dt$ normalized to give the (unique) maximal compact subgroup volume $1$. When $F$ is archimedean, we adopt the conventions of Tate's thesis and endow $\rr^\times$ with the measure $\frac{dx}{|x|}$, $\cc^\times$ with $\frac{2dxdy}{x^2+y^2}$, and $S^1$ with the radial measure $d\theta$ so that $\vol_{d\theta}(S^1) = 2\pi$. 

\subsubsection{Abelian $L$-functions and centralizers}\label{Section: Lvalue measure}
These choices have the following consequence on global measures of regular centralizers. Fix a quadratic extension $E/F$, and let $\eta$ be the associated quadratic character of $\A^\times_F$. {Following the notation of \cite{Xiaothesis}, let $S_1=\{i: F_i\not\supset E\}$ and $S_2=\{i: F_i\supset E\}$. By local class field theory, we see that $i\in S_1$ if and only if $\eta_i:=\eta\circ \Nm_{F_i/F_0}\not\equiv 1$.} 

Let $L(s,\T,\eta):=\prod_i L(s,\eta_i)$ denote the $L$-function associated to the character on $T\subset \GL_n$ induced by $\eta$. When $F$ is a number field, it is clear that the order of the pole of $L(s,\T,\eta)$ at $s=1$ is $|S_2|$.

For a regular $\U(V_\tau)$-semi-simple element $x\in \Herm_{\tau}^\circ(F)$, there is a similar decomposition $$F[x]:=F[X]/(\car_{x}(X))=\prod_{i=1}^mF_i,$$ where $\car_{x}(X)$ denotes the characteristic polynomial of $x$. To construct the centralizer $\T_x\subset \U(V_\tau)$, set $E_i=E\otimes_F F_i$. Then we have
\[
E[x]=\prod_i E_i=\prod_{i\in S_1}E_i\times\prod_{i\in S_2}F_i\oplus F_i,
\]
where $S_1$ and $S_2$ are as above. There is a natural norm map $\Nm: E[x] \lra F[x]$ inducing by the norm map on each factor.
\begin{Lem}\label{Lem: centralizers}
For  a regular $\U(V_\tau)$-semi-simple element $x\in \Herm_{\tau}^\circ(F)$, let $\T_x\subset \U(V_\tau)$ denote the centralizer. Then 
\[
\T_x\cong  \ker[\Res_{E[x]/F}(\Gm)\overset{\Nm}{\lra}\T_x^{op}:=\Res_{F[x]/F}(\Gm)].
\]
\end{Lem}
\begin{Rem}
    We note that this encompasses all regular stabilizers arising from the unitary symmetric varieties.
\end{Rem}

\quash{
Suppose now that $E/F$ is a quadratic extension of number fields and assume that every archimedean place of $F$ splits in $E$. Assume $x\in \Herm_{\tau}^\circ(F)$ is regular semi-simple and consider the short exact sequence
\[
1\lra \T_x \lra \Res_{E[x]/F}(\Gm)\overset{\Nm}{\lra}\T_x^{op}:=\Res_{F[x]/F}(\Gm)\lra 1.
\]
As above, we associate to $x$ the $L$-function $L(s,\T_x^{op},\eta)$.  For each non-split place $v$ of $F$, we similarly consider
\[
F_v[x]\simeq \prod_{i\in S(v)}F_i = \prod_{j\in S_1(v)}F_j\times \prod_{k\in S_2(v)}F_k,
\]
where $S_1(v)=\{i\in S(v): F_i\not\supset E_v\}$ and $S_1(v)=\{i\in S(v): F_i\supset E_v\}$. Note that by assumption $(\T_x)_\infty\simeq (\T_x^{op})_{\infty}$, so we adopt the measure convention at archimedean places above, and equip $\T_x(\A_F)$ with the product measure.

\begin{Lem}\label{Lem: vol of centralizer}\cite[Lemma 2.6]{LXZfund} Let $E/F$ be a quadratic extension of number fields and assume that every archimedean place of $F$ splits in $E$. Assume $x\in \Herm_{\tau}^\circ(F)$ is elliptic, so that $S_2=\emptyset$ and $\T_x$ is anisotropic. With respect to the measure $dt=\prod_vdt_v$ above, 
    \[
    \vol([\T_x]) = 2^{|S_1|-|S_x|}L(0,\T_x^{op},\eta),
    \]
    where $|S_x| = \sum_{v\in S}|S_1(v)|$ and $S$ is the set of places of $F$ that ramify in $E$.
\end{Lem}
}

Suppose now that $E/F$ is a quadratic extension of number fields. Assume $x\in \Herm_{\tau}^\circ(F)$ is regular semi-simple and consider the short exact sequence
\[
1\lra \T_x \lra \Res_{E[x]/F}(\Gm)\overset{\Nm}{\lra}\T_x^{\mathrm{op}}:=\Res_{F[x]/F}(\Gm)\lra 1.
\]
As above, we associate to $x$ the $L$-function $L(s,\T_x^{\mathrm{op}},\eta)$.  For each non-split place $v$ of $F$, we similarly consider
\begin{equation}\label{eqn: decomp of stab at v}
    F_v[x]\simeq \prod_{i\in S_v(x)}F_i = \prod_{j\in S_{v,1}(x)}F_j\times \prod_{k\in S_{v,2}(x)}F_k,
\end{equation}
where $S_{v,1}(x)=\{i\in S_v(x): F_i\not\supset E_v\}$ and $S_{v,2}(x)=\{i\in S_v(x): F_i\supset E_v\}$. We adopt the measure convention at archimedean places above, equip $\T_x(\A_F)$ with the product measure, and $[\T_x]:=\T_x(F)\bs\T_x(\A_F)$ with the quotient measure.

\begin{Lem}\label{Lem: vol of centralizer} Let $E/F$ be a quadratic extension of number fields. Assume $x\in \Herm_{\tau}^\circ(F)$ is elliptic, so that $S_2=\emptyset$ and $\T_x$ is anisotropic. With respect to the measure $dt=\prod_vdt_v$ above, 
    \[
    \vol([\T_x]) = 2^{|S_1|-|S_x|}L(0,\T_x^{\mathrm{op}},\eta),
    \]
    where $|S_x| = \sum_{v\in S}|S_{v,1}(x)|$ and $S$ is the set of places of $F$ that ramify in $E$.\footnote{This includes archimedean places $v$ of $F$ such that $F_v\simeq\rr$ and $E_v\simeq \cc$.}
\end{Lem}

 \section{Geometric preliminaries}\label{Section: orbital integrals}
 In this section, we let $E/F$ be a quadratic extension of local fields and define the relevant orbital integrals on $\G'$ and $\G$ by relating them to the orbital integrals introduced in \cite[Section 4]{LXZfund}. We define the necessary transfer factors and notions of smooth transfer. Finally, we state the relevant results from \cite{LXZfund} on transfer and fundamental lemmas.

We set $\eta=\eta_{E/F}$. For any character $\chi:F^\times\to \cc^\times$ and $h=(h^{(1)},h^{(2)})\in \rH'(F)$, we set
\[
\chi(h) = \chi(\det(h^{(1)})\det(h^{(2)})^{-1}).
\]
\subsection{Recollection of invariant theory} We here recall a few relevant facts on the invariant theory for the symmetric spaces. Proofs may be found in \cite[Section 3]{LXZfund}.

Let $(\G',\rH') =(\GL_{2n},\GL_n\times \GL_n)$ be the split symmetric pair. Set $\X=\G'/\rH'$ and set $s_{\X}:\G'\to \X$ the natural quotient map, which is surjective on $F$-points since $H^1(F,\rH')=0$. Assume that $\rH'\subset \G'$ is embedded in a block diagonal manner. Given an element $x=s_{\X}(g)\in \X(F)$, we may write
$
x= \begin{psmatrix}
    A&B\\C&D
\end{psmatrix},
$ so that $(g,h)\in \rH'(F)$ acts by
\begin{equation*}
   (g,h)\cdot x =\left(\begin{array}{cc}gAg^{-1}&gBh^{-1}\\hCg^{-1}&hDh^{-1}\end{array}\right). 
\end{equation*}
By \cite[Lemma 3.6]{LXZfund}, the $\rH'$-invariant map $\car_{lin}:\X\to \A^n$ given by sending $x$ to the coefficients of the monic polynomial $\det(tI_n-A)$ is a categorical quotient for $(\rH',\X)$.

The geometric expansion of the linear RTF \eqref{eqn:linear RTF general} is indexed not by the $\rH'$-orbits on $\X(F)$, but certain $\rH'(F)$-orbits on $\X(F)\times F_{n}$; here, $F_{n}$ denotes $1\times n$ matrices over $F$. The $\rH'$-action is given by
\[
h\cdot (x,w) = (h\cdot x,w h_2^{-1}),
\]
where $h=(h_1,h_2)\in \rH'(F)$. We say that $(x,w)$ is \textbf{strongly regular} if $x\in \X^{rss}(F)$ and $F_{n}=\sspan\{w, wR(x),\ldots,w R(x)^{n-1}\}$. This locus is denoted $(\X(F)\times F_n)^{sr}$. We say $(\ga,w)\in \G'(F)\times F_{n}$ is strongly regular if $(s_{\X}(\ga),w)$ is.

Now assume that $\rH\subset \G$ is a unitary symmetric pair and set $\Q=\G/\rH$. We generally set $\pi:\G\to \Q$ for the quotient map. To unify notation for our cases of interest, we set $(\G_\bullet,\rH_\bullet)$ to be a unitary pair with $\bullet\in\{si, ii\}$ (for \emph{split-inert} or \emph{inert-inert}).  Set $\X:=\G'/\rH'$ and $\Q_{\bullet} = \G_{\bullet}/\rH_{\bullet}$. We refer the reader to \cite[Section 3]{LXZfund} for the invariant theory of these symmetric varieties. In particular, we have the diagram of categorical quotients
\[
\begin{tikzcd}
\X\ar[dr,"\car_{lin}",swap]&&\Q_\bullet\ar[dl,"\car_{\bullet}"]\\
&\A^n,&
\end{tikzcd}
\]
where $\car_{lin}$ is defined in \cite[Lemma 3.6]{LXZfund} and the categorical quotients $\car_{\bullet}$ are defined in \cite[Lemmas 3.4 and 3.5]{LXZfund}.    We say that regular elements $y\in \X(F)$ and $x\in {\Q}_\bullet(F)$ \emph{match} if
\[
\car_{lin}(y) = \car_{\bullet}(x),
\]
with the natural extension to $\G'(F)$ and $\G(F)$ via pullback.

\subsection{The linear orbital integrals}\label{section: linear orbital ints}
We begin with the orbital integrals that occur in the geometric expansion of the linear RTF introduced in \S \ref{Section: linear RTF}. We allow for more flexibility than will be used later due to possible interest in the more general forms of trace formula \eqref{eqn:linear RTF general}. After several intermediate steps, the final formula is given in Lemma/Definition \ref{LemDef: linear OI}.

 For $i=0,1,2$, fix now characters $\eta_i:F^\times \to \cc^\times$ and set $\underline{\eta} = (\eta_0,\eta_1,\eta_2)$.
  For $(\ga,w)$ strongly regular and $f'\otimes \Phi\in C_c^\infty(\G'(F)\times F_{n})$, we consider the local orbital integral, 
\begin{equation*}
  \Orb^{\underline{\eta}}_{\ul{s}}(f'\otimes\Phi,(\ga,w))  =\int_{\rH'(F)}\int_{\rH'(F)}f'(h_1^{-1}\ga  h_2)\Phi(wh_1^{(2)})\eta_{0}(h_1^{(2)})\eta_{s}(h_1,h_2)dh_1dh_2,
\end{equation*}
where
\begin{equation}\label{eq eta s local}
\eta_{\ul{s}}(h_1,h_2) = |h_1^{(2)}|^{s_0}|h_1|^{s_1}|h_2|^{s_2}\eta(h_1,h_2),
\end{equation}
for $\ul{s}=(s_0,s_1,s_2)\in \cc^3$ and $\eta(h_1,h_2) =\eta_{1}(h_1) \eta_{2}(h_2)$.%

\subsubsection{Transfer factors} 
Note that
\begin{equation*}
 \Orb^{\underline{\eta}}_{\ul{s}}(f'\otimes\Phi,(h_1,h_2)\cdot(\ga,w)) =\eta_0(h_1^{(2)})\eta_{\ul{s}}(h_1,h_2) \Orb^{\underline{\eta}}_{\ul{s}}(f'\otimes\Phi,(\ga,w)).
\end{equation*}
The transfer factor is pulled back from the factor on $\X(F)\times F_n$ (cf. \cite[Definition 4.1]{LXZfund}).

\begin{Def}\label{Def: transfer factor on X}
   Let $(\ga,w)$ be strongly regular, and  write $s_{\X}(\ga)= \begin{pmatrix}
       A&B\\C&D
    \end{pmatrix}\in \X^{rss}(F)$. We define the \textbf{transfer factor} on $\G'(F)\times F_{n}$, as
\begin{align}\label{eqn: transfer factor on X}
\Omega(\ga,w)&:=\omega(s_{\X}(\ga),w)\eta_2(\det(\ga))\\ &= \eta^n_2(BC)({\eta_1\eta_2})(C){\eta_0}(\det(w|wD|\ldots|wD^{n-1}))\eta_2(\det(\ga)).\nonumber
\end{align} 
\end{Def}

Since the factor $\eta_2(\det(\ga))$ is invariant, the following lemma is then an immediate consequence of \cite[Lemma 4.3]{LXZfund}.

\begin{Lem}
  For any strongly regular pair $(\ga,w)$, the transfer factor $\Omega(\ga,w)$ is well-defined. For any $\wt{f}\in C_c^\infty(\G'(F)\times F_{n})$, we have
\[
\omega((h_1,h_2)\cdot(y,w))\Orb^{\underline{\eta}}_{s}(\wt{f},(h_1,h_2)\cdot (\ga,w)) =|h_1^{(2)}|^{s_0}|h_1|^{s_1}|h_2|^{s_2}\Omega(\ga,w)\Orb^{\underline{\eta}}_{s}(\wt{f},(\ga,w))
\]
for any $(h_1,h_2)\in \rH'(F)\times \rH'(F)$.
\end{Lem}

\subsubsection{Orbital integrals on $\G'(F)\times F_{n} $}\label{Section: orbital integrals linear}
{We now make the following additional assumptions (cf. Lemma \ref{cor pole})
\begin{equation}\label{eqn: required to descend to X}
\eta_2^2=1,\quad \;\text{ and }\quad s_2=0.
\end{equation}
\begin{Rem}\label{Ref: no eta1}
    For our applications, we also set $\eta_1=1$. We include this case in the current section for future use in the biquadratic setting. As a result, we will often set $\ul{\eta}=(\eta_0,\eta_2)$ when working with orbital integrals. This is inconsistent with the notation in \S \ref{Section: RTF}, where $\ul{\eta}=(\eta_0,\eta_1,\eta_2)$. The only occurrence of $\eta_1\neq 1$ is in the $L$-factor calculations in Proposition \ref{Prop: factorize the linear rel char} and \S \ref{Section: proof L to P}.
\end{Rem}
Since $\eta_2$ is of order two, we have $\eta_2(h_2)=\eta(\det(h_2))$.
For $(\ga,w)$ strongly regular, set $y=s_{\X}(\ga)\in \X^{rss}(F)$. For $f'\otimes \Phi\in C_c^\infty(\G'(F)\times F_{n})$, we set
\begin{equation}\label{eqn: eta projection}
    s_{\X,!}^{\eta_2}(f')(y):=\int_{\rH'(F)}f'(\ga h_2)\eta_2(\det(\ga h_2)) dh.
\end{equation}
Since $s_{\X}:\G'\to \X$ is an $\rH'$-torsor, it follows that $s^{\eta_2}_{\X,!}(f')\in C_c^\infty(\X(F))$. One thus computes 
\begin{align*}
\eta_2(\ga)\Orb^{\underline{\eta}}_{s}(f'\otimes\Phi,(\ga,w))  &=\int_{\rH'(F)}s^{\eta_2}_{\X,!}(f')(h^{-1}\cdot y)\Phi(wh^{(2)})\eta_{0}(h^{(2)}) (\eta_{1}\eta_2)(h) |h^{(2)}|^{s_0} |h|^{s_1}\,dh\nonumber\\
                                                &=\Orb^{\underline{\eta}}_{s}(s^{\eta_2}_{\X,!}(f')\otimes\Phi,(y,w)).
\end{align*}
where the orbital integral $\Orb^{\underline{\eta}}_{s}(s_{\X,!}^{\eta_2}(f')\otimes\Phi,(y,w))$ is as defined in \cite[Section 4.1]{LXZfund}.

\quash{Setting $\phi= s_{\X,!}^{\eta_2}(f')$ and $y=s_{\X}(\ga)$, we have
\begin{align*}
\Omega(\ga,w)\Orb^{\underline{\eta}}_{s}(f'\otimes\Phi,(\ga,w)) &=\omega(y,w)\Orb^{\underline{\eta}}_{s}(\phi\otimes\Phi,(y,w)).
\end{align*}
Noting that $s_{\X,!}: C_c^\infty(\G'(F))\to C_c^\infty(\X(F))$ is surjective, the right-hand side extends linearly to a  functional on $C_c^\infty(\X(F)\times F_{n})$ satisfying
\[
\omega(h\cdot(y,w))\Orb^{\underline{\eta}}_{s}(\wt{\phi},h\cdot (y,w)) =|h^{(2)}|^{s_0}|h|^{s_1}\omega(y,w)\Orb^{\underline{\eta}}_{s}(\wt{\phi},(y,w)).
\]}
}

For any strongly regular pair $(\ga,w)\in\G'(F)\times F_{n}$, set $y=s_{\X}(\ga)$ and let $T_y\subset \rH'$ denote the stabilizer of $y$ in $\rH'$. For any character $\chi:\rH'(F)\to \cc^\times$, we consider the local (abelian) $L$-function $L(s,T_y,\chi)$ on $T_y(F)$ associated to the restriction of this character.  
{For any $\wt{f}\in C^\infty(\G'(F)\times F_{n})$, consider the \emph{normalized orbital integral} at $(\ga,w)$
\begin{align}\label{eqn: normalized OI two var}
   \Orb_{s_0,s_1}^{\ul{\eta},\natural}(\wt{f},(y,w)):=\Omega(\ga,w)\frac{\Orb^{\underline{\eta}}_{s}(\wt{f},(\ga,w))}{L(s_0, T_y,\eta_0)}.
\end{align}

 \begin{LemDef}\label{LemDef: linear OI}
This normalized integral \eqref{eqn: normalized OI two var} converges for $\Re(s_0)>0$, admits a meromorphic continuation to the whole $s_0$-plane, and is holomorphic at $s_0=0$. 
The value at $s_0=0$ is holomorphic in $s_1\in \cc$ and independent of $w$. We therefore write 
\begin{equation}\label{eqn: linear OI final}
\Orb^{\rH',\ul{\eta}}_{s_1}(\widetilde{f},\ga):=\Orb_{0,s_1}^{\ul{\eta},\natural}(\wt{f},(\ga,w)),
\end{equation} for any $w$ such that $(\ga,w)$ is strongly regular.
\end{LemDef}
\begin{proof}For simplicity, we assume that $\wt{f} = f'\otimes \Phi$. For $\Re(s_0)$ large enough, we see that
\[
\Orb^{\ul{\eta},\natural}_{s_0,s_1}(f\otimes \Phi,(\ga,w)):=\Orb^{\ul{\eta},\natural}_{s_0,s_1}(s^{\eta_2}_{\X,!}(f)\otimes \Phi,(s_{\X}(\ga),w)),
\]
 so that the lemma follows from \cite[Lemma/Definition 4.4]{LXZfund}.
     \end{proof}

\subsubsection{The split case} When $E=F\times F$, we assume $\eta_0=\eta_1=\eta_2=1$. In this setting, pulling back along $s_{\X}$ the statements of \cite[Section 4.1.2]{LXZfund}, we find that
\begin{equation}\label{eqn: orbital int split}
   \Orb^{\rH',\ul{1}}_{s_1}(f\otimes\Phi,\ga)=\Phi(0)   \int_{T_y(F)\backslash\rH'(F)}s_{\X}^{1}(f)(h^{-1}\cdot s_{\X}(\ga))|h|^{s_1}dh.
\end{equation}

\subsection{The unitary orbital integrals and notions of transfer}
Recall our quadratic extension $E/F$ with associated character $\eta:F^\times \to \cc^\times$. Assume that $\rH\subset \G$ is a unitary symmetric pair and set $\Q=\G/\rH$.For $f\in C^\infty(\G(F))$ and $g\in \G(F)$, we consider the standard stable orbital integral 
\[
\SO^{\rH\times\rH}(f,g) = \SO^{\rH}(\pi_!(f),\pi(g))
\]as discussed in \S \ref{Section: orbital integrals conventions}, where $\pi:\G\to \Q$ is the natural quotient map (on the right). Due to certain cohomological complexities, we consider the two cases separately.

\subsubsection{The split-inert case}\label{Section: orbital si} Fix $\tau_1,\tau_2\in \calv_n(E/F)$ and consider the symmetric pair
\begin{align*}
	 \G_{\tau_1,\tau_2} &= \U(V_{\tau_1}\oplus V_{\tau_2}),\\
\rH_{\tau_1,\tau_2} &=\U(V_{\tau_1})\times \U(V_{\tau_2}).
\end{align*}
An elementary Galois cohomology argument shows that there is a disjoint union
\begin{equation}\label{eqn: orbits to sym var}
    \Q_{\tau_1,\tau_2}(F) = \bigsqcup_{(\tau_1',\tau_2')}\G_{\tau_1,\tau_2}(F)/\rH_{\tau_1',\tau_2'}(F),
\end{equation}
where $(\tau_1',\tau_2')$ range over pairs such that there is an isometry $V_{\tau_1}\oplus V_{\tau_2}\simeq V_{\tau_1'}\oplus V_{\tau_2'}$. Let $\pi_{\tau_1',\tau_2'}:\G_{\tau_1,\tau_2}\to \Q_{\tau_1,\tau_2}$ denote the quotient map with fiber $\rH_{\tau_1',\tau_2'}$. In particular, the map
\begin{align*}
  \bigoplus_{(\tau'_1,\tau'_2)}C_c^\infty(\G_{\tau_1,\tau_2}({F}))&\lra C^\infty_c(\Q_{\tau_1,\tau_2}({F}))\\
  (f^{\tau_1',\tau_2'})_{(\tau'_1,\tau'_2)}&\longmapsto \sum_{(\tau_1',\tau_2')}\pi_{\tau_1',\tau_2',!}(f^{\tau_1',\tau_2'})
\end{align*}
 is surjective.

 Note that the categorical quotients of these pure inner twists are all canonically identified, so we set $[\Q_{si}\sslash\rH]$ for this common quotient. The next lemma gives the bound on how many inner twists may be considered at once without redundancy on the categorical quotient.
\begin{Lem}\label{Lem: surjective quotient si} Let $E/F$ be a quadratic extension of fields (local or global).  Fix $\tau_2\in \calv_n(E/F)$. For any $\tau_1\in \calv_n(E/F)$, let $\G^{rrs}_{\tau_1,\tau_2}$ be the points which map under $\pi_{\tau_1,\tau_2}$ to the $\rH_{\tau_1,\tau_2}$-regular semi-simple locus of $\Q_{\tau_1,\tau_2}$.  There is a bijection
    \[
    \bigsqcup_{\tau_1\in \calv_n(E/F)}\bigsqcup_{(\tau_1',\tau_2')}\rH_{\tau_1,\tau_2}(F)\bs\G^{rrs}_{\tau_1,\tau_2}(F)/\rH_{\tau_1',\tau_2'}(F)=\bigsqcup_{\tau_1\in \calv_n(E/F)}\rH_{\tau_1,\tau_2}(F)\bs\Q^{rss}_{\tau_1,\tau_2}(F),
    \]
    where $(\tau_1',\tau_2')$ range over pairs such that there is an isometry $V_{\tau_1}\oplus V_{\tau_2}\simeq V_{\tau_1'}\oplus V_{\tau_2'}$.
\end{Lem}

Combining the lemma with \cite[Lemma 4.6]{LXZfund}, we see that the invariant map $\car_{si}$ from \cite[Lemma 3.4]{LXZfund} induces a map
\begin{equation}\label{eqn: regular orbits on the base}
    \car_{si}:\bigsqcup_{\tau_1\in \calv_n(E/F)}\bigsqcup_{(\tau_1',\tau_2')}\G^{rrs}_{\tau_1,\tau_2}(F)\lra[\Q_{si}\sslash\rH]^{rss}(F),
\end{equation}
satisfying that if $(\tau_1,\tau_2)\neq (\tau_1',\tau_2')$, then 
$
 \car_{si}\left(\G^{rrs}_{\tau_1,\tau_2}(F)\right)\cap \car_{si}\left(\G^{rrs}_{\tau_1',\tau_2'}(F)\right)=\emptyset.
$
This map is surjective when $\tau_2$ is split.

Now let $E/F$ be local fields. In \cite[Section 4.3]{LXZfund}, we defined a notion of transfer defined comparing orbital integrals on $\X(F)\times F_{n}$ to  collections
    \[
    \underline{\phi}=(\phi_{\tau_1})\in \bigoplus_{\tau_1\in \calv_n(E/F)}C_c^\infty(\Q_{\tau_1,\tau_2}(F))
    \]
    with $\tau_2$ fixed. Fix $\tau_2\in \calv_n(E/F)$. We thus consider families $\underline{f}=(f^{\tau_1',\tau_2'})$ with $f^{\tau_1',\tau_2'}\in C^\infty_c(\G_{\tau_1',\tau_2'}(F))$ such that
\begin{equation}\label{eqn: pushforward family local}
   \phi_{\tau_1} = \sum_{(\tau_1',\tau_2')}\pi_{\tau_1',\tau_2',!}(f^{\tau_1',\tau_2'}),
\end{equation}
where the indices running over all pairs $(\tau_1',\tau_2')$ such that there exists $\tau_1$ with 
    \[
    V_{\tau_1}\oplus V_{\tau_2}\simeq V_{\tau_1'}\oplus V_{\tau_2'}.
    \]
We denote this by $\pi_!(\ul{f}) = \ul{\phi}$.

\begin{Def}[split-inert: $\underline{\eta} = (\eta,\eta)$]\label{Def: transfer si} Fix $\tau_2\in \calv_n(E/F)$. We say that $\widetilde{f}\in C^\infty_c(\G'(F)\times F_{n})$ and a collection
    \[
    \underline{f}=(f^{\tau_1',\tau_2'})\in \bigoplus_{\tau_1\in \calv_n(E/F)}\bigoplus_{(\tau_1',\tau_2')}C_c^\infty(\G_{\tau_1,\tau_2}(F))
    \]
    are \emph{$(\eta,\eta)$-transfers} if the functions $\wt{\phi} = s_{\X}^\eta(\wt{f})$ and $\underline{\phi}=\pi_!(\ul{f})$ are $(\eta,\eta)$-transfers in the sense of \cite[Definition 4.11]{LXZfund}. For a fixed $\G_{\tau_1,\tau_2}$, we also say that  $\widetilde{f}\in C^\infty_c(\G'(F)\times F_{n})$ and $f\in C_c^\infty(\G_{\tau_1,\tau_2}(F))$  are $(\eta,\eta)$-transfers (or simply ``transfers", when there is no confusion) if the former and the collection
    $\underline{f}=(f^{\tau_1',\tau_2'})$ defined by the following are $(\eta,\eta)$-transfers 
    $$f^{\tau_1',\tau_2'}= \begin{cases}f, &  (\tau_1',\tau_2')=(\tau_1,\tau_2),\\
  0, &\text{otherwise}.  
\end {cases}
    $$ 
\end{Def}
When $F$ is non-archimedean, a weak form of the existence of the $(\eta,1,\eta)$-transfers is established in \cite[Theorem 4.15]{LXZfund}; see also Theorem \ref{Thm: fundamental lemma si} below.

\subsubsection{A very weak transfer in the split-inert archimedean case}
The following is used in the proof of Theorem \ref{thm:arch nonzero lin char}. We note that the \cite[Theorem 4.14]{LXZfund} proves a stronger result in the non-archimedean case.
\begin{Prop}\label{prop:tr infty}
Let $E/F=\BC/\BR$. Let $f\in C_c^\infty(\G(F))$ have regular semisimple support, then $(\eta,\eta)$-transfers $\wt f'\in C^\infty_c(\G'(F)\times F_{n})$ of $f$ exist. 
\end{Prop}

\begin{proof}We only sketch the argument, as it is standard yet tedious. Given the restriction to the regular semi-simple locus, the question is local (e.g., by a partition of units argument) and we may reduce it to the Lie algebra version. Note that the unitary group $\rH(F)$ is compact and all regular semisimple orbits are elliptic. The compactness of $\rH(F)$ implies that the orbital integral functions (as functions on the invariant quotient $F^n$) exhaust all $C_c^\infty$-functions on the image of the map $\fX(F)^{rrs}\to F^n$ in the base $F^n=\BR^n$. Similarly, the $\rH'(F)$-equivariant map $(\fX'(F) \times F_n)^{rrs}\to F^n$ is locally trivial $\rH'(F)$-bundle and hence the orbital integral functions
also exhaust all $C_c^\infty$-functions on the image in the base. This completes the proof.
\end{proof}

\subsubsection{The inert-inert case}\label{Section: unitary OI ii}We assume that $V_{2n} = L\oplus L^\ast$ is a split Hermitian space with polarization $L\oplus L^\ast$. Then we have
\[
\G_{ii}= \U(V_{2n})\:\text{ and }\: \rH_{ii} = \Res_{E/F}(\GL(L)).
\]
Set $\Q_{ii}:=\G_{ii}/\rH_{ii}$. Since $H^1(F,\rH_{ii})=1$, we have $\Q_{ii}(F)=\G_{ii}(F)/\rH_{ii}(F)$. Combining this with \cite[Lemma 4.7]{LXZfund}, we see that the map
 \begin{equation}\label{Lem: surjective quotient ii}
     \car_{ii}:\G_{ii}(F)\to F^n
 \end{equation}
   is surjective on regular semi-simple orbits.

   \begin{Def}[inert-inert: $\underline{\eta} = (\eta,1)$]\label{Def: transfer ii} We say that $\widetilde{f}\in C^\infty_c(\X(F)\times F_{n})$ and $\phi\in C_c^\infty(\Q_{ii}(F))$
    are \emph{$(\eta,1,1)$-transfers} if the functions $\wt{\phi} = s_{\X}^1(\wt{f})$ and $\phi_{ii}=\pi_{ii,!}(\ul{f})$ are $(\eta,1)$-transfers in the sense of \cite[Definition 4.12]{LXZfund}.
   \end{Def}
When $F$ is non-archimedean, a weak form of the existence of the $(\eta,1,1)$-transfers is established in \cite[Theorem 4.14]{LXZfund}; see also Theorem \ref{Thm: fundamental lemma ii} below.

\subsubsection{Endoscopic comparison}\label{Section: endoscopic}
Finally, we lift the endoscopic transfer to the group level. Consider a $\rH\times \rH$-regular semi-simple element $g\in \G_{ii}(F)$ and set $x=\pi_{ii}(g)\in \Q_{ii}(F)$. Fix representatives $x_\al\in \calo_{st}(x)$ where $[\inv(x,x_\al)]=\al\in H^1(F,\rH_x)$ and write
  \[
  x_\al = \left(\begin{array}{cc}
      a_\al& b_\al\\c_\al&a_\al^\ast
  \end{array}\right)\in \Q_{ii}^{rss}(F).
  \]
The following is contained in \cite[Section 4.2.2]{LXZfund}
\begin{Lem}\label{Lem: varep char}The map
\begin{align}\label{eqn: special kappa}
\varepsilon: H^1(F,\rH_x)&\lra \cc^\times\\
\al&\longmapsto\eta(\det(b_\al)\det(b))^{-1}\nonumber
\end{align}
 gives a non-trivial character $\varepsilon\in H^1(F,\rH_x)^D$.
    \quash{\begin{enumerate}
        \item\label{Lem: building special kappa} For each $\al\in H^1(F,\rH_x)$, $b_\al,c_\al\in \Herm_{n}^{\circ}(F)$. 
        \item We have the equivalence
  \[
  \al\in \ker^1_{ab}(\rH_x,\U(V_b);F)\iff \eta(\det(b_\al)) = \eta(\det(b)).
  \]
  \item The map
\begin{align}\label{eqn: special kappa}
\varepsilon: H^1(F,\rH_x)&\lra \cc^\times\\
\al&\longmapsto\eta(\det(b_\al)\det(b))^{-1}\nonumber
\end{align}
 gives a non-trivial character $\varepsilon\in H^1(F,\rH_x)^D$.
    \end{enumerate}}
\end{Lem}

Fix $\tau_2=\tau_n\in \calv_n(E/F)$ to be split. Recall that $x\in \Q_{ii}^{rss}(F)$ matches $y\in \Q_{\tau_1,\tau_2}^{rss}(F)$ if they have matching invariant polynomials. Comparing Lemma \ref{Lem: surjective quotient si} with \eqref{Lem: surjective quotient ii}, we see that for any $g\in \G_{ii}(F)$ there exists a unique $\tau_1$ and $(\tau_1',\tau_2')$ such that 
\[
V_{\tau_1}\oplus V_{\tau_2}\simeq V_{\tau_1'}\oplus V_{\tau_2'}
\] such that there exists $g'\in \G_{\tau_1,\tau_2}$ such that $x=\pi_{ii}(g)$ matches some $y=\pi_{\tau_1',\tau_2'}(g')\in \Q_{\tau_1,\tau_2}(F)$. When this occurs we write $g\leftrightarrow_{\tau_1}g'$. 
\begin{Def}[$\varepsilon$-endoscopic transfer]\label{Def: varepsilon transfer}
 Define the transfer factor $\De_{\varepsilon}: \G_{ii}^{rrs}(F)\to \cc^\times$ by
\[
       \De_{\varepsilon}(g):= \eta(\det(b)), \quad\text{ where   $\pi_{ii}(g) = \begin{psmatrix}
      a& b\\c&a^\ast
  \end{psmatrix}$. }
\]
Assume $\tau_2=\tau_n\in \calv_n(E/F)$ is split. We call $f_{ii}\in C^\infty_c(\G_{ii}(F))$ and a family $\ul{f}=(f^{\tau_1',\tau_2'})$ as in \S \ref{Section: orbital si} 
\textbf{$\varepsilon$-transfers} if $\phi=\pi_{ii,!}(f_{ii})$ and $\ul{\phi} = (\phi_{\tau_1})$ are {$\varepsilon$-transfers} in the sense of \cite[Conjecture 4.10]{LXZfund}, where 
$\ul{\phi} = \pi_!(\ul{f})$.
\end{Def}
When $F$ is non-archimedean, a weak form of the existence of the $\varepsilon$-transfers is established in \cite[Theorem 4.15]{LXZfund}; see also Theorem \ref{Thm: fundamental lemma varepsilon} below.

\quash{\begin{proof}
 If we write  $x = \begin{psmatrix}
      a& b\\c&a^\ast
  \end{psmatrix}$,  we set 
    \[
       \De_{\varepsilon}(x):= \eta(\det(b)),
    \]
    and we define $\phi_{\varepsilon}\in C^\infty_c(\Q^{invt}_{ii}(F))$ by $\phi_{\varepsilon}(x)= \eta(\det(b))\phi(x)$.
    Then we compute
    \begin{align*}
        \De_{\varepsilon}(x) \Orb^{\rH,\varepsilon}(\phi,x) &=\eta(\det(b))\sum_{x_\al\in\calo_{st}(x)}\eta(\det(b_\al))\eta(\det(b))^{-1}\Orb^{\rH}(\phi,x_\al)\\
                                                            &=\sum_{x_\al\in\calo_{st}(x)}\Orb^{\rH}(\phi_\varepsilon,x_\al)=\SO^{\rH}(\phi_{\varepsilon},x).
    \end{align*}
    This proves the first claim.

    Now assume that $E/F$ is unramified and set $\phi = \bfun_{\Q_{ii}(\calo_{F})}$. Assume $x\in \Q_{ii}(F)$ is as in the conjecture. Clearly, $\Orb^{\rH,\varepsilon}(\phi,x)=0$ unless $x$ lies in the stable orbit of an element $x_0\in \Q_{ii}(\calo_{F})$, so we may as well assume that $x\in\Q_{ii}(\calo_{F})$. This implies that $\eta(\det(b)) = 1$. \textcolor{red}{Moreover, if $\eta(\det(b_\al))\neq \eta(\det(b))$, then $\Orb^{\rH}(\phi,x_\al)=0$.}\WL{This line is wrong.} Thus,
    \begin{align*}
        \De_{\varepsilon}(x) \Orb^{\rH,\varepsilon}(\phi,x) =\eta(\det(b))\sum_{\stackrel{x_\al\in\calo_{st}(x)}{\eta(\det(b_\al))= \eta(\det(b))}}\Orb^{\rH}(\phi,x_\al)=\SO^{\rH}(\phi_{\varepsilon},x),
    \end{align*}
    proving the second claim.
\end{proof}
}

\subsubsection{The split case}\label{Section: split transfer} For our global applications, we also need the following ``split-split case'' (cf. \cite[Section 4.3.1]{LXZfund}). We assume that $E=F\times F$, so that $\underline{\eta}=(1,1)$. In this case, 
\[
\G= \GL_{2n} = \G', \:\text{and}\:\rH=\GL_n\times \GL_n = \rH',
\]
and the notion of matching of orbital integrals and transfer is trivial. More precisely, recall the formula \eqref{eqn: orbital int split}. If $f\in C^\infty(\G'(F))$ and $\Phi\in C_c^\infty(F_{n})$, we see $\Phi(0)f\in C_c^\infty(\G(F))$ is a smooth $(1,1)$-transfer for $f\otimes \Phi$.

\subsection{The fundamental lemmas}\label{Section: fundamental lemmas on variety}
We now state the various fundamental lemmas needed in \S \ref{Section: global proofs}, proved in \cite{LXZfund}. We thus assume that $E/F$ is an unramified quadratic extension of $p$-adic local fields.  
As in \cite[Section 4.4]{LXZfund}, we adopt the following assumption in this subsection.

\begin{Assumption}
Assume that if $e$ denotes the ramification degree of $F/\Q_p$, then $p>\max\{e+1,2\}$.
\end{Assumption}

\subsubsection{Integral models}\label{Section: integral models} Let $k\in \zz_{\geq0}$. Viewing $(\G',\rH')$ as $\calo_F$-group schemes by using the standard lattice, let $\G'_k$ denote the $k$-th congruence subgroup $\calo_F$-scheme of $\G'$, so that $\G'_k(\calo_F) = I_{2n}+\vp^k\fg'(\calo_F)$. We similarly have the congruence group $\calo_F$-schemes $\rH'_k$, and set $\X_k:=\G'_k/\rH'_k$.

Similarly, assume $(\G,\rH)$ is a unitary symmetric pair consisting of \emph{unramified groups} for which we fix a smooth $\calo_F$-model which we denote by $(\G_0,\rH_0)$. This gives a nice simply-connected symmetric pair over $\calo_F$ in the sense of \cite[Section 4.1]{Lesliedescent}. For each $k> 0$, we also obtain a $k$-th congruence pair of $\calo_F$-schemes $(\G_k,\rH_k)$ (cf. \cite[Definition A.5.12]{KalethaPrasad}); these are all smooth $\calo_F$-schemes with connected special fibers \cite[Proposition A.5.22]{KalethaPrasad}.  Set $\Q_k:=\G_k/\rH_k$.

\begin{Lem}\label{Lem: integral points for all k} For any $k\in \zz_{\geq0}$, we have $$\X_k(\calo_F)=\G'_k(\calo_F)/\rH'_k(\calo_F),\quad\text{and}\quad\Q_k(\calo_F)=\G_k(\calo_F)/\rH_k(\calo_F).$$
\end{Lem}
\begin{proof}
    Note that all the group $\calo_F$-schemes are smooth \cite[Lemma A.5.10]{KalethaPrasad}.  Since the special fibers of $\rH'_{k,\ff}$ and $\rH_{k,\ff}$ are smooth and connected group schemes over the residue field $\ff:=\calo_F/\fp$, Lang's theorem shows that $$H^1(\ff,\rH'_{k,\ff}) =H^1(\ff,\rH_{k,\ff})= 1.$$ Smoothness and Hensel’s lemma now imply that $$H^1_{\acute{e}t}(\Spec(\calo_F),\rH'_k) =H^1_{\acute{e}t}(\Spec(\calo_F),\rH_k)= 1,$$ which suffices to prove the claim.
\end{proof}
\quash{I want to add in the characteristic functions of 
\[
\X_k(\calo_F)=K_k\cap \X(F),
\]
where $K_k = I_{2n}+\vp^k\fg'(\calo_F)$. Note that this is $\G'_k(\calo_F)$ for $\G'_k = \G'_{\calo_F}$. See \cite[Definition A.5.12]{KalethaPrasad} for the unitary case as well. The only step needed is to see that 
\[
H^1_{et}(\calo_F,\rH'_k) = \{1\}.
\] But this should follow as it is defined as the dilatation of a trivial special fiber. By \cite[Proposition A.5.22]{KalethaPrasad}, the special fiber is a vector group over the residue field, so I think it follows.}

Let $\bfun_{\G'_k(\calo_F)}$ (resp., $\bfun_{\G_k(\calo_F)}$) denote the appropriate indicator function for $\G'(F)$ (resp., $\G(F)$). The fundamental lemma for the group is concerned with the orbital integrals of $\bfun_{\G'_k(\calo_F)}\otimes \bfun_{\calo_{F,n}}$ and $\bfun_{\G_k(\calo_F)}$.
\begin{Lem}\label{Lem: push on units}
Let $g\in \G'({F})$ be semi-simple and set $s_{\X}(g)=y\in \X({F})$. For any unramified character $\eta':\rH'({F})\to \cc^\times$, we have
\[
\bfun_{\X_k(\calo_F)}(y) =\frac{1}{\vol(\rH'_k(\calo_F))}\displaystyle\int_{\rH'({F})} \bfun_{\G'_k(\calo_F)}(gh)\eta'(h)dh.
\]
In particular, $s^{\eta'}_{\X,!}(\bfun_{\G'_0(\calo_F)}) = \bfun_{\X_0(\calo_F)}$. Similarly, $\pi_{!}(\bfun_{\G_k(\calo_F)}) = \vol(\rH_k(\calo_F))\bfun_{\Q_k(\calo_F)}$.
\end{Lem}
\begin{proof}
By Lemma \ref{Lem: integral points for all k}, we have $\X_k(\calo_F)= \G'_k(\calo_F)/\rH'_k(\calo_F)$. Thus, for any $g\in \G'({F})$, the left-hand side is non-zero if and only if $gh\in \G'_k(\calo_F)$ for some $h\in \rH'({F})$. We may thus assume $g\in \G'_k(\calo_F)$. The lemma is now clear.
\end{proof}

The following lemma is important for our global applications.
\begin{Lem}\label{Lem: almost everywhere 1}
    Assume that $E/F$ is unramified or split and that $\ga\in \G'(\calo_F)$ is chosen so that $$s_{\X}(\ga)\in \X(\calo_F)=\G'(\calo_F)/\rH(\calo_F)$$ is absolutely semi-simple in the sense of \cite[Section 3]{Lesliedescent}. Then
    \[
    \Orb^{\rH',\ul{\eta}}_{s}(\bfun_{\G'(\calo_F)}\otimes\bfun_{\calo_{F,n}},\ga) =1
    \]
    for all $s\in \cc$.
\end{Lem}
\begin{proof}
    By Lemma \ref{Lem: push on units}, we have
  $ s_{\X}^{\eta_2}(\bfun_{\G'(\calo_F)})= \bfun_{\X(\calo_F)},
    $ 
    so it suffices to prove that if $y\in \X(\calo_F)$ is absolutely semi-simple in the sense of \cite[Section 3]{Lesliedescent}, then
    \[
    \Orb^{\rH',\ul{\eta}}_{s}(\bfun_{\X(\calo_F)}\otimes\bfun_{\calo_{F,n}},y) =1
    \]
    for all $s\in \cc$.

   Let $w\in F_{n}$ be chosen so that $(y,w)$ is strongly regular. It is clear we may choose $w\in \calo_{F,n}$, at which point it easily follows that $\omega(y,w)=1$. In this case, the left-hand side is the specialization at $s_0=0$ of
\[
  \int_{\T_y(F)\backslash\rH'(F)}\bfun_{\X(\calo_F)}(h^{-1} \cdot y)\left(\frac{\int_{T_y(F)}\bfun_{\calo_{F,n}}(w t h^{(2)})|th^{(2)}|^{s_0}\eta_0(t)dt}{L(s_0, T_y,\eta_0)}\right)\eta(h)|h|^{s_1} dh,
\] 
where $\T_y\subset \rH'$ is the stabilizer of $y\in \X(F)$. By \cite[Proposition 4.5]{Lesliedescent}, $T_y$ extends to a smooth group $\calo_F$-torus (in particular, $\T_y$ is unramified) and 
\[
\bfun_{\X(\calo_F)}(h^{-1} \cdot y)=0 \text{ unless }h\in \rH_x(F)\rH'(\calo_F).
\]
Thus, the integral becomes
\[
\frac{\vol(\rH'(\calo_F))}{\vol(T_y(\calo_F))}\frac{\int_{\T_y(F)}\bfun_{\calo_{F,n}}(w t)|t|^{s_0}\eta_0(t)dt}{L(s_0, T_y,\eta_0)} = \frac{\int_{\T_y(F)}\bfun_{\calo_{F,n}}(w t)|t|^{s_0}\eta_0(t)dt}{L(s_0, \T_y,\eta_0)}
\]
by our normalization of measures \ref{measures}. Since $\T_y$ is unramified, the normalization of measures in \S \ref{Section: Lvalue measure} agrees with the normalization in Tate's thesis, so that
\[
\int_{T_y(F)}\bfun_{\calo_{F,n}}(w t)|t|^{s_0}\eta_0(t)dt=L(s_0, T_y,\eta_0)\qedhere
\]
\end{proof}

\subsubsection{The split-inert case}
Recall that $\tau_n=w_n$ is our distinguished split Hermitian form, and $V_n = V_{\tau_n}$. In our setting, $|\calv_n(E/F)|=2$ and we set $\calv_n(E/F)= \{\tau_0, \tau_1\}$, where $\tau_1$ is non-split, and write 
\[
\ul{f}=(f_{0,0},f_{0,1},f_{1,0},f_{1,1})
\]in the sense of Definition \ref{Def: transfer si}. Recall now the congruence subgroups $\calo_F$-schemes $(\G_{si,k},\rH_{si,k})$ from the preceding section.

\begin{Thm}\label{Thm: fundamental lemma si} Fix $\tau_2=\tau_n\in \calv_n(E/F)$.
For all $k\in \zz_{\geq0}$, the functions 
\[
\frac{1}{\vol(\rH'_k(\calo_F))}\bfun_{\G'_k(\calo_F)}\otimes \bfun_{\calo_{F,n}}\quad\text{ and}\quad \left(\frac{(-1)^k}{\vol(\rH_{k}(\calo_F))}\bfun_{\G_{k}(\calo_F)},0,0,0\right)
\]are $(\eta,1,\eta)$-transfers.
    \end{Thm}
    \begin{proof}
     Integrating down to the varieties $\X(F)$ and $\Q_{si}(F)$ and applying Lemma \ref{Lem: push on units}, this reduces to \cite[Theorem 4.21]{LXZfund}.
    \end{proof}

\subsubsection{The inert-inert case}  Recall now the congruence subgroups $\calo_F$-schemes $(\G_{ii,k},\rH_{ii,k})$ from the preceding section.

\begin{Thm}\label{Thm: fundamental lemma ii}
For all $k\in \zz_{\geq0}$, the functions
\[
\frac{1}{\vol(\rH'_k(\calo_F))}\bfun_{\G'_k(\calo_F)}\otimes \bfun_{\calo_{F,n}}\quad\text{ and}\quad \frac{1}{\vol(\rH_{ii,k}(\calo_F))}\bfun_{\G_{ii,k}(\calo_F)}
\]
    are $(\eta,1,1)$-transfers.
    \end{Thm}
        \begin{proof}
     Integrating down to the varieties $\X(F)$ and $\Q_{ii}(F)$ and applying Lemma \ref{Lem: push on units}, this reduces to \cite[Theorem 4.22]{LXZfund}.
    \end{proof}

    \subsubsection{The endoscopic case} Finally, we have the endoscopic fundamental lemma relating the inert-inert case to the split-inert case. Let all notations be as in the previous two subsections.

    \begin{Thm}\label{Thm: fundamental lemma varepsilon}
For all $k\in \zz_{\geq0}$, the functions 
\[
\frac{1}{\vol(\rH_{ii,k}(\calo_F))}\bfun_{\G_{ii,k}(\calo_F)}\quad\text{ and }\quad \left(\frac{(-1)^k}{\vol(\rH_{k}(\calo_F))}\bfun_{\G_{k}(\calo_F)},0,0,0\right)
\]are $\varepsilon$-transfers.
    \end{Thm}

\section{Spectral preliminaries}\label{Section: spectral prelim FJ}
We now return to the global setting and recall the properties of certain period integrals and their relations to (twisted) standard and exterior square $L$-functions. We also recall the necessary local representation theoretic results. When $F$ denotes a fixed number field, let $\A_{F}$ denote its ring of adeles. Let $\G'=\GL_{2n}$, $\rH'=\GL_n\times\GL_n$, and $\pi$ a cuspidal automorphic representation of $\G'(\A_{F})$ with central character $\omega=\omega_{\pi}$.  We use $V_{\pi}$ to denote the vector space of the representation $\pi$. As before, for any character $\chi$, for $h=(h^{(1)},h^{(2)})\in \rH'(A_{F})$ we set $$\chi(h) = \chi(\det(h^{(1)})(\det(h^{(2)})^{-1}).$$
Finally, we write $[\rH'] = Z_{\G'}(\A_{F})\rH'({F})\backslash\rH'(\A_{F})$, and similar notation for other groups. Our measure conventions are found in \S \ref{measures}.

\subsection{Whittaker models}\label{Sec: Whittaker}
Suppose that $F$ is a local field. For any non-trivial additive character $\psi: F\to \cc^\times,$
we denote by $\psi_0$ the generic character of $N_n(F)$
\[
\psi_0(u)=\psi\left(\sum_iu_{i,i+1}\right).
\]

For any irreducible generic representation $\pi$, we denote by $\pi^\vee$ the abstract contragredient representation. 
Set $\calw^\psi(\pi)$ to be the Whittaker model of $\pi$ with respect to the generic character $\psi_0$. The action is given by
\[
\pi(g) W(h) = W(hg),\quad g,h\in \GL_n(F),\: W\in \calw^{\psi}(\pi).
\]
When $\pi$ is unitary, we obtain an isomorphism
\[
\overline{(\cdot)}:\overline{\calw(\pi)}\lra \calw^{\psi^{-1}}(\pi^\vee),
\]
given by $\ov{W}(g) = \overline{W(g)}$.

If $F$ is global, $\pi$ a cuspidal automorphic representation of $\GL_n(\A_{F})$, we denote by $\W^\varphi$ the $\psi_0$-Fourier coefficient of $\varphi\in V_{\pi}$:
\[
\W^\varphi(g) = \int_{[N_n]}\varphi(ng)\psi_0^{-1}(n)dn,
\]
where $\psi_0$ is our generic character of the unipotent subgroup $N_n(\A_{F})$.

Suppose that $S$ is a finite set of places, containing the archimedean ones, such that $\pi_v$ is unramified and $\psi_{0,v}$ has conductor $\calo_{F_v}$ for $v\notin S$. Let $\varphi\in V_\pi$ be such that $W^\varphi$ is factorizable (for simplicity, we will say that $\varphi$ is factorizable), write $\W^\varphi(g) = \prod_vW_v(g_v)$, where $W_v\in \calw^{\psi_v}(\pi_v)$.  We may assume that for all $v\notin S$, $W_v$ is spherical and normalized so that $W_v(I_{n})=1$.

\subsection{Peterson inner product}\label{Section: inner product} 
Suppose $\pi$ is a cuspidal automorphic representation of $\GL_n(\A_{F})$, and let $\pi^\vee\simeq \ov{\pi}$ denote the contragredient representation of $\pi$. Consider the inner product 
\begin{equation}\label{eqn: global inner product}
    \la\varphi,{\varphi}'\ra_{Pet} = \int_{Z_n(\A_F)\GL_n(F)\backslash\GL_n(\A_F)}\varphi(g)\ov{\varphi'(g)}dg;
\end{equation}
this is a $\GL_n(\A_{F})$-invariant inner product on $\pi.$ 

For each place $v$ of $F$, recall the inner product
\begin{equation*}
(W_v,{W}_v')_{\pi_v}=\int_{N_n(F_v)\backslash P_n(F_v)}W_v(h)\ov{W'_v(h)}dh,
\end{equation*}
where $W_v,W_v'\in \calw^{\psi_v}(\pi_v)$ and $P_n\cong \GL_{n-1}\rtimes \mathbb{G}_a^{n-1}$ is the mirabolic subgroup of $\GL_n.$ 
When $\pi_v$ is unramified and $W_v$ is the spherical vector normalized so that $W_v(I_{n}) =1,$ then
\begin{equation}\label{eqn: unramified Whittaker inner}
(W_v,{W}_v)_{\pi_v}=L(1,\pi_v\times \pi_v^\vee),
\end{equation}
where $L(s,\pi_v\times \pi_v^\vee)$ denotes the local Rankin-Selberg $L$-factor.
\begin{Prop}\cite[Section 3.1]{ZhangRankin}
Assume that $\varphi\in V_\pi$ is factorizable as above. There is a corresponding factorization
\begin{align}\label{Prop: nice inner product part 1}
\la\varphi,{\varphi}\ra_{Pet}&=\frac{n\Res_{s=1}L(s,\pi\times\pi^\vee)}{\vol(F^\times\bs\A^1_{F})}\prod_{v}(W_v,{W}_v)^\natural_{\pi_v},
\end{align}
where \begin{equation}
    (W_v,{W}_v)^\natural_{\pi_v}=\frac{(W_v,{W}_v)_{\pi_v}}{L(1,\pi_v\times\pi^\vee_v)}.
\end{equation}
\end{Prop}

\subsection{An Eisenstein series}\label{ss:Eis}
{Let $\eta$ be a Hecke character such that $\eta|_{\BR_+}=1$ under the  usual decomposition $\BA^\times_{F}=\BA^1_{F}\times \BR_+$ (for example, if $\eta$ is of finite order).
We  define
 \begin{align}\label{eqn: our Eisenstein}
 E(h,\Phi,s,\eta):=&|\det(h)|^s\eta(\det(h))\int_{[\GL_1]}\sum_{0\neq v\in F_{n}}\Phi(avh)|a|^{ns}\eta(a^n)da
 \\=&\int_{[\GL_1]}  \sum_{0\neq v\in F_{n}}\Phi(avh)|\det(ah)|^{s}\eta(\det(ah))da  \notag
 \end{align}
 where $\Phi\in \CS(\A_{F}^n)$ is in the space of Schwartz functions, and $h\in \GL_n(\A_{F})$. 

 We have the following proposition recording the basic properties of the Eisenstein series.
 \begin{Prop}\cite[Proposition 2.1]{CogdellGLn}\label{Prop: Eisenstein properties}
The Eisenstein series $ E(h,\Phi,s,\eta)$ converges absolutely whenever $\mathrm{Re}(s)>1$ and admits a meromorphic continuation to the all of $\cc$. It has (at most) simple poles $s=0,1$.

As a function of $h$, it is smooth of moderate growth and as a function of $s$ it is bounded in vertical strips away from the poles, uniformly in $g$ in compact sets. Moreover, we have the functional equation
\[
 E(h,\Phi,s,\eta) =  E({}^th^{-1},\widehat{\Phi},1-s,\eta^{-1}),
\]
where $\widehat{\Phi}=\mathcal{F}_{n}(\Phi)$ and $\mathcal{F}_{n}:\mathcal{S}(\A_{F}^n)\lra \mathcal{S}(\A_{F}^n)$ is the Fourier transform.
The residue at $s=1$ (resp. $s=0$) is 
\[
 \eta(h)\frac{\int_{[\GL_1]^1}\eta(a)^n da}{n}\wh{\Phi}(0)\quad
 \left(\text{resp.} \quad
 \eta(h)\frac{\int_{[\GL_1]^1}\eta(a)^n da}{n}\Phi(0) \right).
 \]
 \end{Prop}
 }

 \subsection{Friedberg--Jacquet periods}  
Let $\pi$ denote an irreducible cuspidal automorphic representation of $\G'(\A_{F})$ with central character $\omega$. We consider the \emph{Friedberg--Jacquet period} from \cite{FriedbergJacquet}
\begin{equation}\label{eqn: Friedberg-Jacquet integral}
Z^{FJ}(s,\varphi,\eta,\eta'):=    \displaystyle\int_{[\rH']}\varphi\left(\iota(h^{(1)}
      , h^{(2)})\right) |h|^{{s}}\eta(h)\eta'(h^{(2)})dh,
\end{equation}
 where $h=(h^{(1)},h^{(2)})$, $\iota:\rH'\hra \G'$ is a particular embedding (cf. \cite{MatringeBF}), and $\eta'(h^{(2)})= \eta'(\det(h^{(2)}))$. When $\eta'=1$, we simplify notation by setting $$Z^{FJ}(s,\varphi,\eta):=Z^{FJ}(s,\varphi,\eta,1).$$ This is convergent for all $s\in \cc$ and is not identically zero only if the Shalika functional (see \cite{JacquetShalika}) is not identically zero on $\pi$; in particular, it is necessary that ${\eta'}^n\omega=1$. \quash{We assume this is the case; by \cite[Proposition 2.3]{FriedbergJacquet}, this integral unfolds to
\begin{equation}\label{eqn: unfold FJ}
    \int_{\GL_n(\A_{F})}V_{\varphi,{\eta'}}\left(\begin{array}{cc}
     g&  \\
      & I_n
 \end{array}\right) |\det(g)|^{\WL{s}}{\eta}(\det(g))dg
\end{equation}
This integral is a holomorphic multiple of $L({s+1/2},\pi_0\otimes {\eta})$, and it is shown in \cite{FriedbergJacquet} that there exists a choice of $\varphi\in V_{\pi_0}$ such that this integral equals this $L$-function. }
Combining Proposition 2.3 and Theorem 4.1 of \cite{FriedbergJacquet} with \cite[Section 8, Thm. 1]{JacquetShalika}, we have the following characterization of these periods. 
\begin{Prop}\label{Prop: FJ periods} Let $\pi$ be an irreducible cuspidal automorphic representation of $\G'(\A_{F})$ such that  ${\eta'}^n\omega=1$. The integral $Z^{FJ}(s,\varphi,{\eta},{\eta'})$ is non-vanishing at ${s=0}$ for some $\varphi\in V_\pi$ if and only if $L(s',\pi,\wedge^2\otimes {\eta'})$ has a simple pole at $s'=1$ and $L(1/2,\pi\otimes{\eta})\neq 0$.
\end{Prop}
Now assume that $\eta'=1$; we consider a factorization of $Z^{FJ}(s,\varphi,\eta)$ as in \cite[Appendix A]{Xuedichotomy}. For any place $v$ of $F$, consider the Whittaker model $\mathcal{W}^{\psi_v}(\pi_{v})$ as in \S \ref{Section: inner product}. For $W_v\in\mathcal{W}^{\psi_v}(\pi_{v})$, consider the integral
\[
Z^{FJ}_v(s,W_v,\eta_v):= \int_{(\rH'(F_v)\cap N_{2n}(F_v))\bs (\rH'(F_v)\cap P_{2n}(F_v))} W_v\left(\iota(h^{(1)},p)\right)\eta_v(h^{(1)}p^{-1})|h^{(1)}p^{-1}|_v^sdh_1d_Rp,
\]
where $P_{2n}\subset \G'$ is the mirabolic subgroup and $d_Rp$ is a right Haar measure on $P_{2n}(F_v)$. By \cite[Proposition 4.18]{MatringeBF}, this gives a $(\rH(F),\eta_v|\cdot|_v^{-s})$-invariant  functional on $\pi_{v}$. 

Normalizing this to $$Z^{FJ,\natural}_v(s,W_v,\eta_v):=\frac{Z^{FJ}_v(s,W_v,\eta_v)}{L(s+1/2,\pi_{v}\otimes \eta_v)L(1,\pi_{v},\wedge^2)},$$ it is shown in \cite[Appendix A]{Xuedichotomy} that $Z^{FJ,\natural}_v(s,W_v,\eta_v)=1$ when all data is unramified and that
\begin{equation}\label{eqn:factorize FJ}
    Z^{FJ}(s,\varphi,\eta) = \frac{n}{\vol(F^\times\bs\A^1_{F})}L(s+1/2,\pi\otimes \eta)\Res_{s=1}L(s,\pi,\wedge^2)\prod_v Z^{FJ,\natural}_v(s, W_v, \eta_v).
\end{equation}

\subsection{The Bump--Friedberg integral}\label{Section: BF integral}
 Let $\eta$ and $\eta'$ be two quadratic characters of $F^\times\bs\A_{F}^1$, and let $(s_0,s_2)\in \cc^\times$. Following \cite{BumpFriedberg} and \cite[Appendix A]{Xuedichotomy}, we also consider the integral
\begin{align}\label{eqn: BumpFriedberg}
Z^{BF}(\varphi,\Phi,\eta,\eta',s_0,s_1):&=\displaystyle\int_{[\rH']}\varphi(\iota(h))\eta(h)|h|^{{s_1}}E(h^{(2)},\Phi,{s_0},\eta')\, dh,
\end{align}
where $\varphi\in V_\pi$. Note the above integral vanishes unless $\omega=1$, which we assume for the remainder of the section. The integral $Z^{BF}(\varphi,\Phi,\eta,\eta',s_0,s_1)$ is absolutely convergent  away from the poles of the Eisenstein series which has at most a simple pole at $s_0=0$ and $1$ when ${\eta'}^n =1$. When ${\eta'}^n=1$, we have
\begin{equation}\label{BF pole}
 \Res_{s_0=0}Z^{BF}(\varphi,\Phi,\eta,\eta',s_0,s_1) = \frac{1}{2n}\vol(F^\times\backslash \A_{F}^1)\widehat{\Phi}(0)Z^{FJ}(s_1,\varphi,\eta,\eta').   
\end{equation} 

 For each place $v$, we have local integrals  $Z^{BF}_v(W_v,\Phi_v,\eta_v,\eta_v',s_0,s_1)$ 
\begin{equation}\label{eqn: local BF integral}
\displaystyle\int_{(N_n\times N_n)(F_v)\backslash (\GL_n\times \GL_n)(F_v)}W_v(\iota(h_1,h_2))\Phi_v(e_nh_2)\eta_v(h)\eta'_v(h_2)|h_1/h_2|_v^{{s_1}}|h_2|_v^{s_0}\,dh_1dh_2,  
\end{equation}
where $W_v\in \mathcal{W}^{\psi_v}(\pi_{v})$ and $\Phi_v\in C_c^\infty(F_{n})$. When $\eta'=1$, this integral is discussed in  \cite{MatringeBF, MatringeBF2,MatringeSpecialization,Xuedichotomy}; we recall the relevant facts now.

\begin{Prop}\label{Prop: BF}
\begin{enumerate} 
\item For each place $v$, the local integrals  $Z^{BF}_v(W_v,\Phi_v,\eta_v,\eta_v',s_0,s_1)$ are convergent for $\Re(s_0)>1$ and $\Re(s_1)>0$. We may always choose local data so that this integral is non-vanishing. 
\item\label{unramified identity BF} When $v$ is non-archimedean, both $\eta_v$ and $\eta'_v$ are unramified, $W_v$ is the normalized spherical Whittaker function for $\pi_{v}$, and $\Phi_v=\bfun_{\calo_{F_v,n}}$ is the indicator function of the standard lattice, then
\[
Z^{BF}_v(W_v,\Phi_v,\eta_v,\eta'_v,s_0,s_1)=L(s_1+\frac{1}{2},\pi_v\otimes \eta_v)L(s_0,\pi_v,\wedge^2\otimes\eta_v').
\]
\item If $\varphi\in V_\pi$ has a factorizable Whittaker function $W^\varphi=\prod_vW_v$, then
\[
Z^{BF}(\varphi,\Phi,\eta,\eta',s_0,s_1)=\prod_vZ^{BF}_v(W_v,\Phi_v,\eta_v,\eta'_v,s_0,s_1)
\]
when $\Re(s_0)>1$ and $\Re(s_1)>0$.
\item \label{Lem: BF local identity on functionals}
   Assume $\pi_{v}$ is irreducible and generic. Assume that $\eta'=1$. For $W_v\in \mathcal{W}^{\psi_v}(\pi_{v})$, we have the identity
    \[
 Z^{BF}_v(W_v,\Phi_v,\eta_v,1,1/2 ,0)= \widehat{\Phi}(0)  Z^{FJ}_v(0,W_v,\eta_v).
    \]
\end{enumerate}
\end{Prop}
\begin{proof}
    When $\eta'=1$, this material is contained in \cite[Appendix]{Xuedichotomy}, where the notations satisfy $Z^{BF}(\varphi,\Phi,\eta,1,2s,{s-1/2}) = I(s,\varphi, \eta, \Phi)$. When $\eta'\neq1$, the proofs of \emph{ibid.} go through without change (relying on Proposition \ref{Prop: Eisenstein properties}), except for the unramified calculation \eqref{unramified identity BF}. When $\eta'=1$, this calculation is given in \cite[Section 3.2]{MatringeSpecialization}. When $\eta'\neq1$, let $\eta_v'=|\cdot|_v^{a}$, then 
by the definition we have 
$$
Z^{BF}_v(W_v,\Phi_v,\eta_v,\eta'_v,s_0,s_1)= Z^{BF}_v(W_v,\Phi_v,\eta_v,1,s_0+a,s_1)
$$
Applying \cite[Section 3.2]{MatringeSpecialization} for the $\eta=1$ case now gives
$$
Z^{BF}_v(W_v,\Phi_v,\eta_v,\eta'_v,s_0,s_1)= L({s_1+1/2},\pi_v\otimes \eta_v)L(s_0+a,\pi_v,\wedge^2)
 $$
 It is easy to check 
$$
L(s_0+a,\pi_v,\wedge^2)=L(s_0,\pi_v,\wedge^2\otimes\eta_v').
$$
Indeed, if $\{\alpha_i: i=1,\cdots, 2n\}$ are the Satake parameters for $\pi_v$, then
  \begin{align*}
L( s_0+{a},\pi_v, \wedge^2)&=\prod_{i\neq j}(1-\alpha_i \alpha_j q^{-(s_0+{a})})^{-1}\\
&= \prod_{i\neq j}(1-\alpha_i \alpha_j \eta_v(\varpi)q^{-s_0})^{-1}\\
&=L(s_0,\pi_v, \wedge^2\otimes \eta_v').
 \end{align*}   
    The final claim regarding the identity of local Bump--Friedberg functionals with local Friedberg--Jacquet functionals is proved in \cite[Appendix A]{Xuedichotomy}.
\end{proof}

For each local place $v$, define
\[
Z^{BF,\natural}_v(W_v,\Phi_v,\eta_v,\eta_v',s_0,s_1)=\frac{Z^{BF}_v(W_v,\Phi_v,\eta_v,\eta_v',s_0,s_1)}{L({s_1+1/2},\pi_v\otimes \eta_v)L(s_0,\pi_v,\wedge^2\otimes\eta_v')}.
\]
Then for $Re(s_0)>1$ and $\Re(s_1)>0$, the preceding proposition implies that we have the identity
\begin{equation}\label{eqn: BF Lfunction}
    Z^{BF}(\varphi,\Phi,\eta,\eta',s_0,s_1)=L({s_1+1/2},\pi\otimes \eta)L(s_0,\pi,\wedge^2\otimes\eta')\prod_vZ^{BF,\natural}_v(W_v,\Phi_v,\eta_v,\eta_v',s_0,s_1),
\end{equation}
where $Z^{BF,\natural}_v(W_v,\Phi_v,\eta_v,\eta_v',s_0,s_1)=1$ for almost all places $v$. In particular, the above equality extends meromorphically to all $(s_0,s_1)\in \cc^2$.

For any irreducible admissible $\pi_{v}$, Matringe has verified in the $p$-adic setting \cite{MatringeBF2} when $\eta'=1$ that the local integral \eqref{eqn: local BF integral} is a holomorphic multiple of the corresponding $L$-factor associated to the Langlands parameter. In particular, when $\pi_{v}$ is unramified, the preceding argument extends this statement to $\eta'\neq1$. For the archimedean case, the corresponding statement is proved in \cite{Ishii}. In particular, the $L$-functions appearing in \eqref{eqn: BF Lfunction} agree with the standard $L$-functions appearing in \cite{LapidRallis}.

\subsection{$\rH$-elliptic representations}\label{Section: elliptic} Due to our use of simple trace formulas, we need to impose certain  constraints on our test functions. We now briefly discuss support constraints on local relative characters needed to state our main results. 

We thus assume $F$ is a local field, and $(\G,\rH)$ is a unitary symmetric pair over $F$, and let $\rH_1,\rH_2\subset \G$ be pure inner forms of $\rH$. For an irreducible admissible unitary representation $\pi$ of $\G(F)$, we assume that $\pi$ is $\rH_i$-distinguished for $i=1,2$ in the sense that $\Hom_{\rH_i(F)}(\pi,\cc)\neq0$. Let $\lam_i\in \Hom_{\rH_i(F)}(\pi,\cc)$ be a non-zero invariant functional and consider the corresponding \emph{relative character}
\[
J_{\lam_1,\lam_2}(f) = \sum_{v}\lam_1(\pi(f)v)\overline{\lam_2(v)},
\]
where $f\in C^\infty_c(\G(F))$ the sum ranges over an orthonormal basis of $\pi$ with respect to a fixed inner product (cf. \S \ref{Section: inner product}). Since $\lam_i$ are non-zero functionals, there exists $f$ for which $J_{\lam_1,\lam_2}(f)\neq0$.

Recall from \S \ref{Section: orbital integrals conventions} the $\rH_1\times \rH_2$-elliptic locus 
\[
\G(F)^{ell} =\{x\in \G(F): (\rH_1\times \rH_2)_x \text{ is anisotropic modulo $Z_{\G}$}\}.
\]
We say that $\pi$ is \textbf{$(\rH_1,\rH_2)$-elliptic} if for any non-zero invariant functionals  $\lam_i\in \Hom_{\rH_i(F)}(\pi,\cc)$, there exist $f\in C^\infty_c(\G(F)^{ell})$  such that $J_{\lam_1,\lam_2}(f)\neq0$. When $\rH=\rH_1=\rH_2$, we say $\pi$ is $\rH$-elliptic if it is $(\rH,\rH)$-elliptic.

For the purposes of establishing a simple trace formula, we prove the following statement, strengthening a result in \cite[Appendix]{XueZhang}.
\begin{Thm}\label{Thm: FJ elliptic}
    Let $F$ be a non-archimedean local field of characteristic zero. Let $(\G',\rH')=(\GL_{2n},\GL_n\times \GL_n)$ be the split unitary symmetric pair over $F$. Let $\pi$ be a supercuspidal representation of $\G'(F)$ and assume that $\Hom_{\rH'(F)}(\pi,\cc)\neq0$. Then $\pi$ is $\rH'$-elliptic.
\end{Thm}

For the sake of continuity, we postpone the proof of this result to Appendix \ref{Appendix: elliptic}, as the tools and calculations needed are quite involved and not directly related to our global aims. 

We also prove some results on local relative characters in the archimedean split-inert case when the relevant Hermitian spaces are positive definite in Appendix \ref{Sec compact}.

 \section{Relative trace formulas}\label{Section: RTF}
We now assume $E/F$ is a quadratic extension of number fields and introduce the various relative trace formulas used in our global results. In particular, we establish the required simple trace formulas. This involves establishing the relationship between global relative characters and $L$-values in the linear case, and (pre-)stabilization of the geometric expansion in the unitary cases.
 \subsection{The linear side}\label{Section: linear RTF}
Let $(\G',\rH')= (\GL_{2n}, \GL_n\times \GL_n)$ be the linear symmetric pair. Since the relative trace formula on $\G'$ will have two weight factors, we denote the two copies of $\rH'$ by $\rH'_1$ and $\rH_2'$ to distinguish. 

Suppose that $f'\in C_c^\infty(\G'(\A_{F}))$ and consider the automorphic kernel
\[
K_{f'}(x,y)=\int_{[Z_{\G'}]}\sum_{\ga\in \G'(F)}f'(x^{-1}\ga zy)dz
\]
where $[Z_{\G'}] = Z_{\G'}(F)\bs Z_{\G'}(\A_F)$.
{
Let $\underline{\eta}=(\eta_0,\eta_1,\eta_2)$ be a triple of unitary characters. We set 
\begin{equation}\label{eq def eta 12}
\eta(h_1,h_2):=\eta_{1}(h_1) \eta_{2}(h_2),
\end{equation}
where for any character $\chi$ and $h=(h^{(1)},h^{(2)})\in \rH'(\A)$, we define
$$\chi(h) =\chi(\det(h^{(1)})\det(h^{(2)})^{-1}), \quad \chi(h^{(2)})=\chi(\det( h^{(2)})).
$$
}

\subsubsection{General formulation}\label{Section: RTF comparison main}

{For $\ul{s}=(s_0,s_1,s_2)\in\BC^3$ we use  $\Re(\ul{s})\gg0$ to indicate that $\Re(s_i)\gg0$ for all $i=0,1,2$.
For $\wt{f} \in C^\infty(\G'(\A_F)\times\A_{F,n})$, we consider the distribution $I^{\underline\eta}_\omega(\wt{f},\ul{s})$ defined on pure tensors $\wt{f} = f'\otimes \Phi$ by the integral,
 \begin{equation}\label{eqn:linear RTF general}
\displaystyle\int_{[\rH']}\int_{[\rH']}K_{f'}(h_1,h_2) E(h_{1}^{(2)},\Phi,s_0,\eta_0)|h_1|^{s_1}|h_2|^{s_2}\eta(h_1,h_2)\, dh_1\, dh_2,
 \end{equation}
 where $[\rH'] = Z_{\G'}(\A_{F})\rH'({F})\backslash\rH'(\A_{F})$ and $\Re(\ul{s})\gg0$. }

 \begin{Rem}
 As in Section \ref{section: linear orbital ints}, we will soon set $s_2=0$. We nevertheless keep it in the formulation for potential future applications toward derivatives of $L$-functions.    
 \end{Rem}

\begin{Rem}[Cases of interests]\label{Rem:characters of interest}
 {We will only consider the relative trace formulas when $\eta_1=1$. 
\begin{enumerate}
\item(split-inert) we set
\[
\eta_0=\eta_2= \eta_{E/F}.
\]
\item(inert-inert) we set
\[
\eta_0= \eta_{E/F}, \quad \eta_2=1.
\]
\end{enumerate}
}
\noindent 
\end{Rem}

\subsubsection{Convergence and geometric expansion} 
We introduce the space of \emph{nice} test functions for which convergence of \eqref{eqn:linear RTF general} may be established.  Recall that an element $\ga\in \G'(F)$ (resp. $(\ga,w)\in \G'(F)\times F_{n}$) is \emph{relatively regular semi-simple} (resp. strongly regular) if $s_{\X}(\ga)\in \X(F)$ is regular semi-simple (resp $(s_{\X}(\ga),w)\in \X(F)\times F_{n}$ is strongly regular).
\begin{Def}\label{nice test functions} 
We say $\wt{f}=\bigotimes_v\wt{f}_v\in C_c^\infty(\G'(\A_{F})\times \A_{F,n})$ is an \emph{elliptic nice test function} if 
\begin{enumerate}
\item\label{supercuspidal constraint nice} There exists a non-archimedean place $v_1$ of $F$ and a finite union $\Omega$ of cuspidal Bernstein components of $\G'(F_{v_1})$ such that $\wt{f}_{v_1}=f'_{v_1}\otimes\Phi_{v_1}$ with $f'_{v_1}\in C^\infty_c(\G'(F_{v_1}))_{\Omega}$ and $\Phi_{v_1}(0)\neq 0$.
\quash{\item\label{eqn: regular constraint} For at least one place $v_2\neq v_1$ which splits in $E$, $\wt{f}_{v_2}$ is supported on the strongly regular semi-simple locus of $\G'(F_{v_2})\times F_{v_2,n}$. This place is not required to be non-archimedean.
\end{enumerate}
If we replace the condition \eqref{eqn: regular constraint} with the stronger constraint that
\begin{enumerate}[resume]}
\item\label{eqn: elliptic constraint} For at least one place $v_2\neq v_1$ which splits in $E$, $\wt{f}_{v_2} = f'_{v_2}\otimes\Phi_{v_2}$ satisfies  $\supp({f}'_{v_2}))\subset \G'(F_{v_2})$ is contained in the \emph{$\rH'$-regular elliptic locus} in the sense that
\[
\supp(f'_{v_2})\subset \{\ga\in \G'(F_{v_2}): (\rH'\times \rH')_{\ga} \simeq L_{\ga}^\times\},
\]
where $L_{\ga}/F_{v_1}$ is a degree $n$ field extension and $\Phi_{v_2}(0)\neq 0$.
\end{enumerate}
\end{Def}

\begin{Lem}\label{Lem: converge and expand}
 Suppose that $\wt{f}=\bigotimes_v\wt{f}_v$ is nice. Then as a function on $\rH'(\A_F)\times \rH'(\A_F)$, 
 \begin{equation}\label{eqn: linear kernel sum}
     \sum_{(\ga,w)\in \G'(F)\times F_{n}\setminus{\{0\}}}\wt{f}(h_1^{-1}\ga h_2,w h_1^{(2)})
 \end{equation}
is compactly supported modulo $\rH'(F)\times \rH'(F)$. In particular, $ I^{\underline\eta}(\wt{f},\ul{s})$ converges absolutely for $\Re(s_1)> 1/2$ and possesses the geometric expansion
 \begin{equation*}
   I^{\underline{\eta}}_\omega(\wt{f},\ul{s})=\sum_{(\gamma,w)/\sim} \Orb^{\underline{\eta}}_{\ul{s}}(\wt{f},(\gamma,w)), 
\end{equation*}
where the sum runs over the strongly regular double cosets 
\[
\rH'(F)\backslash (\G'(F)\times F_{n})^{sr}/\rH'(F),
\] with $\ga\in \G'(F)$ $\rH'$-regular elliptic, and where 
\begin{equation}\label{eqn: adelic orbital}
  \Orb^{\underline{\eta}}_{\ul{s}}(\wt{f},(\gamma,w))  =\int_{\rH'(\A)}\int_{\rH'(\A)}\wt{f}(h_1^{-1}\ga  h_2,wh_1^{(2)})\eta_0(h_1^{(2)})\eta_{\ul{s}}(h_1,h_2) \,dh_1\,dh_2.
\end{equation}
\end{Lem}
\begin{proof}  By Assumption \ref{eqn: elliptic constraint}, the sum \eqref{eqn: linear kernel sum} is supported on pairs $(\ga,w)$ with $\ga\in \G'(F)$ regular elliptic. We claim that any such pair lies in strongly regular locus of $\G'(F)\times F_{n}$. By definition, this requires  $\{w,wR(s_{\X}(\ga)),\ldots, wR(s_{\X}(\ga))^{n-1}\}$ to span all of $F_n$. But this subspace is a quotient of the degree-$n$ {field extension} $F[R(s_{\X}(\ga))]\simeq L_{\ga}$ of $F$, so is all of $F_n$. In particular, only strongly regular pairs $(\ga,w)$ appear in the sum \eqref{eqn: linear kernel sum}. 

By \cite[Lemma 3.7]{LXZfund}, the stabilizer of such an element in $\rH'(F)\times \rH'(F)$ is trivial. Thus, we rewrite the sum
\begin{equation}\label{inner sum}
    \sum_{(\ga,w)/\sim}\left(\sum_{(h_1,h_2)\in \rH(F)\times \rH'(F)}\wt{f}(h_1^{-1}\ga h_2,w h_1^{(2)})\right),
\end{equation}
only finitely many terms of the outer sum only has a non-zero contribution by the same argument as \cite[Lemma 2.2]{ZhangFourier}. Moreover, repeating the argument from \emph{ibib.} also shows that for a fixed strongly regular $(\ga,w)$ the function $(h_1,h_2)\mapsto \wt{f}(h_1^{-1}\ga h_2,w h_1^{(2)})$ has compact support on $\rH'(\A_{F})\times\rH'(\A_{F})$. This implies the inner sum is also finite for a fixed pair $(\ga,w)$.

By analytic properties of the Eisenstein series \eqref{eqn: our Eisenstein}, it follows that for $\Re(s_0)>1$, $I^{\underline\eta}(\wt{f},\ul{s})$ is equal to \[
 \iint\left(\sum_{(\ga,w)}\wt{f}(h_1^{-1}\ga z h_2,awh_1^{(2)})\right)\eta_0(ah_1^{(2)})|ah_1^{(2)}|^{s_0}|h_1|^{s_1}|h_2|^{s_2}\eta(h_1,h_2)dadzdh,
\] 
where the sum ranges over $(\ga,w)\in({\G'}(F)\times F_{n})^{sr}$ and where we integrate over \[(h_1,h_2,z,a)\in {[\rH']}\times{[\rH']}\times{[Z_{\G'}]}\times{[\BG_{m,F}]}.\]
 We first affect the change of variables $h_1\mapsto a^{-1}h_1$, and unfold the integration along $a$ with $h_1$. Similarly changing variables $h_2\mapsto z^{-1}h_2$ and unfolding, we obtain
\[
 \int_{\rH(F)\backslash\rH'(\A)}\int_{\rH(F)\backslash\rH'(\A)}\left(\sum_{(\ga,w)}\wt{f}(h_1^{-1}\ga  h_2,wh_1^{(2)})\right)\eta_0(h_1^{(2)})\eta_{\ul{s}}(h_1,h_2)dh_1dh_2,
\] 
where 
$\eta_{\ul{s}}(h_1,h_2) = |h_1^{(2)}|^{s_0}|h_1|^{s_1}|h_2|^{s_2}\eta(h_1,h_2)
$ 
for $\ul{s}=(s_0,s_1,s_2)\in \cc^3$.

Now a simple unfolding argument with the expression \eqref{inner sum} gives
\begin{equation*}
   I^{\underline{\eta}}_\omega(\wt{f},\ul{s})=\sum_{(\gamma,w)/\sim} \Orb^{\underline{\eta}}_{\ul{s}}(\wt{f},(\gamma,w)), 
\end{equation*}
as in the statement of the lemma.
\end{proof}

\subsubsection{Global relative characters}
Set $\ul{\eta} = (\eta_0,\eta_1,\eta_2)$. We now introduce the terms with appear in the spectral expansion of \eqref{eqn:linear RTF general}. 
\begin{Def}
Let $\pi_0$ be an irreducible cuspidal automorphic representation of $\G'(\A_{F})$. We define the (global) \emph{linear relative character} $I^{\underline{\eta}}_{\pi_0}(f,\Phi,\ul{s})$ for $f'\in C_c^\infty(\G'(\A_{F}))$ and $\Phi\in C_c^\infty(\A_{F,n})$ by
\begin{equation}\label{eqn: twisted global bessel}
I_{\pi_0}^{\underline\eta}(f'\otimes\Phi,\ul{s})=\sum_{\varphi}Z^{BF}(\pi_0(f')\varphi,\Phi,\eta_1,\eta_0,s_0,s_1)\ov{Z^{FJ}\left(s_2,{\varphi},\eta_2\right)},
\end{equation}
where the sum runs over an orthonormal basis for $\pi_0$ with respect to the inner product \eqref{eqn: global inner product}. Note that if $I_{\pi_0}^{\underline\eta}(f'\otimes\Phi,s)$ is non-zero for some $f'\otimes \Phi$ and $s$, then $\omega = 1$.  By linearity, this extends to a functional on $\wt{f}\in C^\infty_c(\G'(\A_{F})\times \A_{F,n})$.
\end{Def}

 Comparing the relative character to the properties of Friedberg--Jacquet periods and the Bump--Friedberg integral, we obtain the following corollary. 
\begin{Lem}\label{cor pole}Suppose $\pi_0$ is a cuspidal automorphic representation with trivial central character. 
\begin{enumerate}
    \item $I_{\pi_0}^{\underline\eta}(\wt{f},\ul{s})$ is absolutely convergent and meromorphic in $\underline{s}\in \cc^3$, with possible poles only along the planes $s_0=0$ or $1$.
    \item If there exists $\wt{f}\in C^\infty_c(\G'(\A_{F})\times \A_{F,n})$ such that $I^{\underline\eta}_{\pi_0}(\wt{f},\ul{s})\neq 0$ for {$s_0\neq0,1$} and ${s_2=0}$, then $L(s,\pi_0,\wedge^2)$ has a pole at $s=1$ and $L(\frac{1}{2},\pi_0\otimes \eta_2)\neq 0$.
\end{enumerate}
\end{Lem}{
\begin{proof}
The assumption implies that the integral $Z^{FJ}\left({0},\varphi,\eta_2\right)$ does not vanish for some $\varphi\in V_{\pi_0}$. It follows from Proposition \ref{Prop: FJ periods} that $L(s,\pi_0,\wedge^2)$ has a pole at $s=1$ and $L(\frac{1}{2},\pi_0\otimes \eta_2)\neq 0$.
\end{proof}
}

\begin{LemDef}\label{LemDef: rel character}Assume that the base change $\Pi$ of $\pi_0$ to $\G'(\A_{E})$ is cuspidal. 
    Suppose also that
    \begin{equation}\label{eqn: character assump 1 spectral}
        \text{ $\eta_0=\eta_{E/F}\neq1$ (as in our cases of interest).}
    \end{equation} If there exists $\wt{f}\in C^\infty_c(\G'(\A_{F})\times \A_{F,n})$ such that $I^{\underline\eta}_{\pi_0}(\wt{f},\ul{s})\neq 0$ for {$s_0\neq0,1$} and ${s_2=0}$, then $Z^{BF}(\varphi,\Phi,\eta_1,\eta_0,s_0,s_1)$ is holomorphic for all $\Phi\in C_c^\infty(\A_{F,n})$ and all $(s_0,s_1)\in \cc^2$. {We set  
    \begin{equation}
        I^{\underline\eta}_{\pi_0}(\wt{f},{s}) := I^{\underline{\eta}}_{\pi_0}(\wt{f},0,{s},{0}).
    \end{equation}}
\end{LemDef}
Note that these restrictions on $\ul{\eta}$ are compatible with those of \eqref{eqn: required to descend to X}.
\begin{proof}
    {By assumption, $Z^{BF}(\varphi,\Phi,\eta_1,\eta_0,s_0,s_1)$ is non-vanishing for some $\varphi\in V_{\pi_0}$ and $\Phi\in C_c^\infty(\A_{F,n})$. It is holomorphic at $s_0=0$ when $\eta_0^n\neq1$, so we need only consider the case when $\eta_0^n=1$, where the residue is as in \eqref{BF pole}. In particular, we will show that $Z^{FJ}(s_1,\varphi,\Phi,\eta_1,\eta_0)\equiv0$. By Proposition \ref{Prop: FJ periods}, it suffices to show that $L(s,\pi_0,\wedge^2\otimes \eta_0)$ is holomorphic at $s=1$.
     
   Since $\Pi$ is cuspidal and 
    \[
    L(s,\Pi\otimes\Pi) = L(s,\Pi,\wedge^2)    L(s,\Pi,\Sym^2)
    \]
    has at most a simple pole at $s=1$, each factor has at most a simple pole (note that neither factor vanishes at $s=1$). We also have the product factorization of the exterior square L-function attached to $\Pi$
    \[
    L(s,\Pi,\wedge^2) = L(s,\pi_0,\wedge^2) L(s,\pi_0,\wedge^2\otimes \eta_0).
    \]
  It follows that each factor on the right hand side has a pole of order at most $1$ at $s=1$, and at most one of them has such a pole. Lemma \ref{cor pole} implies $L(s,\pi_0,\wedge^2)$ has a pole at $s=1$. Therefore, $L(s,\pi_0,\wedge^2\otimes \eta_0)$ is holomorphic at $s=1$.}
\end{proof}
\begin{Cor}\label{cor nonvanish}
Keeping the assumptions from Lemma/Definition \ref{LemDef: rel character}, if there exists $\wt{f}\in C^\infty_c(\G'(\A_{F})\times \A_{F,n})$ such that $I^{\underline\eta}_{\pi_0}(\wt{f},{0})\neq0$, then  $L(s,\pi_0,\wedge^2)$ has a pole at $s=1$,  $L(s,\pi_0,\wedge^2\otimes \eta_0)$ is holomorphic at $s=1$, and
{
\[
L(\frac{1}{2},\pi_0\otimes \eta_1)L(\frac{1}{2},\pi_0\otimes \eta_2)\neq 0.
\]
}

In our cases of interest, we have the following (cf. \cite[Conjecture 1.1]{LXZfund}): 
\begin{enumerate}
    \item (split-inert case) when $E_0 =F\times F$, then $L(\frac{1}{2},\mathrm{BC}_E(\pi_0))\neq 0$;
    \item  (inert-inert case) when $E_0 =E$, then $L(\frac{1}{2},\pi_0)\neq 0$.
\end{enumerate}
\end{Cor}
\begin{proof}
    We have already seen in Lemma \ref{cor pole} that $L(\frac{1}{2},\pi_0\otimes \eta_2)\neq 0$.  As we have seen, the assumption implies that $Z^{BF}(\varphi,\Phi,\eta_1,\eta_0,s_0,s_1)$ is holomorphic and non-zero at $s_0=0$. The non-vanishing of $L(\frac{1}{2},\pi_0\otimes \eta_1)$ now follows from \eqref{eqn: BF Lfunction}.
\end{proof}

\subsubsection{Factorization} The periods integrals used to define the relative character $I^{\ul{\eta}}_{\pi_0}$ are Eulerian, which gives the following factorization.

\begin{Prop}\label{Prop: factorize the linear rel char}
    Let $\pi_0$ be a cuspidal automorphic representation of $\G'(\A_{F})$ with trivial central character such that the base change $\Pi$ of $\pi_0$ to $\G'(\A_{E})$ is cuspidal. 
Assume also that $\eta_0=\eta_{E/F}\neq1$. Assume $I^{\underline\eta}_{\pi_0}(\wt{f},s)\neq 0$ for some $\wt{f}=\bigotimes_v\wt{f}_v\in C_c^\infty(\G'(\A_F)\times \A_{F,n})$. We have the equality of distributions
    \[
    I^{\ul{\eta}}_{\pi_0}(\wt{f},0) = \mathcal{L}^{\ul{\eta}}_{\pi_0}\prod_vI^{\ul{\eta},\natural}_{\pi_{0,v}}(\wt{f}),
    \]
    where
    \[
    \mathcal{L}^{\ul{\eta}}_{\pi_0}:=\frac{L(1/2,\pi_0\otimes \eta_1)L(1/2,\pi_0\otimes \eta_2)L(1,\pi_0,\wedge^2\otimes\eta_0)}{L(1,\pi_0,\Sym^2)},
    \]
    and where $I^{\ul{\eta}}_{\pi_{0,v}}$ is the local relative character given on pure tensors $\wt{f}_v= f'_v\otimes \Phi_v$ by
    \begin{equation}\label{eqn: linear local rel char}
    I^{\ul{\eta}}_{\pi_{0,v}}(f'_v\otimes\Phi_v) = \sum_{W_v}\frac{Z^{BF}(\pi_{0,v}(f'_v)W_v,\Phi_v,\eta_1,\eta_0,0,0)\ov{Z^{FJ}(0,{W}_v,\eta_2)}}{(W_v,{W}_v)_{\pi_0,v}}.
    \end{equation}
 and $I^{\ul{\eta},\natural}_{\pi_{0,v}}$ is the normalized version defined using the normalized functionals. When $v$ is inert in $E$, $f'_v\otimes\Phi_v =\bfun_{\G'(\calo_{F_v})}\otimes \bfun_{\calo_{F_v,n}}$, and $\pi_{0,v}$ is unramified, then $I^{\ul{\eta},\natural}_{\pi_{0,v}}(f'_v\otimes\Phi_v)=1$.
\end{Prop}
\begin{proof}
  This follows directly from the factorizations \eqref{Prop: nice inner product part 1}, \eqref{eqn:factorize FJ}, and \eqref{eqn: BF Lfunction}. 
\end{proof}
\subsubsection{Non-vanishing of local relative characters} We pause to establish a relationship between the local relative character above and the local relative character determined by the local Friedberg--Jacquet period. 

\begin{Prop}\label{Prop: elliptic linear rel character}
   Assume that $F$ is a non-archimedean local field and set $\ul{\eta}=(1,1,1)$. If $\pi_{0}$ is $\rH'$-elliptic in the sense of Section \ref{Section: elliptic}, then there exists $f'\otimes \Phi\in C_c^\infty(\G'(F)\times F_{n})$ with $f'$ of regular elliptic support such that
   \[
    I^{\underline{\eta},\natural}_{\pi_{0}}(f'\otimes\Phi) \neq 0
   \]
\end{Prop}
\begin{proof}
Recall the formula for the relative character
    \[
    I^{\underline{\eta},\natural}_{\pi_{0}}(f'\otimes\Phi) = \sum_{W}\frac{Z^{BF,\natural}(\pi_{0}(f')W,\Phi,1,1,0,0)\ov{Z^{FJ,\natural}(0,W,1)}}{[W,W]^\natural_{\pi_{0}}}.
    \]
We now use the local functional equation for the Bump--Friedberg integral \cite[Proposition 4.16]{MatringeBF}\footnote{Note the two-variable functional equation is stated in \cite[Corollary 3.2]{MatringeBF2}, but with a typo. On the left-hand side of (2) in \emph{ibid.}, the entry $1-s$ should read $1/2-s$, which may be readily verified by comparing to \cite[Proposition 4.16]{MatringeBF}.} to see that there exists a non-zero number $\epsilon(\pi_{0},\psi,0,0)\in \cc^\times$ such that
\[
 \epsilon(\pi_{0},\psi,0,0)Z^{BF,\natural}(\pi_{0}(f')W,\Phi,1,1,0,0) =Z^{BF,\natural}(\pi^\vee_{0}({f'}^\theta)\hat{W},\widehat{\Phi},1,1,\frac{1}{2},0),
\]
where ${f'}^\theta(g) = f'(g^\theta)$, $\hat{W}(g) = W(g^\theta)$, and where $\widehat{\Phi}$ denotes the Fourier transform with respect to the additive character $\psi$.   Since $\eta_{0}=1$, we may use Proposition \ref{Prop: BF} (\ref{Lem: BF local identity on functionals}) to find that
   \begin{align}\label{eqn: split rel char identity}
        I^{\ul{\eta},\natural}_{\pi_{0}}(f'\otimes\Phi) = \epsilon(\pi_{0},\psi,0,0)\Phi(0) \sum_{W}\frac{Z^{FJ,\natural}(0,\pi^\vee_{0}(\widehat{f'})\hat{W},1)\ov{Z^{FJ,\natural}(0,\hat{W},1)}}{[\hat{W},\hat{W}]^\natural_{\pi^\vee_{0}}}.
   \end{align}
   The right-hand side is a non-zero $(\rH',\rH')$-invariant relative character for the $\rH'$-elliptic representation $\pi_{0}^\vee$. Thus by definition of ellipticity, there exists a choice of $\Phi\in C^\infty_c(F_n)$ and $f'\in C_c^\infty(\G'(F))$ with $\rH'$-elliptic support for which \eqref{eqn: split rel char identity} is non-zero, and the proposition follows.
\end{proof}

\subsubsection{The simple trace formula} We may now state our simple form of the relative trace formula. Let $\wt{f}$ be a nice test function. For $\Re(s_0)>1$, Assumption \ref{supercuspidal constraint nice} and the spectral expansion expansion of $K_{f'}(x,y)$ imply the expansion
\begin{align*}
   I^{\underline{\eta}}(\wt{f},\ul{s})=  \sum_{\pi_0} I^{\underline\eta}_{\pi_0}(\wt{f},\ul{s}), 
\end{align*}
where the sum runs over irreducible cuspidal automorphic representations $\pi_0$ of $\G(\A_{F})$ with trivial central character satisfying $\pi_{0,v_1}\in \Omega$ is supercuspidal and (applying Proposition \ref{Prop: elliptic linear rel character}) $\pi_{0,v_2}$ is $\rH'$-elliptic. Specializing first to $s_2=0$, Lemma \ref{cor pole} implies that only those representations $\pi_0$ such that  $L(s,\pi_0,\wedge^2)$ has a pole at $s=1$ and $L(\frac{1}{2},\pi_0\otimes \eta_2)\neq 0$ appear. Lemma/Definition \ref{LemDef: rel character} now implies that each summand is holomorphic at $s_0=0$, so we set
\begin{equation}\label{eqn: spectral expansion linear}
       I^{\underline{\eta}}(\wt{f},s): = I^{\underline{\eta}}(\wt{f},(0,s,0))=\sum_{\pi_0} I^{\underline\eta}_{\pi_0}(\wt{f},s),
\end{equation}
which is holomorphic in $s\in \cc$.

On the geometric side, we begin with the expansion from Lemma \ref{Lem: converge and expand},
\begin{equation*}
 I^{\underline{\eta}}_\omega(\wt{f},\ul{s})=\sum_{(\gamma,w)} \Orb^{\underline{\eta}}_{\ul{s}}(\wt{f},(\gamma,w)). 
\end{equation*}
When $\wt{f}=\bigotimes_v\wt{f}_v$, then it is clear from the definition of the transfer factors in Definition \ref{Def: transfer factor on X} that when $(y,w)\in \X(F)\times F_{n}$ is strongly regular with $y=s_{\X}(\ga)$, then
\[
\prod_v   \Omega_v(\ga,w)=\prod_v   \omega_v(y,w)\eta_{2,v}(\ga)=1;
\]
For $\Re(s_0)>1$, we may thus re-normalize \eqref{eqn: adelic orbital} as
\[
\Orb^{\underline{\eta}}_{\ul{s}}(\wt{f},(\ga,w)) =L(s_0,T_{y},\eta_0)\prod_v \Orb^{\underline{\eta},\natural}_{\ul{s}}(s_{\X,!}^{\eta_2}(\wt{f}_v),(y_v,w_v))
\]
where $s_{\X,!}^{\eta_2}(\wt{f}_v)$ is the obvious extension of the pushforward \eqref{eqn: eta projection} and $\Orb^{\underline{\eta},\natural}_{\ul{s}}(s_{\X,!}^{\eta_2}(\wt{f}_v),(y_v,w_v))$ 
is the normalized orbital integral defined in Lemma/Definition \ref{LemDef: linear OI}. 

\begin{Prop}\label{Prop: simple RTF linear}
    Assume that $\wt{f}=\bigotimes_v\wt{f}_v$ is an \emph{elliptic nice test function} with respect to the Bernstein blocks $\Omega$ $v_1$. Then the right-hand side of \eqref{eqn: linear geo expansion 2} is holomorphic at $s_0=0$. We obtain the simple trace formula for $I^{\underline{\eta}}(\wt{f},s)$
\begin{align}\label{eqn: linear geo expansion 2}
 \sum_{y} L(0,T_{y},\eta_0)\prod_v \Orb^{\rH',\underline{\eta}}_{s}(s_{\X,!}^{\eta_2}(\wt{f}_v),y_v)=\sum_{\pi_0} I^{\underline\eta}_{\pi_0}(\wt{f},s),
\end{align}
where the first sum runs over regular elliptic semi-simple orbits $y\in \rH'(F)\backslash \X^{rss}(F)$ and the second sum runs over irreducible cuspidal automorphic representations $\pi_0$ of $\G'(\A_{F})$ satisfying 
\begin{enumerate}
    \item $\pi_{0,v_1}\in \Omega$ is supercuspidal, 
    \item$\pi_{0,v_2}$ is $\rH'$-elliptic,
    \item $L(s,\pi_0,\wedge^2)$ has a pole at $s=1$ and $L(\frac{1}{2},\pi_0\otimes \eta_2)\neq 0$.
\end{enumerate}
\end{Prop}

\begin{proof}
By Lemma \ref{Lem: converge and expand}, the right-hand side of \eqref{eqn: linear geo expansion 2} is a finite sum, so we consider the convergence of each term. We therefore fix a strongly regular pair $(y,w)$. For all $(s_0,s_1)\in \cc$, each orbital integral is convergent by Lemma/Definition \ref{LemDef: linear OI} and all but finitely many factors in the above product equal $1$ by Lemma \ref{Lem: almost everywhere 1}. In particular, for each $(\ga,w)$ the product
\[
\prod_v {\Orb^{\underline{\eta},\natural}_{\ul{s}}(\wt{f}_v,(\ga_v,w_v))},
\]
is convergent for all $(s_0,s_1)\in \cc$. This implies that for $\Re(s_0)>1$, we have the expansion.
\[
I^{\underline{\eta}}(\wt{f},(s_0,s_1,0)) = \sum_{(y,w)} L(s_0,\T_{y},\eta_0))\prod_v \Orb^{\underline{\eta},\natural}_{\ul{s}}(s_{\X,!}^{\eta_2}(\wt{f}_v),(y_v,w_v))
\]

Since $\wt{f}$ is assumed to have \emph{elliptic support} at a split place $v_2$, the only regular semi-simple $y\in \X(F)$ to contribute satisfy that $\T_y\simeq\Res_{L/F}(\Gm)$ for some degree $n$ field extension such that $E\not\subset L$. Indeed, since $v_2$ splits in $E$ if $E\subset L$, then $L_{v_2}$ would not be simple. It follows that $\eta_0\circ\Nm_{L/F}\neq1$ so that $L(s,\T_y,\eta_0)$ is holomorphic for all such $y$. Thus, we may evaluate each term on the right-hand side of the above sum. Lemma/Definition \ref{LemDef: linear OI} implies the resulting sum is independent of $w$, so we obtain
\begin{equation*}
I^{\underline{\eta}}(\wt{f},s):=I^{\underline{\eta}}(\wt{f},(0,s,0)) = \sum_{y} L(0,\T_{y},\eta_0)\prod_v \Orb^{\rH',\underline{\eta}}_{s}(s_{\X,!}^{\eta_2}(\wt{f}_v),y_v)
\end{equation*}
where the sum runs over regular elliptic semi-simple orbits 
\[
\rH'(F)\backslash \X^{rss}(F) {=} 
\rH'(F)\backslash (\X(F)\times F_{n})^{sr},
\]
where the equality of cosets follows from \cite[Lemma 3.7]{LXZfund}. Combining this with \eqref{eqn: spectral expansion linear} gives the result.
\end{proof}


\subsection{The unitary side} We now consider the relative trace formula for a unitary symmetric pair $(\G,\rH)=(\G_\bullet,\rH_\bullet)$ where $\bullet\in \{si,ii\}$.  Suppose that $\rH_1,\rH_2\subset \G$ are two pure inner twists of $\rH$ corresponding to classes in $\ker[H^1(F,\rH)\to H^1(F,\G)]$. For $f\in C^\infty_c(\G(\A_{F}))$, the relative trace formula we consider are of the form
\begin{equation}\label{eqn: unitary RTF}
J_{\rH_1,\rH_2}(f):= \int_{[\rH_1]}\int_{[\rH_2]}K_f(h_1,h_2)dh_1dh_2.
\end{equation}

As above, this will not always converge and we introduce the appropriate notion of \emph{nice test functions}. 

\begin{Def}\label{elliptic test functions}  We say ${f}=\bigotimes_v{f}_v\in C_c^\infty(\G(\A_{F}))$ is an \emph{elliptic nice test function} if 
\begin{enumerate}
\item There exists a non-Archimedean place $v_1$ of $F$ and a finite union $\Omega$ of cuspidal Bernstein components of $\G(F_{v_1})$ such that the function $f_{v_1}\in C^\infty_c(\G(F_{v_1}))_{\Omega}$.
\item There exists a non-archimedean {split} place $v_2$ of $F$ such that the function $f_{v_2}\in C^\infty_c(\G(F_{v_2}))$ is supported on the \emph{regular elliptic locus}; that is,
\[
\supp(f_{v_2})\subset \{g\in \G(F_{v_2}): (\rH_1\times \rH_2)_g \simeq L_g^\times\},
\]
where $L_g/F_{v_1}$ is a degree $n$ field extension.
\end{enumerate}
\end{Def}

Recall that a torus $T$ over $F$ is \emph{$E$-simple} if $T_E\simeq\Res_{L/E}\Gm$ for some finite field extension $L/E$.

With these constraints we have the following simple trace formula statement.
\begin{Prop}\label{Prop: Simple RTF unitary}
     Suppose that $f=\bigotimes_vf_v$ is a nice test function. Then
     \begin{enumerate}
         \item if $\ga\in \G(F)$ such that $f(h_1^{-1}\ga h_2)\neq 0$ for some $(h_1,h_2)\in \rH_1(\A_{F})\times \rH_2(\A_{F})$, then the stabilizer
         \[
         \rH_{\ga}=\{(h_1,h_2)\in \rH_1\times \rH_2: h_1^{-1}\ga h_2 = \ga\}
         \]
         is an elliptic $F$-simple torus of rank $n$. In particular, $\vol([\rH_\ga])<\infty$;
         \item we have an equality of absolutely convergent sums
         \begin{equation}\label{eqn: unitary simple rtf}
               \sum_{\pi}J_{\pi}(f)= J_{\rH_1,\rH_2}(f) = \sum_{[\ga]}\vol([\rH_\ga])\Orb^{\rH_1\times \rH_2}(f,\ga),
         \end{equation}
         where the left-hand side ranges over cuspidal automorphic representations $\pi$ of $\G(\A_{F})$ which are {supercuspidal at $v_1$ and $(\rH'_1,\rH'_2)$-elliptic at $v_2$}, and where
         \[
         J_{\pi}(f):=\sum_{\varphi\in ONB(\pi)}\left(\int_{[\rH_1]}\pi(f)\varphi(h)dh\right)\overline{\left(\int_{[\rH_2]}\varphi(h)dh\right)}
         \]
         is the usual relative trace, and where the right-hand side ranges over regular elliptic orbits
         \[
         \rH_1(F)\bs \G(F)^{re}/\rH_2(F)
         \]and for each such orbit
         \[
         \Orb^{\rH_1\times \rH_2}(f,\ga) = \int_{\rH_{\ga}(\A_{F})\bs\rH_1(\A_{F})\times \rH_2(\A_{F})}f(h_1^{-1}\ga h_2)dh
         \]
         is the (adelic) orbital integral.
     \end{enumerate}
\end{Prop}
\begin{proof}Since $f_{v_2}$ is assumed to have regular elliptic support, we may conclude that if $f(h_1^{-1}\ga h_2)\neq 0$, then 
\begin{enumerate}
    \item $\ga\in \G(F)$ is (relatively) regular semi-simple, and 
    \item $\rH_{\ga}\subset \rH_1\times \rH_2$ is a \emph{simple} elliptic torus of rank $n$.
\end{enumerate}
The rest of the claims follows from a special case of the simple trace formula established in \cite[Theorem 18.2.5]{GetzHahnbook}.
\end{proof}
In order to compare with the trace formula in Proposition \ref{Prop: simple RTF linear}, the geometric expansion needs refinement.\quash{As a first point, it is easy to see that $\car_{lin}$ is surjective on $F$-points. For the unitary symmetric variety, we need to incorporate pure inner forms to obtain all regular semi-simple orbits in some cases. Beyond this, we need to give a pre-stabilization of the geometric sides as we match \emph{stable} orbital integrals on unitary side to the linear side.} Since the two cases behave so differently, we handle them separately.
\subsubsection{Some recollections on Galois cohomology}
 Due to our use of elliptic nice functions in the trace formula, we do not need to consider the full pre-stabilization of the elliptic locus (cf. \cite[Section 10]{LeslieEndoscopy}). Nevertheless, we need to recall some basic notations from abelianized Galois cohomology from \cite{LabesseBook}. Assume that $\I\subset \rH$ is an inclusion of connected reductive groups.

  For a place $v$ of $F$, we let $F_v$ denote the localization. We set
  \[
  \mathfrak{C}(\I,\rH;F_v):=\ker[H^1_{ab}(F_v,\I)\to H^1_{ab}(F_v,\rH)];
  \]
  We also have the global variant
  \begin{equation*}
       \mathfrak{C}(\I,\rH;\A_{F}/F):=\coker[H^0_{ab}(\A_{F},\rH)\to H^0_{ab}(\A_{F}/F,\I\bs\rH)],  
  \end{equation*}
 equipped with a natural map 
  \begin{equation}\label{eqn: adelic map on C-groups}
    \mathfrak{C}(\I,\rH;\A_{F}):=\prod_v\mathfrak{C}(\I,\rH;F_v)\to   \mathfrak{C}(\I,\rH;\A_{F}/F).  
  \end{equation}
This induces a localization map $ \mathfrak{C}(\I,\rH;\A_{F}/F)^D\to  \prod_v\mathfrak{C}(\I,\rH;F_v)^D$ on Pontryagin dual groups.
  
The group $\mathfrak{C}(\I,\rH;\A_{F}/F)$ sits in an exact sequence \cite[Proposition 1.8.4]{LabesseBook}
\begin{equation}\label{eqn: exact sequence Galois cohom}
  \ker^1_{ab}(F,\rH)\to    \mathfrak{C}(\I,\rH;\A_{F}/F) \lra H^1_{ab}(\A_{F}/F, \I)\lra H^1_{ab}(\A_{F}/F, \rH),  
\end{equation}
where $\ker^1_{ab}(F,\rH):=\ker[H^1_{ab}(F,\rH)\to H^1_{ab}(\A_{F},\rH)]$. Note that since $\I$ and $\rH$ are connected, this implies $\mathfrak{C}(\I,\rH;\A_{F}/F)$ is a finite group. Moreover, $\ker^1_{ab}(F,\rH)=0$ if and only if $\rH$ satisfies the Hasse principle \cite[Corollaire 1.6.11]{LabesseBook}. We note that for any unitary symmetric pair $(\G,\rH)$ considered in this paper, it is well-known that $\rH$ satisfies the Hasse principle so that $\ker^1_{ab}(F,\rH)=0$.
\subsubsection{The split-inert case} As we saw locally, we must work with several pure inner twists in this case. Fix $\tau_1,\tau_2\in \calv_n(E/F)$ and consider the symmetric pair
\begin{align*}
	 \G_{\tau_1,\tau_2} &= \U(V_{\tau_1}\oplus V_{\tau_2}),\\
\rH_{\tau_1,\tau_2} &=\U(V_{\tau_1})\times \U(V_{\tau_2}).
\end{align*}
We saw in \S \ref{Section: orbital si} that
\begin{equation*}
    \Q_{\tau_1,\tau_2}(F) = \bigsqcup_{(\tau_1',\tau_2')}\G_{\tau_1,\tau_2}(F)/\rH_{\tau_1',\tau_2'}(F),
\end{equation*}
where $(\tau_1',\tau_2')$ range over pairs such that there is an isometry $V_{\tau_1}\oplus V_{\tau_2}\simeq V_{\tau_1'}\oplus V_{\tau_2'}$. Letting $\pi_{\tau_1',\tau_2'}:\G_{\tau_1,\tau_2}\to \Q_{\tau_1,\tau_2}$ be the quotient, this implies that the map
\[
\bigoplus_{(\tau'_1,\tau'_2)}C_c^\infty(\G_{\tau_1,\tau_2}(\A_{F}))\lra C^\infty_c(\Q_{\tau_1,\tau_2}(\A_{F})),
\]
given by $\phi = \sum_{(\tau_1',\tau_2')}\pi_{\tau_1',\tau_2',!}(f^{\tau_1',\tau_2'})$, is surjective. By definition, for any collection $\underline{f}=(f^{\tau_1',\tau_2'})$ satisfies only finitely many terms are non-zero. We say $\phi$ is an elliptic nice test function if each $f^{\tau_1',\tau_2'}\in C_c^\infty(\G_{\tau_1,\tau_2}(\A_{F}))$ is.

 For such an elliptic nice $\phi$, we set
\begin{equation}
    J_{\Q_{\tau_1,\tau_2}}(\phi) :=\sum_{(\tau_1',\tau_2')}J_{\rH_{\tau_1,\tau_2},\rH_{\tau_1',\tau_2'}}(f^{\tau_1',\tau_2'}).
\end{equation}
We work with this sum as it is formally easier to consider the geometric side in terms of the variety $\Q$; it is straightforward to work with a single pair $(\rH_{\tau_1,\tau_2},\rH_{\tau_1',\tau_2'})$ by taking $f^{\tau_1'',\tau_2''}=0$ for $(\tau_1'',\tau_2'')\neq(\tau_1',\tau_2')$. 
We may now give the geometric expansions in terms of the symmetric variety, as well as the stabilized geometric expansion.
\begin{Prop}\label{Prop: prestab si} Let $\phi \in C^\infty_c(\Q_{\tau_1,\tau_2}(\A_{F}))\quash{= \sum_{(\tau_1',\tau_2')}\pi_{\tau_1',\tau_2',!}(f^{\tau_1',\tau_2'})}$ be an elliptic nice test function.
   \begin{enumerate}
       \item  We have the following geometric expansion
    \begin{equation}
         J_{\Q_{\tau_1,\tau_2}}(\phi)=\sum_{[x]\in \rH_{\tau_1,\tau_2}(F)\bs \Q_{\tau_1,\tau_2}(F)^{re}}\vol([\rH_x])\Orb^{\rH_{\tau_1,\tau_2}}(\phi,x),
    \end{equation}
    where the sum ranges over regular elliptic $\rH_{\tau_1,\tau_2}(F)$-orbits satisfying that the stabilizer $\rH_x\subset \rH_{\tau_1,\tau_2}$ is a simple elliptic torus of rank $n$, and 
    \[
    \Orb^{\rH_{\tau_1,\tau_2}}(\phi,x)=\int_{\rH_x(\A_{F})\bs \rH_{\tau_1,\tau_2}(\A_{F})}\phi(h^{-1}\cdot x)dh.
    \]
\item We have the (pre-)stabilization
 \begin{equation}\label{eqn: stable RTF si}
     J_{\Q_{\tau_1,\tau_2}}(\phi)= 2\sum_{a\in[\Q_{si}^{re}\sslash\rH](F)}{2^{-|S_a|}}L(0,\rH^{op}_a,\eta) \SO^{\rH_{\tau_1,\tau_2}}(\phi,a),
 \end{equation}
 where $\Q_{si}^{re}$ denotes the regular elliptic locus and 
 \[
 \SO^{\rH_{\tau_1,\tau_2}}(\phi,a)=\int_{(\rH_{x_a}\backslash\rH_{\tau_1,\tau_2})(\A_{F})}\phi(h^{-1}\cdot x_a)dh = \prod_v \SO^{\rH_{\tau_1,\tau_2}}(\phi_v,a_v)
 \]
 is the (adelic) stable orbital integral. Finally, $|S_a| = |S_{x}|$ is defined in \S \ref{Section: Lvalue measure} for any $x$ mapping to $a$.
    \end{enumerate}
 \end{Prop}
    
 \begin{proof}
The first claim follows from combining the previous lemma with Proposition \ref{Prop: Simple RTF unitary} and properties of orbital integrals along contraction maps (cf. \cite[Section 2.2]{LXZfund}).   Collecting terms in the various fibers the map $\car_{si}$ gives the expression
     \[
    J_{\Q_{\tau_1,\tau_2}}(\phi)= \sum_{a\in[\Q_{si}^{re}\sslash\rH](F)}\vol([\rH_a])\sum_{\stackrel{x\in \Q_{\tau_1,\tau_2}(F)/\sim}{\car_{si}(x)=a}}\Orb^{\rH_{\tau_1,\tau_2}}(\phi,x),
     \]
     where $\rH_a\simeq\rH_{x_a}$ for any $x_a$ over the class $a$.  Our assumption that $\phi$ is an elliptic nice function implies that the only $a\in [\Q_{si}^{re}\sslash\rH](F)$ which contribute satisfy that $\rH_a$ is an $F$-simple elliptic torus of rank $n$, so that $|H^1(\A_{F}/F,\rH_a)| = 2$.
     
     The pre-stabilization of the inner sum may now be worked out via standard tools of Galois cohomology \cite[Chapter 1]{LabesseBook} and Fourier inversion on the finite group $\mathfrak{C}(\rH_a,\rH;\A_{F}/F)^D$ using the exact sequence \eqref{eqn: exact sequence Galois cohom} (see \cite[Proposition 4.2.1]{LabesseBook} or \cite[Proposition 7.2]{GetzWambach} for an example in the relative context), and we find
         \[
     \sum_{\stackrel{x\in \Q_{\tau_1,\tau_2}(F)/\sim}{\car_{si}(x)=a}}\Orb^{\rH_{\tau_1,\tau_2}}(\phi,x)= \frac{\tau(\rH_{\tau_1,\tau_2})}{\tau(\rH_a)d_a}\sum_\ka\Orb^{\rH_{\tau_1,\tau_2},\ka}(\phi,a),
     \]
     where $\ka\in \mathfrak{C}(\rH_a,\rH_{\tau_1,\tau_2},\A_{F}/F)^D$, $\tau(\rH)$ denotes the Tamagawa number, and 
     \[
     d_a = \#\coker[H^1(\A_{F}/F, \rH_a)\to H^1_{ab}(\A_{F}/F, \rH)].
     \] The adelic $\ka$-orbital integral is given by
\begin{equation}\label{eqn: global kappa}
    \Orb^{\rH,\ka}(\phi,a)= \prod_v\Orb^{\rH,\ka_v}(\phi_v,a),
\end{equation}
where $(\ka_v)_v$ is the image of $\ka$ along the map dual to \eqref{eqn: adelic map on C-groups}. This product is well defined by \cite[Proposition 4.11]{Lesliedescent}, which implies that all but finitely many of the terms in the product reduce to $1$. More precisely, any global representative $x_a\in\Q_{\tau_1,\tau_2}(F)$ of the class $a$ is lies in $\Q_{\tau_1,\tau_2}(\calo_{F_v})$ and gives an absolutely semi-simple element  for all but finitely many localizations. The claim now follows from \emph{ibid.} for any such place where $\phi_v= \bfun_{\Q_{\tau_1,\tau_2}(\calo_{F_v})}$.

In this case, the map on Galois cohomology is injective and $$|H^1_{ab}(\A_{F}/F, \rH_{\tau_1,\tau_2})|=\tau(\rH_{\tau_1,\tau_2})=4.$$ Thus, $d_a=2$ and since $\tau(\rH_a) =2$, the right-hand side simplifies as
     \[
      \frac{\tau(\rH_{\tau_1,\tau_2})}{\tau(\rH_a)2}\sum_\ka\Orb^{\rH_{\tau_1,\tau_2},\ka}(\phi,a)= \SO^{\rH_{\tau_1,\tau_2}}(\phi,a).
     \]
     Appealing to the volume calculation in  Lemma \ref{Lem: vol of centralizer} now completes the proof of the proposition.
 \end{proof}

              Note that by \eqref{eqn: regular orbits on the base}, the stabilized formula \eqref{eqn: stable RTF si} only possesses non-trivial contributions for $a\in \car_{si}(\Q_{\tau_1,\tau_2}^{rss}(F))$, and we obtain the rest of the regular semi-simple contributions by varying $\tau_1\in \calv_n(E/F)$. If we fix a choice $\tau_2\in \calv_n(E/F)$ and consider a family $\underline{\phi}:=(\phi_{\tau_1})\in \bigoplus_{\tau_1}C_c^\infty(\Q_{\tau_1,\tau_2}(\A_{F}))$ of elliptic nice functions, we may more generally consider the sum
   \begin{align}\label{eqn: stack sum}
           J_{si,\tau_2}(\underline{\phi})&:=\sum_{\tau_1\in \calv_n(E/F) }J_{\Q_{\tau_1,\tau_2}}(\phi_{\tau_1}).
    \end{align} 
     \begin{Cor}\label{Cor: stacky stabilized} Set $[\Q_{si}^{re}\sslash\rH](F)_{\tau_1,\tau_2}:=\car_{si}(\Q_{\tau_1,\tau_2}^{rss}(F))$.
         For any family $\underline{\phi}:=(\phi_{\tau_1})\in \bigoplus_{\tau_1}C_c^\infty(\Q_{\tau_1,\tau_2}(\A_{F}))$ of elliptic nice functions, the sum \eqref{eqn: stack sum} admits the stabilization
         \[
          J_{si,\tau_2}(\underline{\phi})= 2\sum_{a\in[\Q_{si}^{re}\sslash\rH](F)}{2^{-|S_a|}}L(0,\rH^{op}_a,\eta) \SO^{\rH}(\ul{\phi},a),
         \]
         where $\SO^{\rH}(\ul{\phi},a):=\SO^{\rH_{\tau_1,\tau_2}}({\phi}_{\tau_1},a)$ if $a\in[\Q_{si}^{re}\sslash\rH](F)_{\tau_1,\tau_2}$.
     \end{Cor}

\subsubsection{The inert-inert case} We assume that $V_{2n} = L\oplus L^\ast$ is a (globally) split Hermitian space with polarization $L\oplus L^\ast$. Then we have
\[
\G= \U(V_{2n})\:\text{ and }\: \rH = \Res_{E/F}(\GL(L)).
\]
We set $J_{ii}:=J_{\rH,\rH}$. Since $H^1(F,\GL(L))=0$, the quotient map $\pi_{ii}: \G\to \Q_{ii} = \G/\rH$ is surjective on adelic points, so that $\pi_{ii,!}:C_c^\infty(\G(\A_{F}))\to C_c^\infty(\Q_{ii}(\A_{F}))$ is surjective. Setting $\phi=\pi_{ii,!}(f)$, we rewrite the geometric expansion of $J_{ii}(f)$ as
\begin{equation}
\sum_{[\ga]}\vol([\rH_\ga])\Orb(f,\ga) = \sum_{[x]\in \rH(F)\bs \Q_{ii}(F)}\vol([\rH_x])\Orb^{\rH}(\phi,x),
\end{equation}
where 
\[
\Orb^{\rH}(\phi,x)= \int_{\rH_x(\A_{F})\bs\rH(\A_{F})}\phi(h^{-1}\cdot x)dh.
\]

As in the local setting, for any $x\in \Q_{ii}^{rss}(F)$, we have an element $\varepsilon\in H^1(\A_{F}/F,\rH_x)^D$ given by 
\[
\varepsilon(\al) :=\eta(\det(b_\al))\eta(\det(b))^{-1},
\]
where we borrow notation from the local setting so that $x_\al=\begin{psmatrix}
    a_\al&b_\al\\c_\al&a_\al^\ast
\end{psmatrix}\in \calo_{st}(x)$. This character clearly induces the local characters $\varepsilon: H^1(F_v,\rH_x)\to \cc^\times$ from Lemma \ref{Lem: varep char} (cf. \cite[Lemma 4.9]{LXZfund}) along the localization map 
\[
\mathfrak{C}(\rH_x,\rH;\A_{F}/F)^D\lra  \prod_v\mathfrak{C}(\rH_x,\rH;F_v)^D.
\]

\begin{Prop}\label{Prop: prestab ii}
 For a factorizable $f=\bigotimes_vf_v\in C^\infty(\G(\A_{F}))$ be a nice function, and set $\phi=\pi_{ii,!}(f)=\bigotimes_v\phi_v$. 
 \[
 J_{ii}(f)   = \sum_{a\in[\Q_{ii}^{re}\sslash\rH](F)} {2^{-|S_a|}}L(0,\rH_a,\eta)\left(\SO^{\rH}(\phi,a)+  \Orb^{\rH,\varepsilon}(\phi,a)\right),
 \]
where 
   $ \Orb^{\rH,\ka}(\phi,a)= \prod_v\Orb^{\rH,\ka_v}(\phi_v,a).$
 \end{Prop}
 \begin{proof}
Arguing as in the previous case, we have\quash{   We have already seen that  
     \[
     J_{ii}(f) = \sum_{[x]\in \rH(F)\bs \Q_{ii}(F)}\vol([\rH_x])\Orb^{\rH}(\phi,x).
     \]
     Collecting terms in the sum in the fibers over $\car_{ii}(F)$ gives } the expression
     \[
    J_{ii}(f) = \sum_{a\in[\Q^{re}_{ii}\sslash\rH](F)}\vol([\rH_a])\sum_{\stackrel{x\in \Q_{ii}(F)/\sim}{\car_{ii}(x)=a}}\Orb^{\rH}(\phi,x).
     \]
     \quash{Our assumption that $f$ is an elliptic nice function implies that the only $a$ which contribute satisfy that $\rH_a$ is a simple elliptic maximal torus maximal torus in a unitary subgroup of $\GL(L)$, so that $H^1(\A_{F}/F,\rH_a) = 2$. As before, }The pre-stabilization of the inner sum gives
     \[
     \sum_{\stackrel{x\in \Q_{ii}(F)/\sim}{\car_{ii}(x)=a}}\Orb^{\rH}(\phi,x)= \frac{\tau(\rH)}{\tau(\rH_a)d_a}\sum_\ka\Orb^{\rH,\ka}(\phi,a),
     \]
     where we have $\ka\in \mathfrak{C}(\rH_a,\rH;\A_{F}/F)^D$.  It is easy to compute in this case that $|H^1_{ab}(\A_{F}/F, \rH)|=\tau(\rH)=1$, so that $d_a=1$.  Since $\rH_a$ is $F$-simple elliptic, we see that $\mathfrak{C}(\rH_x,\rH;\A_{F}/F)^D = \{1,\varepsilon\}$, so the sum simplifies to 
\[
\frac{1}{\tau(\rH_a)}\left(\SO^{\rH}(\phi,a)+  \Orb^{\rH,\varepsilon}(\phi,a)\right).
\]
Appealing to the volume calculation in  Lemma \ref{Lem: vol of centralizer} now completes the proof of the proposition as in the split-inert case.
 \end{proof}

  \section{Proof of the global results}\label{Section: global proofs}
In this section, we compare the RTFs to obtain our global results. Due to the lack of the smooth transfer statements in \cite[Conjectures 4.10 and 4.14]{LXZfund}, we impose the following constraints on our number fields. We assume that $E/F$ is a quadratic extension of number fields such that
\begin{enumerate}
    \item\label{field1} $E/F$ is everywhere unramified,
    \item\label{field2} $E/F$ splits over every finite place $v$ of $F$ such that $$p\leq \max\{e(v/p)+1,2\},$$ where  $v|p$ and $e(v/p)$ denote the ramification index of $v$ in $p$ (this affects only finitely many places).
\end{enumerate}
Note that the final condition is necessary for access the fundamental lemma statements at each inert place.
\begin{Rem}\label{rem:assump}
    We remark that we may drop the assumptions \eqref{field1} and \eqref{field2} by appealing to the weak transfer statements of \cite[Theorems 4.15 and 4.16]{LXZfund}. This however introduces problems regarding relative characters which we opt to avoid in the present paper. 
\end{Rem}

We again note that the unramified hypothesis implies that the number of non-split archimedean places is even.

\subsection{Weak base change}Let $\tau\in \calv_{2n}(E/F)$ and let $\G_{\tau}:=\U(V_{\tau})$. We recall the notion of weak automorphic base change.

\begin{Def}
     Let $\pi$ be an irreducible cuspidal representation of $\G_{\tau}(\A_{F})$. We say that an isobaric automorphic representation $\Pi$ of $\GL_n(\A_E)$ is a weak automorphic base change of $\pi$ if for all but finitely many places $v$ of $F$ split in $E$, the (split) local base change of $\pi_v$ is isomorphic to $\Pi_v$. If weak automorphic base change of $\pi$ exists, then it is unique up to isomorphism by the strong multiplicity one theorem \cite[Theorem A]{ramakrishnan2018theorem}; we denote it by $\mathrm{BC}_E(\pi)$.
\end{Def}
The following is a special case (sufficient for the present article) of \cite[Theorem 4.14]{BPLZZ}.
\begin{Prop} Let $\G_{\tau}$ be as above.   If $\pi$ is cuspidal automorphic satisfying that there exists a prime $v$ of $F$ split
in $E$ such that $\pi_v$ is supercuspidal, then the weak automorphic base change $\mathrm{BC}_E(\pi)$ of $\pi$ exists and is cuspidal, conjugate self-dual, and satisfies that $L(s, \Pi, \mathrm{As}^{(-1)^n})$ is regular at $s=1$.
\end{Prop}

\subsection{Period non-vanishing implies $L$-value non-vanishing}\label{Section: period implies Lvalue} 
\ Let $\pi_0=\bigotimes_v\pi_{0,v}$ be a cuspidal automorphic representation of $\G'(\A_{F})$.  We let $\mathrm{BC}_E(\pi_0)$ denote the base change of $\pi_0$ to $\G'(\A_F)$ \cite{ArthurClozel}.

\begin{Thm}\label{Thm: global si result period to Lvalue}
 Let $\pi_{0}$ be a cuspidal automorphic representation of $\G'(\A_{F})$ satisfying 
\begin{enumerate}
    \item there is a split finite place $v_1$ such that $\pi_{v_1}$ is supercuspidal,
    \item there is a split finite place $v_2$ such that $\pi_{v_2}$ is $\rH'$-elliptic,
\end{enumerate}
Assume that $\pi_0$ is of {\em symplectic type}. 
Let $\tau_1,\tau_2\in \calv_n(E/F)$ satisfy that both $V_{\tau_1}$ and  $V_{\tau_2}$ are positive definite at all non-split archimedean $v$ and are split at all non-split non-archimedean $v$. If there exists a cuspidal automorphic representation $\pi$ of $\G_{\tau_1,\tau_2}(\A_{F})$ satisfying 
\begin{enumerate}
    \item $\mathrm{BC}_E(\pi) =\mathrm{BC}_E(\pi_{0})$,
    \item $\pi$ is $\rH_{\tau_1,\tau_2}$-distinguished,
\end{enumerate}
then $L(1/2,{\rm BC}_E(\pi_{0}) )\neq 0$.
\end{Thm}
\begin{Rem}
   This result is not new. It was established under weaker local hypotheses in \cite[Section 6]{PollackWanZydor} when $\G_{\tau_1,\tau_2}$ is quasi-split via the residue method. More generally, it was established in \cite[Corollary 7.3]{ChenGan} via theta correspondence techniques when the representation $\pi$ on $\G_{\tau_1,\tau_2}$ is tempered cuspidal. 
\end{Rem}

We now consider the inert-inert case. 
\begin{Thm}\label{Thm: global ii result period to Lvalue} 
We further assume that every archimedean place $v$ of $F$ splits in $E$.
Let $\pi_{0}$ be a cuspidal automorphic representation of $\G'(\A_{F})$ satisfying 
\begin{enumerate}
    \item there is a split finite place $v_1$ such that $\pi_{v_1}$ is supercuspidal,
    \item there is a split finite place $v_2$ such that $\pi_{v_2}$ is $\rH'$-elliptic,
\end{enumerate}
Assume that $\pi_0$ is of {\em symplectic type}. If there exists a cuspidal automorphic representation $\pi$ of $\G_{ii}(\A_{F})$ satisfying
\begin{enumerate}
    \item $\mathrm{BC}_E(\pi) =\mathrm{BC}_E(\pi_{0})$,
    \item $\pi$ is $\rH_{ii}$-distinguished.
\end{enumerate} Then at least one of $L(1/2,\pi_{0})$ or $L(1/2,\pi_{0}\otimes \eta)$ is non-zero.
\end{Thm}
 Given the presence of endoscopic terms, it is not clear if the methods of \cite{PollackWanZydor} or \cite{ChenGan} directly apply in this case. We prove these theorems in \S \ref{Section: proof P to L}.

\subsection{$L$-value non-vanishing implies period non-vanishing}\label{Section: Lvalue implies period}
We now state a converse to the preceding result under stronger hypotheses. 

\begin{Thm}\label{Thm: global si result Lvalue to period}
Let $\pi_{0}$ be a cuspidal automorphic representation of $\G'(\A_{F})$ satisfying 
\begin{enumerate}
    \item there is a split finite place $v_1$ such that $\pi_{v_1}$ is supercuspidal,
    \item there is a split finite place $v_2$ such that $\pi_{v_2}$ is $\rH'$-elliptic,
    \item for each non-split place $v$, $\pi_v$ is unramified,
    \item for each non-split archimedean  place $v$, $\pi_v$ is the descent of the base change of a representation of the compact unitary group $\U_{2n}(\BR)$ that is distinguished by the compact  $\U_{n}(\BR)\times \U_{n}(\BR)$.
\end{enumerate} Assume that $\pi_0$ is of {\em symplectic type}. The following two assertions are equivalent:
\begin{enumerate}
\item $L(1/2,{\rm BC}_E(\pi_{0}) )\neq 0$, and 
\item
There exists $\tau_1,\tau_2\in \calv_n(E/F)$ and a cuspidal automorphic representation $\pi$ of $\G_{\tau_1,\tau_2}(\A_{F})$ satisfying
\begin{enumerate}
    \item $\mathrm{BC}_E(\pi) =\mathrm{BC}_E(\pi_{0})$,
    \item $\pi$ is $\rH_{\tau_1,\tau_2}$-distinguished.
\end{enumerate}
\end{enumerate}
\end{Thm}
We also have a version in the inert-inert case.
\begin{Thm}\label{Thm: global ii result Lvalue to period}
We further assume that every archimedean place $v$ of $F$ splits in $E$.
   Let $\pi_{0}$ be as above. Assume that $\pi_0$ is of {\em symplectic type}. The following two assertions are equivalent:
\begin{enumerate}
\item at least one of $L(1/2,\pi_{0})$ or $L(1/2,\pi_{0}\otimes \eta)$ is non-zero, and 
\item there exists a cuspidal automorphic representation $\pi$ of $\G_{ii}(\A_{F})$ satisfying
\begin{enumerate}
    \item   $\mathrm{BC}_E(\pi) =\mathrm{BC}_E(\pi_{0})$,
    \item $\pi$ is $\rH_{ii}$-distinguished.
\end{enumerate}
\end{enumerate}
\end{Thm}
We prove these theorems in \S \ref{Section: proof L to P}.

\subsection{Compatible transfers} We formulate two transfer statements (Lemmas \ref{Lem: S-compatible si} and \ref{Lem: S-compatible ii}) suitable for working with global trace formulas.

\subsubsection{$(\eta,\eta)$-transfers} Fix $\tau_2\in \calv_n(E/F)$. For $\tau_1\in\calv_n(E/F)$ varying, consider a collection $\underline{f}=(f^{\tau_1',\tau_2'})$ for functions on $\G_{\tau_1,\tau_2}(\A_{F})$ with $V_{\tau_1'}\oplus V_{\tau_2'}\simeq V_{\tau_1}\oplus V_{\tau_2}$. \quash{such that 
\begin{equation}\label{eqn: pushforward family}
   \phi_{\tau_1} = \sum_{(\tau_1',\tau_2')}\pi_{\tau_1',\tau_2',!}(f^{\tau_1',\tau_2'}).
\end{equation}
We denote this by $\pi_!(\ul{f}) = \ul{\phi}$.}

\begin{Def} Let $\wt{f}=\prod_v\wt{f}_v\in C_c^\infty(\G'(\A_{F})\times \A_{F,n})$. We say that $\wt{f}$ and a collection of test functions  $\underline{f}=(f^{\tau_1',\tau_2'})$
 form \textbf{$(\eta,\eta)$-transfers} if $\wt{f}_v$ and $\ul{f_v} = (f^{\tau_1',\tau_2'}_v)$ are $(\eta_v,1,\eta_v)$-transfers in the sense of Definition \ref{Def: transfer si}. 
\end{Def}

\begin{Def}\label{Def: S-compatible transfer si} Fix a finite set of places $S$ of places of $F$. We say that test functions $\wt{f}$ and $\ul{\phi}=\pi_!(\ul{f})$ form a \textbf{$S$-compatible $(\eta,\eta)$-transfer} if they are $(\eta,\eta)$-transfers satisfying the following conditions:
    \begin{enumerate}
    \item For any $v\notin S$, if $f^{\tau_1',\tau_2'}\neq0$ then the Hermitian spaces $V_{\tau_1'}$ and $V_{\tau_2'}$ are unramified at $v$ (this is automatic if $v$ splits in $F$). Moreover, if $v\notin S$ is inert, $\wt{f}_v$ and $\underline{f_v}$ are the test functions with $k=0$ in the fundamental lemma in Theorem \ref{Thm: fundamental lemma si}. 
    \item At every inert non-archimedean place $v\in S$, there exists $k> 0$ so that $\wt{f}_v = c_v(k)\bfun_{\G'_k(\calo_{F_v})}\otimes \bfun_{\calo_{F,n}}$ and $\ul{f_v}$ is given by $f^{\tau_1,\tau_2} = \bfun_{\G_k(\calo_{F_v})}$ if $\tau_1,\tau_2$ are unramified and $f_v^{\tau_1,\tau_2} = 0$ otherwise. Here, $c_v(k)\in \cc^\times$ is the constant so that $\wt{f}_v$ and $\underline{f_v}$ match under Theorem \ref{Thm: fundamental lemma si}.
     \item At every non-split archimedean place $v\in S$, the functions $f^{\tau_1,\tau_2}_v\equiv 0$ unless $V_{\tau_1}\oplus V_{\tau_2}$ is positive definite, and $f^{\tau_1,\tau_2}_v$ has regular semi-simple support in the definite case (Proposition \ref{prop:tr infty}).
\item There exists a non-archimedean split place $v_1\in S$ of $F$ and a finite union $\Omega$ of cuspidal Bernstein components of $\G'(F_{v_1})=\G_{\tau_1,\tau_2}(F_{v_1})$ such that $\wt{f}_{v_1}=f^{\tau_1,\tau_2}_{v_1}\otimes \Phi_{v_1}$ with $\Phi_{v_1}(0)= 1$ and 
\[
f^{\tau_1,\tau_2}_{v_1}\in C^\infty_c(\G'(F_{v_1}))_{\Omega}.
\]
\item There exists a non-archimedean {split} place $v_2\in S$ of $F$ such that $\wt{f}_{v_2}=f^{\tau_1,\tau_2}_{v_2}\otimes \Phi_{v_2}$ with $\Phi_{v_2}(0)= 1$ and 
\[
f^{\tau_1,\tau_2}_{v_2}\in C^\infty_c(\G'(F_{v_2}))
\]is supported on the \emph{regular elliptic locus}.
 
    \end{enumerate}
    When these conditions are satisfied, we say that each test function is \textbf{$((\eta,\eta),S)$-transferable}. 
\end{Def}

\begin{Lem}\label{Lem: S-compatible si}
    Suppose that $S'$ is a finite set of non-split places and let $v_1,v_2$ be two non-archimedean split places; set $S=S'\cup\{v_1,v_2\}$. Assume that $\ul{f}\in \oplus_{\tau_1}\oplus_{(\tau_1',\tau_2')}C^\infty_c(\Q_{\tau_1,\tau_2}(\A_{F}))$ is $((\eta,\eta),S)$-transferable. Then there exist $\wt{f}\in C^\infty(\G'(\A_{F})\times \A_{F,n})$ such that $\wt{f}$ and $\ul{f}$ form a $S$-compatible $(\eta,\eta)$-transfer.  The converse statement also holds for $((\eta,\eta),S)$-transferable $\wt{f}$. 
\end{Lem}
\begin{proof}
    The existence of $\wt{f}$ follows from Theorem \ref{Thm: fundamental lemma si}, which also gives the converse implication. Note that the two constraints at the split places $v_1,v_2\in S$ may be accomplished by the split transfer of \S \ref{Section: split transfer}. 
\end{proof}

\subsubsection{$(\eta,1)$-transfer} We now consider the inert-inert case, where we must also handle the endoscopic transfer.

Let $f\in C^\infty_c(\G_{ii}(\A_{F}))$. As in \S \ref{Section: endoscopic}, fix $\tau_2\in\calv_n(E/F)$ and consider a collection $\underline{f}=(f^{\tau_1',\tau_2'})$ for functions on $\G_{\tau_1,\tau_2}(\A_{F})$ with $V_{\tau_1'}\oplus V_{\tau_2'}\simeq V_{\tau_1}\oplus V_{\tau_2}$. 
\begin{Def}
    We say that $f=\bigotimes_vf_v$ and a collection $\ul{f}=(f^{\tau_1',\tau_2'})$ form \textbf{$\varepsilon$-transfers} if the local components of $f_v$ and $\ul{f_v} =(f^{\tau_1',\tau_2'}_v)$ are $\varepsilon$-transfers in the sense of Definition \ref{Def: varepsilon transfer}.
\end{Def} 

\begin{Lem}
    Assume that $f=\bigotimes_vf_v$ and a collection $\ul{f}$ form $\varepsilon$-transfers, and assume $\ul{f}$ satisfy the restrictions to be a member of an $S$-compatible $(\eta,\eta)$-transfer for some finite set of places $S$. Then there exists $\wt{f}_\varepsilon = \bigotimes_v \wt{f}_{\varepsilon,v}\in C^\infty_c(\G'(\A_F)\times \A_{F,n})$ such that for each $v$ and any matching regular semi-simple elements $\ga\in \G'(F_v)$  and $g\in \G_{ii}(F_v)$, we have
\begin{equation}\label{eqn: varepsilon to linear}
  \Orb^{\rH',(\eta,\eta)}_{0}(\widetilde{f}_{\varepsilon,v},\ga)= \De_{\varepsilon}(g)\Orb^{\rH\times \rH,\varepsilon}(f_v,x).
\end{equation}
\end{Lem}
\begin{proof} 
    By definition, for each place $v$ of $F$ we have 
    \[
     \De_{\varepsilon}(g) \Orb^{\rH\times \rH,\varepsilon}(f_v,g) = \SO^{\rH_{\tau_1,\tau_2}\times \rH_{\tau_1',\tau_2'}}(f^{\tau_1',\tau_2'}_{v},h),
     \]
     for all matching regular semi-simple $g\leftrightarrow_{\tau_1}h$. Combining this with Lemma \ref{Lem: S-compatible si}, we may find $\wt{f}_\varepsilon$ which satisfies the claim of the lemma.
\end{proof}
\begin{Def} Let $f=\bigotimes_vf_v\in C^\infty_c(\G_{ii}(\A_{F}))$ and $\wt{f},\:\wt{f}_\varepsilon\in C^\infty_c(\G'(\A_{F})\times \A_{F,n})$ factorizable.
    We say that $(f,\wt{f},\wt{f}_\varepsilon)$ form a \textbf{complete $(\eta,1)$-transfer} if for all $v$
    \begin{enumerate}
        \item $\wt{f}_v$ and $f_v$ are $(\eta_v,1)$-transfers in the sense of Definition \ref{Def: transfer ii}, and
        \item $\wt{f}_{\varepsilon,v}$ and $f_v$ satisfy \eqref{eqn: varepsilon to linear} for all matching orbits $\ga\leftrightarrow g$.
    \end{enumerate}
\end{Def}

\begin{Def}\label{Def: S-compatible transfer ii} Fix a finite set of places $S$ of $F$. We say that test functions $(f,\wt{f},\wt{f}_\varepsilon)$  form a \textbf{$S$-compatible complete transfer} if they are  complete $(\eta,1)$-transfers satisfying the following conditions:
    \begin{enumerate}
    \item For any inert $v\notin S$, then $\wt{f}_v$ and ${f_v}$ are the test functions in the $k=0$ case of Theorem \ref{Thm: fundamental lemma ii}, and $\wt{f}_{\varepsilon,v}$ is the test function appearing in Theorem \ref{Thm: fundamental lemma si}.
        \item At every inert place $v\in S$, there exists $k> 0$ so that $\wt{f}_v$, $\wt{f}_{\varepsilon,v}$, and ${f_v}$ match under Theorems \ref{Thm: fundamental lemma ii} and \ref{Thm: fundamental lemma varepsilon}.
\item There exists a non-archimedean split place $v_1\in S$ of $F$ and a finite union $\Omega$ of cuspidal Bernstein components of $\G'(F_{v_1})=\G_{ii}(F_{v_1})$ such that $\wt{f}_{v_1}=\wt{f}_{\varepsilon,v_1}=f_{v_1}\otimes \Phi$ where
\[
f_{v_1}\in C^\infty_c(\G'(F_{v_1}))_{\Omega}.
\]
and $\Phi_{v_1}(0) = 1$.
\item There exists a non-archimedean {split} place $v_2\in S$ of $F$ such that $\wt{f}_{v_2}=\wt{f}_{\varepsilon,v_2}=f_{v_2}\otimes \Phi_{v_2}$ with $\Phi_{v_1}(0)= 1$ and  
$f_{v_2}\in C^\infty_c(\G'(F_{v_2}))$ is supported on the \emph{regular elliptic locus}.
    \end{enumerate}
When these conditions are satisfied, we say that each test function is \textbf{$((\eta,1),S)$-transferable}.
\end{Def}

\begin{Lem}\label{Lem: S-compatible ii}
   Suppose that $S'$ is a finite set of non-split non-archimedean places and let $v_1,v_2$ be two split places with $v_1$ non-archimedean; set $S=S'\cup\{v_1,v_2\}$. Assume that $f\in C^\infty_c(\G_{ii}(\A_{F}))$ is $((\eta,1),S)$-transferable. Then there exists $\wt{f},\wt{f}_{\varepsilon}\in C^\infty(\G'(\A_{F})\times \A_{F,n})$ such that $(f,\wt{f},\wt{f}_\varepsilon)$ form an $S$-compatible complete transfer.

    The following converse statement holds: given $((\eta,1),S)$-transferable $\wt{f}\in C^\infty(\G'(\A_{F})\times \A_{F,n})$, there exists an $S$-compatible $(\eta,1)$-transfer $f\in C_c^\infty(\G_{ii}(\A_{F}))$ and $\wt{f}_\varepsilon\in C^\infty(\G'(\A_{F})\times \A_{F,n})$ such that  $(f,\wt{f},\wt{f}_\varepsilon)$ form an $S$-compatible complete transfer.
\end{Lem}
\begin{proof} Given $((\eta,1),S)$-transferable $f\in C^\infty_c(\G_{ii}(\A_{F}))$, the existence of $\wt{f}$ follows Theorem \ref{Thm: fundamental lemma ii} and \S \ref{Section: split transfer}, which also gives the converse implication. On the other hand, Theorem \ref{Thm: fundamental lemma varepsilon} implies the existence of a collection $\ul{f}$ (where $\tau_2$ is split) such that $f$ and $\underline{f}$ form $\varepsilon$-transfers, and $\underline{f}$ satisfy the restrictions to be a member of an $S$-compatible $(\eta,\eta)$-transfer. The existence of $\wt{f}_\varepsilon$ now follows from Lemma \ref{Lem: S-compatible si}. The proof of the converse statement is also clear.
\end{proof}

\subsection{Comparison of relative characters} We establish the comparison of global relative characters. We set $S_{\pi_0}$ to be the finite set of non-archimedean places of $F$ where $v$ is inert and $\pi_{0,v}$ is not unramified. For a finite set of places $S$, we say that $\pi$ is $S$-unramified if $S= S_{\pi_0}$. For simplicity, we will always assume that $S$ contains at least two split places $v_1,v_2\in S$.

\subsubsection{The split-inert case}
We begin with the following comparison of RTFs. 

\begin{Prop}\label{Prop: weak trace comparison si}
Fix $\tau_2\in \calv_n(E/F)$ and a finite set of places $S$ of $F$. Let $\wt{f}$ and $\ul{f}$ be {$S$-compatible $(\eta,\eta)$-transfers}. Then
\[
\sum_{\tau_1}\left(\sum_{\pi}\sum_{(\tau_1',\tau_2')}J^{\rH_{\tau_1,\tau_2},\rH_{\tau_1',\tau_2'}}_\pi(f^{\tau_1',\tau_2'})\right)=J_{si,\tau_2}(\pi_!(\underline{f})) = 2I^{(\eta,\eta)}(\wt{f},0)=2\sum_{\pi_{0}}I_{\pi_{0}}^{(\eta,\eta)}(\wt{f},0),
\]    
where $\pi_{0}$ runs over $S$-unramified cuspidal automorphic representations of $\G'(\A_{F})$ and where for each $\tau_1$, $\pi$ runs over cuspidal automorphic representations of $\G_{\tau_1,\tau_2}(\A_{F})$ which are $(\rH_{\tau_1,\tau_2},\rH_{\tau_1',\tau_2'})$-distinguished.
\end{Prop}
\begin{proof}
   This comparison follows directly from our transfer results (Lemma \ref{Lem: S-compatible si}), Proposition \ref{Prop: simple RTF linear}, and Proposition \ref{Prop: prestab si}. More directly, if we set $\ul{\phi}=\pi_!(\ul{f})$, we have by Corollary \ref{Cor: stacky stabilized}
         \[
          J_{si,\tau_2}(\underline{\phi})= 2\sum_{a\in[\Q_{si}^{re}\sslash\rH](F)}2^{-|S_a|}L(0,\rH^{op}_a,\eta) \SO^{\rH}(\ul{\phi},a),
         \]
         where $\SO^{\rH}(\ul{\phi},a):=\SO^{\rH_{\tau_1,\tau_2}}({\phi}_{\tau_1},a)$ if $a\in[\Q_{si}^{re}\sslash\rH](F)_{\tau_1,\tau_2}$. Note that $|S_a| = |S_{a,\infty}|$ by the assumption that $E/F$ is everywhere unramified. On the other hand, Proposition \ref{Prop: simple RTF linear} gives us
            \[
I^{\underline{\eta}}(\wt{f},0) = \sum_{y} L(0,\T_{y},\eta)\prod_v {\Orb^{\rH',\underline{\eta}}_{0}(s_{\X}^{\eta}(\wt{f}_v),y_v)}.
\]
By the assumption that $\wt{f}$ and $\ul{f}$ be {$S$-compatible $(\eta,\eta)$-transfers} and that $E/F$ is everywhere unramified (so that $|S_{1}^{ram}(y_v)| \neq 0$ only at non-split archimedean places), we have 
\[
2^{|S_{a,\infty}|}\prod_v {\Orb^{\rH',\underline{\eta}}_{0}(s_{\X}^{\eta}(\wt{f}_v),y_v)} = \SO^{\rH}(\ul{\phi},a),
\]
when $y\in\X(F)$ matches $a\in [\Q_{si}^{re}\sslash\rH](F)$. Finally, we note that in this case the two $L$-functions $L(s,\rH^{op}_a,\eta)$ and $L(s,\T_{y},\eta)$ agree as they are both the Tate $L$-function associated to the quadratic character associated with $E_y/F_y$, where $F_y\simeq F[T]/(\car_{y}(T))$ is a degree $n$ field extension of $F$ determined by $y$ and $E_y= EF_y$.
\end{proof}

Making use of the strong multiplicity one theorem \cite[Theorem A]{ramakrishnan2018theorem}, we may separate the terms in the preceding sum. 
For a cuspidal automorphic representation $\Pi$ of $\G'(\A_E)$, we denote by $\mathcal{B}(\Pi)$ the (finite) set of cuspidal automorphic representations $\pi_{0}$ of $\G'(\A_{F})$ such that $\Pi=\mathrm{BC}_E(\pi_0)$. For any pair $\tau_1,\tau_2\in\calv_n(E/F)$, we denote by $\mathcal{B}_{\tau_1,\tau_2}(\Pi)$ the set of cuspidal automorphic representations $\pi$ of $\G_{\tau_1,\tau_2}(\A_{F})$ such that $\Pi=\mathrm{BC}_E(\pi)$. We have the following comparison of relative characters.

\begin{Prop}\label{Prop: global rel char matching si}
 For almost all split places $v$, fix an irreducible unramified representation ${\pi_{v}^0}$. For a fixed split place $v_1$, we fix a supercuspidal representation $\pi_{v_1}$ of $\G'(F_{v_1})$. We also fix an auxiliary split place $v_2$. Then there exists at most one cuspidal automorphic representation $\Pi$ of $\G'(\A_{E})$ such that for $S\supset S_{\Pi}\cup\{v_1,v_2\}$, where 
 \[
 S_{\Pi}=\{v \text{ inert place of $F$}: \text{if }w|v,\:\Pi_w \text{ not unramified}\},
 \]
if $\wt{f}$ and $\ul{f}$ are {$S$-compatible nice transfers} relative to $\Omega= \{\pi_{v_1}\}$, then
\begin{equation*}
 \sum_{\tau_1}\left(\sum_{\pi\in \mathcal{B}_{\tau_1,\tau_2}(\Pi)}\sum_{(\tau_1',\tau_2')}J^{{\tau'_1,\tau'_2}}_\pi(f^{\tau_1',\tau_2'})\right)=2\sum_{\pi\in \mathcal{B}(\Pi)}I_{\pi_{0}}^{(\eta,\eta)}(\wt{f},0),
\end{equation*}
where the sum on the right runs over all cuspidal automorphic representations $\pi_0$ of $\G'(\A_{F})$ such that
\begin{enumerate}
\item\label{property1}$\pi_{0,v}\cong \pi_v^0$ for almost all split $v$, 
\item\label{property2} $\pi_{0,v_1}\simeq \pi_{v_1}$ is our fixed supercuspidal representation,
\item\label{property3} $\pi_{0,v_2}$ is $\rH'$-elliptic,
\item\label{property4} $BC_{E_v}(\pi_{0,v})=BC_{E_v}(\tau_v)$ with $\tau_v$ an irreducible representation of a compact unitary group when $v$ is non-split archimedean,
\end{enumerate}
and where $\Pi=\mathrm{BC}_E(\pi_0)$ is the base change of any $\pi$ appearing in the right-hand sum. 
\end{Prop}
\begin{Rem}
    Note that the existence of $\Pi$ depends only on whether the set of cuspidal automorphic representations $\pi$ of $\G'(\A_F)$ satisfying \eqref{property1} and \eqref{property2} is non-empty. When this is the case, the right-hand sum has two elements $\mathcal{B}'(\Pi) = \{\pi_0,\pi_0\otimes \eta\}$. 
\end{Rem}
\begin{proof} The proof is standard, with the argument mirroring that of \cite[Proposition 2.10]{ZhangFourier},\cite[Proposition 9.9]{LeslieUFJFL} and \cite[Proposition 11.10]{LXZfund}. We therefore omit the details.
\end{proof}

\begin{Cor}\label{Cor: si comparison of global rel char} If we choose $\ul{f}$ so that $f^{\tau_1',\tau_2'}=0$ except for a single pair $(\tau_1',\tau'_2)$, then we have
    \[
    \sum_{\pi\in \mathcal{B}_{\tau_1,\tau_2}(\Pi)}J^{{\tau'_1,\tau'_2}}_\pi(f^{\tau_1',\tau_2'})=2I_{\pi_{0}}^{(\eta,\eta)}(\wt{f},0)+2I_{\pi_{0}\otimes \eta}^{(\eta,\eta)}(\wt{f},0),
    \]
\end{Cor}

\subsubsection{The inert-inert case}
As with Proposition \ref{Prop: weak trace comparison si}, the following comparison of RTFs follows directly from our transfer results (Lemma \ref{Lem: S-compatible ii}), Proposition \ref{Prop: simple RTF linear}, the pre-stabilization in Proposition \ref{Prop: prestab ii}, and the identification of abelian $L$-functions.

\begin{Prop}\label{Prop: weak trace comparison ii}
Fix a finite set of non-archimedean places $S$ of $F$ and let the triple $(f,\wt{f},\wt{f}_\varepsilon)$ be an $S$-compatible complete transfer. Then
\[
J_{ii}({f}) = I^{({\eta,1})}(\wt{f},0)+ I^{({\eta,\eta})}(\wt{f}_{\varepsilon},0).
\]    
\end{Prop}

We denote by $\mathcal{B}_{ii}(\Pi)$ the set of cuspidal automorphic representations $\pi$ of $\G_{ii}(\A_{F})$ such that $\Pi=\mathrm{BC}_E(\pi)$. A second application of \cite[Theorem A]{ramakrishnan2018theorem} gives the following comparison of relative characters.

\begin{Prop}\label{Prop: global rel char identity ii}
 For almost all split places $v$, we fix an irreducible unramified representation ${\pi_{v}^0}$. For a fixed split place $v_1$, we fix a supercuspidal representation $\pi_{v_1}$ of $\G'(F_{v_1})$, and fix an auxiliary split place $v_2$. Then there exists at most one cuspidal automorphic representation $\Pi$ of $\G'(\A_{E})$ such that for $S\supset S_{\Pi}\cup\{v_1,v_2\}$, where 
 \[
 S_{\Pi}=\{v \text{ inert place of $F$}: \text{if }w|v,\:\Pi_w \text{ not unramified}\},
 \]
if the triple $(f,\wt{f},\wt{f}_\varepsilon)$ is an $S$-compatible complete transfer relative to $\Omega= \{\pi_{v_1}\}$, then
\begin{equation*}
\sum_{\pi\in \mathcal{B}_{ii}(\Pi)}J^{\rH_{ii},\rH_{ii}}_\pi(f)=I_{\pi_{0}}^{(\eta,1)}(\wt{f},0)+I_{\pi_{0}}^{{(\eta,\eta)}}(\wt{f}_\varepsilon,0)+I_{\pi_{0}\otimes\eta}^{(\eta,1)}(\wt{f},0)+I_{\pi_{0}\otimes \eta}^{{(\eta,\eta)}}(\wt{f}_\varepsilon,0),
\end{equation*}
where $\{\pi_0,\pi_0\otimes \eta\}$ are the cuspidal automorphic representations of $\G'(\A_{F})$ satisfying
\begin{enumerate}
\item$\pi_{0,v}\cong \pi_v^0$ for almost all split $v$, 
\item$\pi_{,0v_1}\simeq \pi_{v_1}$ is our fixed supercuspidal representation,
\item $\pi_{0,v_2}$ is $\rH'$-elliptic,
\end{enumerate}
and where $\Pi=\mathrm{BC}_E(\pi_0)$ is the base change of any $\pi$ appearing in the right-hand sum. 
\end{Prop}
\begin{proof}
    The proof is analogous to that of Proposition \ref{Prop: global rel char matching si}. We omit the details.
\end{proof}

\subsection{Proof of Theorems \ref{Thm: global si result period to Lvalue} and \ref{Thm: global ii result period to Lvalue}}\label{Section: proof P to L}
We will prove both statements at once.
We thus assume that $(\G,\rH)$ is a unitary symmetric pair satisfying the quasisplit assumption of Theorem \ref{Thm: global si result period to Lvalue}; this implies that $\G$ is unramified at every inert place so that we have access to the local integral models in \S \ref{Section: integral models}. Let $\pi=\bigotimes_v\pi_v$ be a cuspidal automorphic representation of $\G(\A_{F})$ which is supercuspidal at a split place $v_1$ and $\rH$-elliptic at another split place $v_2$. We assume that $\pi$ is $\rH$-distinguished. Set $S= S_{\pi}$.

We construct an $(\ul{\eta},S)$-transferable nice test function $f= \bigotimes_v f_v$ for which the left-hand side of the simple trace formula in Proposition \ref{Prop: Simple RTF unitary} is non-zero. By assumption, there exists $\varphi\in V_\pi$ such that $\mathcal{P}_{\rH}(\varphi)\neq0$. For each inert place $v$, smoothness of $\pi_v$ implies that there exists $k_v\in \zz_{\geq0}$ such that $\pi_{v}(b_{v})\varphi\neq 0$, where $b_{v}$ is the scalar multiple of $\bfun_{\G_{k_v}(\calo_{F_v})}$ (cf. \S \ref{Section: integral models}) which satisfies $b_{v}\ast b^\vee_{v} = b_{v}$; this is possible by rescaling by the appropriate volume factor. In particular, $f_v= b_{k_v}$ is a function of positive type. For each split place $v$, we set $f_v = b_v\ast b^\vee_v$ for some $b\in C_c^\infty(\G(F_v))$ such that $\pi_{v}(b_{v})\varphi\neq 0$. Setting $f_1 = \bigotimes_vf_v$, it follows that $f_1= b\ast b^\vee$ where $b=\bigotimes_vb_v$ is of positive type such that 
\[
J^{\rH,\rH}_\pi(f_1) = \sum_{\varphi} |\mathcal{P}_{\rH}(\pi(b)\varphi)|^2>0.
\]

We now apply multiplicity one of the local functionals at the split places $v_1$ and $v_2$ \cite[Theorem 1.1]{JacquetRallis}. The global period integral $\mathcal{P}_{\rH}$ defines an invariant linear functional on $\pi=\bigotimes_v\pi_v$ which by local uniqueness at $v_1$ and $v_2$ decomposes
        \[
        \mathcal{P}_{\rH}=\lambda^{v_1,v_2} \cdot\lambda_{v_1}\cdot \lambda_{v_2},
        \]
 where for $v\in\{v_1,v_2\}$, the local functional $\lambda_{v}$ is the unique (up to scalar) $\rH(F_v)$-invariant linear functional on $\pi_v$. Writing $f_1= f^{v_1,v_2}f_{1,v_1}f_{1,v_2}$ so that $f^{v_1,v_2}= b^{v_1,v_2}\ast (b^{v_1,v_2})^\vee$ if of positive type on $\G(\A_{F}^{v_1,v_2})$, the relative character becomes
\[
J^{\rH,\rH}_\pi(f_1)=J_{\pi_{v_1}}(f_{1,v_1}) J_{\pi_{v_2}}(f_{1,v_2})\left(\sum_{\varphi^{v_1,v_2}} |\lambda^{v_1,v_2}(\pi^{v_1,v_2}(b^{v_1,v_2})\varphi^{v_1,v_2})|^2\right),
\]
where $J_{\pi_v}$ is the local relative character associated with $\pi_v$ and $\lam_v$. Set $\Pi=\mathrm{BC}_E(\pi)$ to be the weak base change of $\pi$. Then for any other $\pi'$ with $\Pi= \mathrm{BC}_E(\pi')$, we have a similar expression $J^{\rH,\rH}_\pi(f_1)=\al_{v_1}(f_{1,v_1}) \al_{v_2}(f_{1,v_2})J_{\pi'}^{v_1,v_2}(f^{v_1,v_2})$, where
\[
J_{\pi'}^{v_1,v_2}(f^{v_1,v_2}):=\sum_{\varphi^{v_1,v_2}} |\lambda^{v_1,v_2}({\pi'}^{v_1,v_2}(b^{v_1,v_2})\varphi^{v_1,v_2})|^2\geq 0.
\]
We now augment (if necessary) the function $f_{v_1}$ at $v_1$ to lie in the Bernstein block of $\pi_{v_1}$, and augment the function $f_{v_2}$ at $v_2$ to be an elliptically-supported function (cf. Definition \ref{elliptic test functions}).  By our assumption of $\rH$-ellipticity at $v_2$, we may find functions that $\lambda_{v_1}(f_{v_1}) \lambda_{v_2}(f_{v_2})\neq0$. Then $f:= f^{v_1,v_2}f_{v_1}f_{v_2}$ is an $(S,\ul{\eta})$-transferable nice test function on which $J_{\pi}^{v_1,v_2}(f^{v_1,v_2})>0$ and
\[
J_{\pi}(f)=J_{\pi_{v_1}}(f_{v_1}) J_{\pi_{v_2}}(f_{v_2})J_{\pi}^{v_1,v_2}(f^{v_1,v_2})\neq0.
\]


Applying Lemma \ref{Lem: S-compatible si} and Lemma \ref{Lem: S-compatible ii} (depending on the value of $\ul{\eta}$), we obtain $\wt{f}\text{ (and }\wt{f}_\varepsilon)\in C_c^\infty(\G'(\A_{F})\times \A_{F,n})$ for which Corollary \ref{Cor: si comparison of global rel char} (respectively, Proposition \ref{Prop: global rel char identity ii}) holds. 
Therefore at least one of terms in the right-hand sum in  Corollary \ref{Cor: si comparison of global rel char} (respectively, Proposition \ref{Prop: global rel char identity ii}) is non-zero,  Corollary \ref{cor nonvanish} implies the non-vanishing of the appropriate $L$-values. \qed

\subsection{Proof of Theorems \ref{Thm: global si result Lvalue to period} and \ref{Thm: global ii result Lvalue to period}}\label{Section: proof L to P} Suppose now that $\pi_{0}$ is a cuspidal automorphic representation of $\G'(\A_{F})$ of symplectic type, satisfying 
\begin{enumerate}
    \item there is a non-archimedean split place $v_1$ such that $\pi_{0,v_1}$ is supercuspidal,
    \item there is a non-archimedean split place $v_2$ such that $\pi_{0,v_2}$ is $\rH'$-elliptic,
        \item for each inert non-archimedean place $v$, $\pi_{0,v}$ is unramified.
        \item for each non-split archimedean place $v$, $BC_{E_v}(\pi_{0,v})=BC_{E_v}(\tau_v)$ with $\tau_v$ an irreducible representation of a compact unitary group when $v$ is non-split archimedean,
\end{enumerate}
 In particular, we set $S= S_{\infty}\cup\{v_1,v_2\}$, where $S_{\infty}$ is the set of non-split archimedean places (which is assumed to be empty in the inert-inert case).  Since Theorems \ref{Thm: global si result period to Lvalue} and \ref{Thm: global ii result period to Lvalue} contain one direction of the equivalence, we only focus on the $(1)\implies (2)$ direction of both theorems.

\subsubsection{Proof of Theorem \ref{Thm: global si result Lvalue to period}} 
%
Assume that $L(1/2,BC_E(\pi_0))\neq0.$ We claim that there exists an elliptic nice test function $\wt{f}=\bigotimes_v\wt{f}_v$ satisfying the requirements of belonging to an $((\eta,\eta),S)$-compatible transfer (cf. Definition \ref{Def: S-compatible transfer si}) such that  $I_{\pi_{0}}^{(\eta,\eta)}(\wt{f}, 0)\neq0$. As we see below, the existence of such $\wt{f}$ suffices to prove the theorem. Given $\wt{f}$, Proposition \ref{Prop: factorize the linear rel char} tells us that
\[
I_{\pi_{0}}^{(\eta,\eta)}(\wt{f},0) = \mathcal{L}^{(\eta,1,\eta)}_{\pi_0}\prod_vI^{(\eta,\eta),\natural}_{\pi_{0,v}}(\wt{f}_v),
\]
so that the claim reduces to verifying the non-vanishing of the local relative characters for such a test function.

To see that this suffices to prove the theorem, given such a function, Proposition \ref{Prop: factorize the linear rel char} tells us that
\[
I_{\pi_{0}}^{(\eta,\eta)}(\wt{f},0) = \mathcal{L}^{(\eta,1,\eta)}_{\pi_0}\prod_vI^{(\eta,\eta),\natural}_{\pi_{0,v}}(\wt{f}_v),
\]
where $\mathcal{L}^{(\eta,1,\eta)}_{\pi_0}\neq0$ by the assumption that the assumption that $L(1/2,\mathrm{BC}_E(\pi_0))\neq0$. Now consider
   \[
I_{\pi_{0}}^{(\eta,\eta)}(\wt{f},0)+I_{\pi_{0}\otimes \eta}^{(\eta,\eta)}(\wt{f},0) = \mathcal{L}^{(\eta,1,\eta)}_{\pi_0}\left(\prod_vI^{(\eta,\eta),\natural}_{\pi_{0,v}}(\wt{f}_v)+\prod_vI^{(\eta,\eta),\natural}_{\pi_{0,v}\otimes\eta_{v}}(\wt{f}_v)\right).
    \]
For each inert non-archimedean place, Proposition \ref{Prop: BF}\eqref{unramified identity BF} and the discussion after Proposition \ref{Prop: FJ periods} implies that 
\[
I^{(\eta,\eta),\natural}_{\pi_{0,v}}(\wt{f}_v)=I^{(\eta,\eta),\natural}_{\pi_{0,v}\otimes\eta_{v}}(\wt{f}_v)=1;
\]
for each non-split archimedean place, Lemma \ref{lem: branc compact} implies that $\pi_{0,v}\simeq \pi_{0,v}\otimes \eta_v$, so we similarly obtain $I^{(\eta,\eta),\natural}_{\pi_{0,v}}(\wt{f}_v)=I^{(\eta,\eta),\natural}_{\pi_{0,v}\otimes\eta_{v}}(\wt{f}_v)$. Finally, since $\eta_v=1$ when $v$ is split the two local relative characters agree at such places and we obtain
\[
I_{\pi_{0}}^{(\eta,\eta)}(\wt{f},0)+I_{\pi_{0}\otimes \eta}^{(\eta,\eta)}(\wt{f},0) = \mathcal{L}^{(\eta,1,\eta)}_{\pi_0}\left(2\prod_vI^{(\eta,\eta),\natural}_{\pi_{0,v}}(\wt{f}_v)\right)\neq0.
\]
Since $\wt{f}$ is $((\eta,\eta),S)$-transferable, Lemma \ref{Lem: S-compatible si} gives the existence of an $((\eta,\eta),S)$-compatible transfer $\ul{f}$. The result now follows from Proposition \ref{Prop: global rel char matching si}.

We now construct $\wt{f} = \bigotimes_v\wt{f}_v$ as above. Note that away from the places $v_1,v_2$, the requirement for $\wt{f}$ to lie in an $((\eta,\eta),S)$-compatible transfer is that $\wt{f}_v$ be the test vector appearing in the $k=0$ case of Theorem \ref{Thm: fundamental lemma si} for each inert non-archimedean place $v$, and that $\wt{f}_v = f_v'\otimes \Phi_v$ with $f_v'$ supported on regular semi-simple orbits which transfer from a compact unitary group.


When $v$ is inert and non-archimedean, we have already seen that if $\wt{f}_v= \bfun_{\G'(\calo_{F_v})}\otimes\bfun_{\calo_{F_v,n}}$ is the test function occurring in Theorem \ref{Thm: fundamental lemma si}, then $I^{(\eta,\eta),\natural}_{\pi_{0,v}}(\wt{f}_v)\neq0.$
\begin{Rem}
    The reason we are unable to consider deeper ramification in this direction is our lack of information on the non-vanishing of the local relative character $I_{\pi_{0}}^{(\eta,\eta)}(\bfun_{\G'_k(\calo_{F_v})}\otimes\bfun_{\calo_{F_v,n}},0)$ when $k>0$.
\end{Rem}
\noindent
When $v$ is a non-split archimedean place, the existence of $\wt{f}_v$ is given by Theorem \ref{thm:arch nonzero lin char}. For $v\neq v_1,v_2$ split, any choice of $\wt{f}_v$ such that $I^{(\eta,\eta),\natural}_{\pi_{0,v}}(\wt{f}_v)\neq0$ is sufficient as soon as we prove that such a function exist. To see this, note that since $\pi_0$ is assumed symplectic with that $L(1/2,BC_E(\pi_0))\neq0$, so that $Z^{FJ}(0,-,\eta)$ is non-zero on $\pi_0$ by Proposition \ref{Prop: FJ periods}. By \eqref{eqn:factorize FJ}, we see that $\pi_{0,v}$ has a non-zero $\rH'(F_v)$-functional $\lam$ so that the local relative character
\[
 J(f'_v)=\sum_{\phi\in OB(\pi_{0,v})}\lam(\pi_{0,v}(h)\phi)\overline{\lam(\phi)}
\]
is non-zero for some $f'_v\in C^\infty(\G'(F_v))$. By \eqref{eqn: split rel char identity} in the proof of Proposition \ref{Prop: elliptic linear rel character}, for appropriate $\Phi_v\in C_c^\infty(F_{v,n})$, there exists $0\neq c\in \cc$ such that
\[
I^{(\eta,\eta),\natural}_{\pi_{0,v}}({f}'_v\otimes \Phi_v)=cJ(f'_v),
\]
proving the claim.

Finally, we consider the places $v_1$ and $v_2$. Since $I^{(\eta,\eta)}_{\pi_{0,v_1}}$ is a non-zero distribution, it is clear from the Bernstein decomposition that we may assume $\wt{f}_{v_1}= f_{v_1}'\otimes \Phi_{v_1}$ where  $f_{v_1}'\in C_c^\infty(\G'(F_{v_1}))_{\Omega_1}$, where $\Omega_1$ is the supercuspidal Bernstein block associated to $\pi_{0,v_1}$. Finally, that $f_{v_2}'\in C_c^\infty(\G'(F_{v_2}))$ may be chosen to have elliptic support follows from Proposition \ref{Prop: elliptic linear rel character}.\qed

\subsubsection{Proof of Theorem \ref{Thm: global ii result Lvalue to period}} Recall now that $S= \{v_1,v_2\}$. In view of Propositions \ref{Prop: global rel char identity ii}, it suffices to show that the assumption implies we can construct $\wt{f}$ and $\wt{f}_\varepsilon$ which may fit into a $S$-compatible complete transfer and for which 
\begin{equation}\label{eqn: sum of linear traces}
  I_{\pi_{0}}^{(\eta,1)}(\wt{f},0)+I_{\pi_{0}}^{{(\eta,\eta)}}(\wt{f}_\varepsilon,0)+I_{\pi_{0}\otimes\eta}^{(\eta,1)}(\wt{f},0)+I_{\pi_{0}\otimes \eta}^{{(\eta,\eta)}}(\wt{f}_\varepsilon,0)\neq 0.  
\end{equation}

We now construct an elliptic nice function $\wt{f}=\bigotimes_v\wt{f}_v$ for which
\begin{enumerate}
    \item\label{all 4 terms} $\prod_vI^{(\eta,\eta),\natural}_{\pi_{0,v}}(\wt{f}_v) = \prod_vI^{(\eta,\eta),\natural}_{\pi_{0,v}\otimes\eta_v}(\wt{f}_v)=\prod_vI^{(\eta,1),\natural}_{\pi_{0,v}}(\wt{f}_v)=\prod_vI^{(\eta,1),\natural}_{\pi_{0,v}\otimes\eta_v}(\wt{f}_v)\neq0$, and
    \item there exists $f\in C_c^\infty(\G_{ii}(\A_{F}))$ such that $(f,\wt{f},\wt{f})$ forms a complete transfer.
\end{enumerate}

The function $\wt{f}$ is built precisely as in the previous section. When $v$ is inert, we choose $\wt{f}_v$ to be the test function occurring in the $k=0$ case of Theorem \ref{Thm: fundamental lemma ii}. For $v\neq v_1,v_2$ split, we saw in the previous section that the local relative character is non-vanishing so any choice of $\wt{f}_v$ such that $I^{(\eta,\eta)}_{\pi_{0,v}}(\wt{f}_v)\neq0$ is sufficient. For $i\in \{1,2\}$, since $v_i$ is split, the same arguments as in the previous section show we may assume $\wt{f}_{v_1}= f_{v_1}'\otimes \Phi_{v_1}$ where $f_{v_1}'\in C_c^\infty(\G'(F_{v_1}))_{\Omega_1}$, where $\Omega_1$ is the supercuspidal Bernstein block associated to $\pi_{0,v_1}$ and that $f_{v_2}'\in C_c^\infty(\G'(F_{v_2}))$ to have elliptic support.

Thus, for each inert place,
\[
I^{(\eta,1),\natural}_{\pi_{0,v}}(\wt{f}_v) =I^{(\eta,1),\natural}_{\pi_{0,v}\otimes\eta_v}(\wt{f}_v)=I^{(\eta,\eta),\natural}_{\pi_{0,v}}(\wt{f}_v)=I^{(\eta,\eta),\natural}_{\pi_{0,v}\otimes\eta_v}(\wt{f}_v)=1,
\]
and since $\eta_v=1$ when $v$ is split all four local relative characters agree in this case, so we obtain \eqref{all 4 terms}.
By Lemma \ref{Lem: S-compatible ii}, we may find a transfer $f=\bigotimes_vf_v\in C_c^\infty(\G_{ii}(\A_{F}))$. It is then clear from Theorem \ref{Thm: fundamental lemma varepsilon} that $(f,\wt{f},\wt{f})$ forms a $S$-compatible complete transfer.

Proposition \ref{Prop: factorize the linear rel char} now implies that the sum in \eqref{eqn: sum of linear traces} is
   \begin{equation}\label{eqn: sum of Lvalues}
       \left(\mathcal{L}^{(\eta,1,1)}_{\pi_0}+\mathcal{L}^{(\eta,1,\eta)}_{\pi_0}+\mathcal{L}^{(\eta,\eta,1)}_{\pi_0}+\mathcal{L}^{(\eta,\eta,\eta)}_{\pi_0}\right)\prod_vI^{(\eta,1),\natural}_{\pi_{0,v}}(\wt{f}_v).
   \end{equation}
 Recalling that 
\[
\mathcal{L}^{(\eta_0,\eta_1,\eta_2)}_{\pi_0}=\frac{L(1/2,\pi_0\otimes \eta_1)L(1/2,\pi_0\otimes \eta_2)L(1,\pi_0,\wedge^2\otimes\eta_0)}{L(1,\pi_0,\Sym^2)},
\]
and noting that
\[
L(1,\pi_0,\Sym^2)= L(1,\pi_0\otimes \eta,\Sym^2), \quad L(1,\pi_0,\wedge^2)= L(1,\pi_0\otimes \eta,\wedge^2),
\]\eqref{eqn: sum of Lvalues} is a non-zero value times
\[
L(1/2,\pi_0)^2+L(1/2,\mathrm{BC}_E(\pi_0))+L(1/2,\mathrm{BC}_E(\pi_0))+L(1/2,\pi_0\otimes \eta)^2 = (L(1/2,\pi_0)+L(1/2,\pi_0\otimes \eta))^2.
\]
The two summands are central values of standard $L$-functions of cuspidal automorphic representations of symplectic type. By a theorem of Lapid--Rallis \cite[Theorem 1]{LapidRallis}, such values are always non-negative. As noted in \S \ref{Section: BF integral}, the $L$-functions arising from the Bump--Friedberg integral and the Friedberg--Jacquet are known to agree with the standard $L$-functions considered in \emph{ibid.} by the local work of \cite{MatringeBF,MatringeBF2} in the non-archimedean setting and \cite{Ishii} in the archimedean setting.

This implies that 
\[
L(1/2,\pi_0),\;L(1/2,\pi_0\otimes \eta)\geq0;
\]
since our assumption is that at least one factor is non-zero (hence positive), the sum must be strictly positive and the theorem follows.\qed

\appendix

\section{An ellipticity result at non-archimedean places}\label{Appendix: elliptic} 

We assume $F$ is a non-archimedean local field of characteristic zero. In this section, we prove Theorem \ref{Thm: FJ elliptic}, which we recall here for the convenience of the reader. 
\begin{Thm}
    Let $(\G',\rH')=(\G,\rH)=(\GL_{2n},\GL_n\times \GL_n)$ be the split unitary symmetric pair over $F$. Let $\pi$ be a supercuspidal representation of $\G'(F)$ and assume that $\Hom_{\rH(F)}(\pi,\cc)\neq0$. Then $\pi$ is $\rH'$-elliptic. More precisely, for any non-zero $\lam\in \Hom_{\rH(F)}(\pi,\cc)$, there exists $f\in C^\infty_c(\G(F)^{ell})$ supported in the $\rH\times \rH$-elliptic locus such that 
    \[
    J_{\lam,\lam}(f):=\sum_{v\in OB(\pi)}\lam(\pi(f)v)\overline{\lam(v)}\neq0
    \]
\end{Thm}

By \cite[Theorem 1]{JacquetRallis}, we know that if $\pi$ is $\rH$-distinguished, then
\[
\dim\Hom_{\rH(F)}(\pi,\cc)=1.
\]
Let $\lam$ be a non-zero such functional and let $J:=J_{\lam,\lam}$.
We first recall the relation between invariant functionals and matrix coefficients in the supercuspidal case.
\begin{Lem} \cite[Theorem 1.5]{chongzhanglocal}
  For any $\check{v}\in \check{\pi}$ and $v\in \pi$, the matrix coefficient $\Phi_{v,\check{v}}(g) = \la \pi(g)v,v'\ra$ is integrable on $\rH(F)/Z_{\G}(F)$. Moreover, for any $\lam\in \Hom_{\rH(F)}(\pi,\cc)$ there exists $\check{v}\in \check{\pi}$ such that $\lam = \lam_{\check{v}}$, where
  \[
  \lam_{\check{v}}(v)= \int_{\rH(F)/Z_{\G}(F)}\la\pi(h)v,\check{v}\ra dh.
  \]
\end{Lem}

Now fix $\check{v}$ such that $\lam= \lam_{\check{v}}$. By assumption, there exists $v\in \pi$ such that $\lam_{\check{v}}(v)\neq0$. Set  $\Phi(g) = \la \pi(g)v,v'\ra$ and define $\Phi_{\X}\in C^\infty(\X(F))$ by
\[
\Phi_{\X}(s_{\X}(g)) = \int_{\rH(F)/Z_{\G}(F)}\Phi(gh)dh.
\]
Note that $\Phi_{\X}\in C^\infty_c(\X(F))$ has compact support since $\Phi$ has compact support modulo $Z_{\G}(F)$. Moreover, we have
\[
\Phi_{\X}(1) = \lam_{\check{v}}(v) \neq0.
\]

\begin{Lem}\label{Lem: matrix coeff} Let $\pi$ be a $\rH'$-distinguished supercuspidal representation of $\G'(F)$ and  let $\lam=\lam_{\check{v}}\in \Hom_{\rH(F)}(\pi,\cc)$ and $\Phi_{v,\check{v}}$ be as above. Then for any $f\in C_c^\infty(\G(F))$, we have
\[
J(f)= \frac{1}{\Phi_{\X}(1)}\int_{\G^{rrs}(F)}\Orb^{\rH}(\Phi_{\X},s_{\X}(g))f(g)dg.
\]
\end{Lem}
\begin{proof}
    The proof of this follows from the formula for $\lam_{\check{v}}$ via the same argument as that in \cite[Appendix A]{ZhangFourier}. 
\end{proof}

The upshot is that to prove Theorem \ref{Thm: FJ elliptic}, it suffices to show that there exists a $\rH$-regular elliptic element $y\in \X(F)$ such that 
\[
\Orb^{\rH}(\Phi_{\X},y)\neq 0.
\]
subject to the non-vanishing of $\Phi_{\X}(1)=\lam(v)\neq0$.

\begin{Thm}\label{Thm: elliptic non-vanishing}
    Suppose that $f\in C^\infty_c(\X(F))$ satisfies that $f(1)\neq0$. Then there exists $y\in \X(F)$ which is $\rH$-regular elliptic such that
    \[
    \Orb^{\rH}(f,y)\neq0.
    \]
\end{Thm}

To prove this we linearize and pass to the Lie algebra. Set $\fg=\Lie(\G)$ and recall the decomposition (cf. \cite[Section 5.1]{LXZfund})
\[
\fg=\fh\oplus \fX,
\] where $\fkh=\Lie(\rH)$ and $\fX\simeq \fgl_n\times \fgl_n$ where $(h_1,h_2)\in\rH(F)$ acts on $(X,Y)\in \fX$ via
\[
(h_1,h_2)\cdot (X,Y)  = (h_1Xh_2^{-1},h_2Y h_1^{-1}).
\]
\begin{Thm}\label{Thm: elliptic non-vanishing Lie}
    Suppose that $\phi\in C^\infty_c(\fX(F))$ satisfies that $f\phi(0)\neq0$. Then there exists $(X,Y)\in \fX(F)$ which is $\rH$-regular elliptic such that
    \[
    \Orb^{\rH}(\phi,(X,Y))\neq0.
    \]
\end{Thm}

\begin{Lem}
    Theorem \ref{Thm: elliptic non-vanishing Lie} implies Theorem \ref{Thm: elliptic non-vanishing}.
\end{Lem}
\begin{proof}
    
Recall from \cite[Appendix A]{LXZfund} the Cayley transform 
\begin{align*}
\fc:\fX(F)-D_1(F)&\lra \X(F)-D_{-1}(F)\\
			X&\longmapsto (1+X)(1-X)^{-1},
\end{align*}
induces a $\rH(F)$-equivariant homeomorphism where for any $\nu=\pm1$
\[
D_\nu=\{X\in \End(W): \det(\nu I_{2n}-X)=0\}.
\]

Now suppose that $f\in C^\infty_c(\X(F))$ satisfies that $f(1)\neq0$. Then there exists a compact open neighborhood $U\subset \X(F)-D_{-1}(F)$ of $1$ such that
\[
0\neq f_U = f\cdot \bfun_{U}\in C^\infty_c(\X(F)-D_{-1}(F)).
\]
To prove Theorem \ref{Thm: elliptic non-vanishing}, it clearly suffices to prove the claim for $f_U$, so we might as well assume that $f\in C_c^\infty(\X(F)-D_{-1}(F))$. Setting $\phi=f\circ \mathfrak{c}$, we see $\phi\in C_c^\infty(\fX(F)-D_1(F))$ such that $\phi(0)\neq0$. Since $(X,Y)\in \fX(F)$ is elliptic if and only if $\fc(X,Y)\in \X(F)$ is and
\[
  \Orb^{\rH}(\phi,(X,Y))=  \Orb^{\rH}(f,\fc(X,Y)),
\]
we see that Theorem \ref{Thm: elliptic non-vanishing} follows from  Theorem \ref{Thm: elliptic non-vanishing Lie}.
\end{proof}
\subsection{Proof of Theorem \ref{Thm: elliptic non-vanishing Lie}}
We begin with the following lemma.
\begin{Lem}
     Suppose that $\phi\in C^\infty_c(\fX(F))$ satisfies that
    \[
    \Orb^{\rH}(\phi,(X,Y))=0,
    \]
    for all $\rH$-regular elliptic $(X,Y)\in \fX(F)$. Then
    \[
    \displaystyle\int_{\GL_n(F)}\phi(th,h^{-1})dh=0
    \]
    for all $t\in F^\times$.
\end{Lem}
\begin{proof}
Recall the contraction map
\begin{align*}
    r: \fX = \fgl_n\times \fgl_n&\lra \fgl_n\\
    (X,Y)&\longmapsto XY,
\end{align*}
which intertwines the $\rH$-action on $\fX$ with the adjoint action of $\GL_n$ on $\fgl_n$ through the projection onto the first copy of $\fgl_n$. The orbit calculation \cite[Proposition 2.1]{JacquetRallis} shows it suffices to consider elements of the form $(X,I_n)$ where $X\in\GL_n(F)$ is regular semi-simple in the classical sense. Then 
\[
\rH_{(X,I_n)}\simeq (\GL_n)_X,
\]
so that $(X,I_n)$ is $\rH$-elliptic if and only if $X\in\GL_n(F)$ is elliptic.

Now suppose that $\phi\in C_c^\infty(\fX(F))$ satisfies
\[
  \Orb^{\rH}(\phi,(X,I_n))=0
\]
for all elliptic $X\in \GL_n(F)$. 

Fix an elliptic maximal torus $T\subset \GL_n(F)$, and element $t\in F^\times$, and a compact open neighborhood $U_1\subset \GL_n(F)$ of $I_n$. By assumption, $ \Orb^{\rH}(\phi,(tX,I_n))=0$ for all $X\in T^{reg}(F)\cap U_1$. Consider the push-forward $r_!(\phi)$: for any $X\in \GL_n(F)$, this is given by
\[
r_!(\phi)(x) = \int_{\GL_n(F)}\phi(xh,h^{-1})dh.
\]
Setting $f:=r_!(\phi)\cdot\bfun_{tU_1}$, we see that $f\in C_c^\infty(\GL_n(F))$ and
    \[
     \Orb^{\rH}(\phi,(tX,I_n)) =  \Orb^{\GL_n}(f,tX)
    \]
for all $X\in T^{reg}(F)\cap U_1$. Shrinking $U_1$ if necessary, we obtain the Shalika germ expansion \cite[Section 2]{rogawskiGL} around $\ga_0 = tI_n$
\begin{equation}\label{germ expansion group}
  0=  \Orb^{\rH}(\phi,(tX,I_n)) =  \Orb^{\GL_n}(f,tX) = \sum_{\calo}\Ga^T_\calo(tX)\Orb^{\GL_n}(f,t\calo),
\end{equation}
where the sum ranges over unipotent orbits in $\GL_n(F)$ and we range over $X\in T^{reg}\cap U_1$.   We need the following lemma.
\begin{Lem}
    Let $\mathcal{T}_e$ be a set of representatives of the conjugacy classes of elliptic maximal tori $T\subset\GL_n(F)$. Suppose that $\{C_\calo\}_{\calo}$ is a set of complex numbers such that there exists a neighborhood $U\subset \GL_n(F)$ of $I_n$ with
    \[
    \sum_\calo C_\calo\Ga_\calo^T(tX)=0
    \]
    for all $X\in U$ and for all $T\in\mathcal{T}_e$. Then $C_{\{I_n\}}=0$.
\end{Lem}
\begin{proof}
Suppose on the contrary that $C_{\{I_n\}}\neq0$. By \cite[Theorem 2.2]{rogawskiGL}, the germ $\Ga_{\{I_n\}}^T\neq0$ is a constant independent of $T$, so we conclude that there exists a non-zero $C\neq 0$ such that
   \begin{equation}\label{eqn: equal a constant}
       \sum_{\calo\neq \{I_n\}} C_\calo\Ga_\calo^T(tX)=C
   \end{equation}
 for all $X$ sufficiently close to $I_n$. Recall the homogeneity of the Shalika germs \cite[Theorem 14]{HarishChandra}
    \[
    \Ga_\calo^T(t\exp(a^2H)) = |a|^{-d_{\calo}}\Ga_\calo^T(t\exp(H))
    \] for $a\in \calo_F$, $H\in \fgl_n(F)$ sufficiently close to $0$ and where $d_{\calo}=\dim(\calo)$. Applying this to \eqref{eqn: equal a constant}, we find the equality
        \[
    \sum_{\calo\neq \{I_n\}} |a|^{-d_{\calo}}C_\calo\Ga_\calo^T(tX)=C\neq0
    \]
    for any $X\in U$ and all $a\in \calo_F$ sufficiently small. But $d_\calo>0$ for all terms in the sum, so the left-hand side of \eqref{eqn: equal a constant} tends toward $\infty$ as $X\to I_n$, contradicting the equality.
\end{proof}

Applying the lemma to the coefficients $\{\Orb^{\GL_n}(f,t\calo)\}_\calo$, \eqref{germ expansion group} now implies that 
\[
\Orb^{\GL_n}(f,tI_n) = r_!(\phi)(tI_n) = \int_{\GL_n(F)}\phi(th,h^{-1})dh = 0,
\]
completing the proof.
\end{proof}

It is now clear that Theorem \ref{Thm: elliptic non-vanishing Lie} now follows from the next result.

  \begin{Thm}
  \label{thm: germ1}
    Suppose that $\phi\in C_c^\infty(\fX(F))$ satisfies that \[
\wt\phi(t):= \int_{\GL_n(F)}\phi(th,h^{-1})dh = 0,
\]
for all (non-zero) $t$ in a neighborhood of $0$. Then $\phi(0)=0.$
 \end{Thm}
%
%
%


\subsection{Some explicit computations}

Before we present the proof of Theorem \ref{thm: germ1}, we need some preparations. The results here may be of independent interest in view of their connection to the local density (cf. \cite[\S2.4]{FYZ}).

Recall the Cartan decomposition
$$
\GL_n(F)=\coprod_{a\in A(F)/A(O)} K a K
$$ 
where $a=(a_1,\cdots,a_n), |a_1|\geq \cdots\geq |a_n| $. We can simply set $a=\varpi^\lambda, a_i=\varpi^{\lambda_i}, \lambda_i\in\BZ$, and $\lambda_1\leq\lambda_2\leq\cdots\leq \lambda_n.$
We recall the formula for $\vol(K \varpi^\lambda K)$
\begin{equation}\label{eq:vol KaK}
\vol(K \varpi^\lambda K)=\frac{\mu(M_\lambda)}{\mu(G)} \delta(\varpi^\lambda)^{-1}
\end{equation}
where $\delta(\lambda)$ is the modular character,
$M_\lambda$ is the centralizer of $\varpi^\lambda$, and $\mu(M)=\frac{\vol(I_M w_\ell I_M)}{\vol(M(O))}$      where $I_M w_\ell I_M$ is the largest Bruhat cell. The constant $\mu(M_\lambda)$ can be described explicitly: for $\lambda$ such that 
$$
\lambda_1=\cdots=\lambda_{n_1}<\lambda_{n_1+1}=\cdots=\lambda_{n_1+n_2}<\cdots,
$$
we call the (ordered) sequence of the lengths of the segments of equal values $(n_1,\cdots,n_s)$ ($\sum n_i=n$)  {\em the type} of $\lambda$, and define $$
|\lambda|:=\sum\lambda_i.
$$ Then
the centralizer is
$$
M_\lambda\simeq\prod_{i}\GL_{n_i}.
$$
Moreover, we have
$$
\mu(\GL_n)=\prod_{i=1}^n\frac{1-q^{-1}}{1-q^{-i}}=\prod_{i=1}^n\frac{\zeta(i)}{\zeta(1)}.
$$

Let $\BZ^{n,+} $ denote the set of $\lambda$ such that
$$
 \lambda_1\leq \cdots\leq \lambda_n,
$$
and $\BZ^{n,+}_{[a,b]}$ the subset consisting of $\lambda$ satisfying
$$
a\leq \lambda_1\leq \cdots\leq \lambda_n\leq b.
$$
For $n\geq 1$, and $\alpha\in\BZ$, we define for $x\in\BZ_{\geq 0}$
\begin{equation}\label{eq:vol n a}
\vol_{n,\alpha}(x)=\sum_{\lambda\in\BZ^{n,+}_{[0,x]} } q^{\alpha |\lambda|}\vol(K_n\varpi^\lambda K_n).
\end{equation}
As a convention we define $\vol_{n,\alpha}(x)=1$ when $n=0$. 
\begin{Rem}
The function $\vol_{n,0}(x)$ is essentially the Hermitian local density when the quadratic field is split (cf. \cite[\S2.4]{FYZ}), specialized at scalar matrices $\varpi^x {\bf 1}_n$. 
\end{Rem}
Note that there is a ``functional equation" by substituting $\varpi^\lambda$ by $\varpi^{x}\varpi^{-\lambda}$
\begin{equation}\label{eq:vol n a FE}
\vol_{n,\alpha}(x)=  q^{n\alpha x}\vol_{n,-\alpha}(x).
\end{equation}

Recall  the $q$-binomial coefficient 
$$
 {n\choose \beta}_q:=\frac{[n]_q!}{[\beta]_q![n-\beta]_q! },
$$
where $[n]_q!:=\prod_{i=1}^n\frac{(1-q^i)}{(1-q)}$. It is the number of $\beta$-dimensional subspaces of the $n$-dimensional vector space over $\BF_q$. Note that
$$
 {n\choose \beta}_q=q^{\beta(n-\beta)}\frac{\mu( \GL_{\beta})\mu(\GL_{n-\beta})}{\mu(\GL_n)}.
$$ 

\begin{Lem}
We have
\begin{align}\label{eq recur}
\vol_{n,\alpha}(x)=\sum_{0\leq y\leq x} q^{\alpha n y}+ \sum_{1\leq \beta\leq n-1}  {n\choose \beta}_q  \,q^{\alpha(n-\beta)} q^{\alpha n(x-1)}\sum_{0\leq y\leq x-1}
q^{-\alpha n y}\vol_{n-\beta, \alpha+\beta}(y).
\end{align}

\end{Lem}

\begin{proof}
We compute the sum over $\lambda\in\BZ^{n,+}_{[0,x]}$ according to the length $\beta:=n_1$ of the first segment in the type $(n_1,\cdots, n_s)$ of $\lambda$. We have $1\leq \beta\leq n$, and 
$$
0\leq\lambda_1=\cdots=\lambda_{\beta}< \lambda_{\beta+1}\leq \cdots\leq x.
$$
Denote $\lambda^\flat$ the sequence obtained by deleting the first $\beta$ terms, namely $\lambda^\flat=( \lambda_{\beta+1},\cdots,  \lambda_{n})$.
Then $$
M_{\lambda}\simeq \GL_{\beta}\times M_{\lambda^\flat},
$$
and 
$$
\vol(K_n\varpi^\lambda K_n)=\frac{\mu( \GL_{\beta})\mu( M_{\lambda^\flat})}{\mu(\GL_n)}\delta(\varpi^\lambda)^{-1}.
$$
We note that 
\begin{align*}
\delta(\varpi^\lambda)^{-1}&=\prod_{1\leq i\leq\beta, \beta+1\leq j\leq n}q^{\lambda_j-\lambda_i} \delta(\varpi^{\lambda^\flat})^{-1}\\
&=  q^{-\beta(n-\beta)\lambda_1} q^{\beta|\lambda^\flat|} \delta(\varpi^{\lambda^\flat})^{-1}.
\end{align*}
Therefore we have 
$$
\vol(K_n\varpi^\lambda K_n)=\frac{\mu( \GL_{\beta})\mu(\GL_{n-\beta})}{\mu(\GL_n)}q^{-\beta(n-\beta)\lambda_1} q^{\beta|\lambda^\flat|}  \vol(K^\flat\varpi^{\lambda^\flat} K^\flat).
$$
where $K^\flat=K_{n-\beta}$.

We obtain
\begin{align*}
\vol_{n,\alpha}(x)&=\sum_{1\leq \beta\leq n}\sum_{\lambda\in\BZ^{n,+}_{[0,x]}, n_1(\lambda)=\beta} q^{\alpha |\lambda|}\vol(K_n\varpi^\lambda K_n)
\\&=\sum_{1\leq \beta\leq n} \frac{\mu( \GL_{\beta})\mu(\GL_{n-\beta})}{\mu(\GL_n)} \sum_{0\leq \lambda_1\leq x}\sum_{\lambda^\flat\in \BZ^{n-\beta,+}_{[\lambda_1+1,x]}}q^{\alpha |\lambda^\flat|+\alpha\beta \lambda_1} q^{-\beta(n-\beta)\lambda_1} q^{\beta|\lambda^\flat|}  \vol(K^\flat\varpi^{\lambda^\flat} K^\flat).
\end{align*}
We note that  $\lambda^\flat\mapsto \lambda^\flat+(\lambda_1+1)$ (adding every component by the same number) defines a bijection $  \BZ^{n-\beta,+}_{[0,x-\lambda_1-1]}\simeq \BZ^{n-\beta,+}_{[\lambda_1+1,x]}$. We obtain

\begin{align*}\vol_{n,\alpha}(x)&=\sum_{1\leq \beta\leq n} \frac{\mu( \GL_{\beta})\mu(\GL_{n-\beta})}{\mu(\GL_n)} \sum_{0\leq \lambda_1\leq x}q^{(\alpha+\beta)(n-\beta)(\lambda_1+1)+\alpha\beta \lambda_1-\beta(n-\beta)\lambda_1} \vol_{n-\beta, \alpha+\beta}(x-(\lambda_1+1))
\\&= \sum_{1\leq \beta\leq n} \frac{\mu( \GL_{\beta})\mu(\GL_{n-\beta})}{\mu(\GL_n)} q^{(\alpha+\beta)(n-\beta)}\sum_{0\leq \lambda_1\leq x}
q^{\alpha n \lambda_1}\vol_{n-\beta, \alpha+\beta}(x-(\lambda_1+1))
\end{align*}

When $\beta=n$ the summand contributes
$$\sum_{0\leq y\leq x} q^{\alpha n y}
$$
When $\beta<n$, note that the term $\lambda_1=x$ vanishes and we have  
$$
\sum_{0\leq \lambda_1\leq x}
q^{\alpha n \lambda_1}\vol_{n-\beta, \alpha+\beta}(x-(\lambda_1+1))=q^{\alpha n(x-1)}\sum_{0\leq y\leq x-1}
q^{-\alpha n y}\vol_{n-\beta, \alpha+\beta}(y).
$$

\end{proof}

\begin{Lem}\label{lem: ind 2}
We have
\begin{align}\label{eq recur 2}
\vol_{n,\alpha}(x)=\sum_{0\leq \beta\leq n} {n\choose \beta}_q \, q^{\alpha(n-\beta)}\vol_{n-\beta,\alpha+\beta}(x-1).
\end{align}
In particular, when $\alpha=0$ we have
\begin{align}\label{eq recur 2 0}
\vol_{n,\alpha=0}(x)=\sum_{0\leq \beta\leq n} {n\choose \beta}_q\, \vol_{n-\beta,\beta}(x-1).
\end{align}

\end{Lem}

\begin{proof}
We compute $\vol_{n,\alpha}(x)$ according to whether the number $\beta$ of $\lambda_i=0$. Fix $0\leq \beta\leq n$. Assume 
$$
\lambda_1=\lambda_2=\cdots=\lambda_\beta=0, \lambda_{\beta+1}>0.
$$
Write $\lambda^\flat=( \lambda_{\beta+1}-1,\cdots  \lambda_{n}-1)$.
Then 
$$
q^{\alpha |\lambda|}\vol(K\varpi^\lambda K)=q^{(\alpha+\beta)(|\lambda^\flat|+n-\beta)}\frac{\mu( \GL_{\beta})\mu(\GL_{n-\beta})}{\mu(\GL_n)}\vol(K^\flat \varpi^{\lambda^\flat} K^\flat).
$$
We obtain
$$\vol_{n,\alpha}(x)=\sum_{0\leq \beta\leq n} \frac{\mu( \GL_{\beta})\mu(\GL_{n-\beta})}{\mu(\GL_n)} q^{(\alpha+\beta)(n-\beta)}\vol_{n-\beta,\alpha+\beta}(x-1).
$$

\end{proof}

\begin{Lem}\label{lem vol alpha}
Suppose that $n\geq 1$.
\begin{enumerate}
\item The germs when $x\to\infty$ of the function 
$\vol_{n,\alpha}$ on $\BZ_{\geq 0}$ is a linear combination of the following functions:
$ q^{a_i x}x^{b_i}
$
where $a_i, b_i\in\BZ_{\geq 0}$.  In other words,
there exists constants $c_{(a_i,b_i)}\in\BC$ such that 
\begin{equation}\label{eq:asymp n a}
\vol_{n,\alpha}(x)=\sum
c_{(a_i,b_i)}q^{a_i x}x^{b_i}
\end{equation}
holds
for all $x\geq 1$. (Clearly such expression is unique by the linear independence of germs of functions of the forms $q^{a_i x}x^{b_i}$.) 

\item
Assume that $\alpha\geq 1$. If the coefficient of $q^{a_i x}x^{b_i}$ in \eqref{eq:asymp n a} does not vanish, then either $a_i\geq 1$, or $(a_i,b_i)=(0,0)$. 
\end{enumerate}

\end{Lem}

\begin{proof}Part (1): Induction on $n$. The case $n=1$ is clear.  Assume that the assertion holds up to $n-1\geq 1$. 

The first sum in \eqref{eq recur} can be computed explicitly

\begin{equation}\label{beta=n}
\sum_{0\leq y\leq x} q^{\alpha n y}=\begin{cases}\frac{q^{\alpha n (x+1)}-1}{q^{\alpha n}-1},& \alpha\geq 1\\
x+1,& \alpha=0\end{cases}
\end{equation}
Now fix $1\leq \beta\leq n-1$.
By induction hypothesis, the function $\vol_{n-\beta, \alpha+\beta}(y)$, when $y$ is large, is a  linear combination of functions of the form $q^{a y}y^b, a,b\in\BZ_{\geq0}$. Hence the sum
$$
\sum_{0\leq y\leq x-1}
q^{-\alpha n y}\vol_{n-\beta, \alpha+\beta}(y)$$ is, up to a constant, a linear combination of 
$$
\sum_{0\leq y\leq x-1} q^{(a-\alpha n) y}y^b.
$$
The above sum is clear a linear combination of functions of the form $$ q^{(a-\alpha n) x} x^{b_i}, \quad b_i\in\BZ, 0\leq b_i\leq b+1$$
and the constant function.
 Hence $q^{\alpha n x}
\sum_{0\leq y\leq x-1}
q^{-\alpha n y}\vol_{n-\beta, \alpha+\beta}(y)$
 a linear combination of functions of the form $$ q^{a x} x^{b_i}, \quad a\geq0, b_i\in\BZ, 0\leq b_i\leq b+1$$
and $$
q^{\alpha n x}.
$$
By \eqref{eq recur}, the assertion for $n$ is proved.

Part (2). Now assume that $\alpha\geq 1$, then in the induction step for $1\leq \beta\leq n-1$, we can assume $a\geq 1$, in which case the resulting functions are as desired, or $a=b=0$ (the constant function),  in which case the resulting functions are $q^{\alpha n x}$ and $1$.  The term \eqref{beta=n} (for $\beta=n$) is also as desired by direct inspection.

\end{proof}
\quash{\begin{remark}
The alternative proof above could also simplify the proof that ``when $\alpha\geq0$, the coefficients $c_{(a,b)}$ vanishes unless $a\geq1$ or $(a,b)=(0,0)$" noting that the polynomial $\sum c_i x^i$ does not seem to have the desired congruence?
\end{remark}}

Next we compute the coefficient for the linear term $x$ in the unique expansion in \eqref{eq:asymp n a}. For that purpose, we first compute the constant term of $\vol_{n,\alpha}(x)$, which we denote by $A(n,\alpha)$.

\begin{Lem}
Let $\alpha\geq 1$. Then we have
$$
A(n,\alpha)=\prod_{i=0}^{n-1}(1-q^{\alpha+i})^{-1}.
$$
\end{Lem}

\begin{proof}
Induction on $n$. Clearly $A(0,\alpha)=1, A(1,\alpha)=\frac{1}{1-q^\alpha}$. Suppose that $n\geq 2$ and the desired identity holds up to $n-1$.

From Lemma \ref{lem: ind 2} we have a recursion 
$$
(1-q^{\alpha n})A(n,\alpha)=\sum_{\beta=1}^n   {n\choose \beta}_q \, q^{\alpha(n-\beta)} A(n-\beta,\alpha+\beta).
$$
(Note that the higher term $q^{ax}x^b$ in $\vol_{n,\alpha}$ can not contribute nontrivially to the constant term in $\vol_{n,\alpha}(x)-q^{\alpha n}\vol_{n,\alpha}(x-1)$.)

Therefore it suffices to verity 
\begin{align}\label{eq: id q-binom}
(1-q^{\alpha n})\prod_{i=0}^{n-1}(1-q^{\alpha+i})^{-1}=\sum_{\beta=1}^n   {n\choose \beta}_q \, q^{\alpha(n-\beta)}\prod_{i=0}^{n-\beta-1}(1-q^{\alpha+\beta+i})^{-1}.
\end{align}
Equivalently
$$
(1-q^{\alpha n})=\sum_{\beta=1}^n   {n\choose \beta}_q \, q^{\alpha(n-\beta)}\prod_{i=0}^{\beta-1}(1-q^{\alpha+i})
$$
which is equivalent to
$$
1=\sum_{\beta=0}^n   {n\choose \beta}_q \, q^{\alpha(n-\beta)}\prod_{i=0}^{\beta-1}(1-q^{\alpha+i})
$$
where in the term for $\beta=0$, the product $\prod_{i=0}^{\beta-1}(1-q^{\alpha+i})$ is understood as $1$.
It suffices to show $$
q^{-\alpha n}=\sum_{\beta=0}^n   {n\choose \beta}_q \,\prod_{i=0}^{\beta-1}(q^{-\alpha}-q^{i}).
$$This follows from the  following identity of $q$-binomial coefficients, 
$$
X^n=\sum_{\beta=0}^n   {n\choose \beta}_q \,\prod_{i=0}^{\beta-1}(X-q^{i})
$$
as polynomials in $X$. (A sketch of the proof: enough to consider $X=q^m$ for all $m\geq 0$. Let $V_n=\BF_q^n$ be a $n$-dimensional vector space over $\BF_q$. Then both sides count the cardinality of $\Hom_{\BF_q}(V_n,V_m)$.)
This completes the induction.
\end{proof}

\begin{Lem}\label{lem: c(0,1)}
Let $\alpha=0$ and $n\geq 1$. Then the coefficient of $x$ in the expansion of $\vol_{n,0}(x)$ into the form \eqref{eq:asymp n a} is equal to 
\begin{equation}\label{eq: cont Bn}
B_n: = n\prod_{i=1}^{n-1}(1-q^{i})^{-1}.
\end{equation}
In particular, it never vanishes.
\end{Lem}

\begin{proof}

 Using \eqref{eq recur 2 0} we have
\begin{align}
\vol_{n,0}(x)-\vol_{n,0}(x-1)=\sum_{1\leq \beta\leq n}  {n\choose \beta}_q \vol_{n-\beta,\beta}(x-1).
\end{align}
Hence the coefficient of $x$ in $\vol_{n,0}(x)$ is equal to 
$$
\sum_{1\leq \beta\leq n}  {n\choose \beta}_q  A(n-\beta,\beta)=\sum_{1\leq \beta\leq n}  {n\choose \beta}_q\,\prod_{i=0}^{n-\beta-1}(1-q^{\beta+i})^{-1}.
$$
We now apply \eqref{eq: id q-binom} to the case $\alpha=0$. Note that this still makes sense as the LHS can be interpreted as $$
\frac{(1-q^{\alpha n})}{1-q^\alpha}\prod_{i=1}^{n-1}(1-q^{\alpha+i})^{-1}.
$$
We obtain
\begin{align*}
\frac{(1-q^{\alpha n})}{1-q^\alpha}|_{\alpha=0}\prod_{i=1}^{n-1}(1-q^{i})^{-1}=\sum_{\beta=1}^n   {n\choose \beta}_q \, \prod_{i=0}^{n-\beta-1}(1-q^{\beta+i})^{-1}.
\end{align*}
This gives the desired assertion.
\end{proof}

\subsection{Proof of Theorem \ref{thm: germ1}}
We have the following coarse form of a ``Shalika germ expansion", which may be of some independent interest in ``relative" harmonic analysis.  
 \begin{Lem}\label{lem: germ} Let $\phi\in C_c^\infty(\fX(F))$.
 As a function of $t\in F^\times$ near $0$, we have an expansion
\begin{equation}\label{eq: germ}
\wt\phi(t)=\sum_{(n_1,n_2,n_3)\atop n_i\geq 0, n_1+n_2+n_3=n} c_{(n_1,n_2,n_3)}(\phi) q^{n_3(n_1+n_2)v(t)}\vol_{n_2,n_1-n_3} (v(t))
\end{equation}
for some constants $c_{(n_1,n_2,n_3)}(\phi)$ independent of $t$. 

 \end{Lem}

 \begin{proof}

 By scaling the variables we may assume that $\phi$ is invariant under (translation by) the lattice $\Lambda=\fX(\CO_F)$. We assume that the support of $\phi$ is (contained in) $\varpi^{-N}\Lambda$ for some large integer $N$. Furthermore, by first integrating over $K\times K$ we can assume that $\phi$ is bi-$K$-invariant in the sense that $\phi(k_1^{-1}xk_2^{-1},k_2yk_1)=\phi(x,y)$ for all $k_1,k_2\in K$. Then the orbital integral is a finite sum 
 $$
\wt\phi(t)=\sum_{\lambda\in \BZ^{n,+} } \phi(t\varpi^{-\lambda},\varpi^\lambda)\vol(K\varpi^ \lambda K ).
 $$
By the support condition we may assume
$$
-N\leq \lambda_i\leq v(t)+N.
$$

We break the sum into sub-sums according to each variable $\lambda_i$ in comparison with $ 0$ and $v(t)$ respectively. There are the following possibilities: for  $n_1,n_2,n_3\geq 0, n_1+n_2+n_3=n$, we consider the condition for $\lambda$:
$$
-N\leq \lambda_1\leq \cdots\leq \lambda_{n_1}<0\leq \lambda_{n_1+1}\leq\cdots\leq \lambda_{n_1+n_2}\leq v(t)<\lambda_{n_1+n_2+1}\leq\cdots \leq \lambda_n\leq  v(t)+N.
$$
We call $\lambda$ of class $(n_1,n_2,n_3)$. Then when $v(t)$ is large, using the $\Lambda$-invariance, we have
\begin{align*}
\wt\phi(t)=&\sum_{\lambda\in \BZ^{n,+}, \atop \text{Class }  (n_1,n_2,n_3)} \vol(K\varpi^ \lambda K )\\
 &\phi\left(\left(\begin{array}{ccc}
 0_{n_1}    & &  \\
      &  0_{n_2} &\\ 
      && t \varpi^{-(\lambda_{n-n_3+1},\cdots,\lambda_{n})} 
 \end{array}\right) ,\left(\begin{array}{ccc}
 \varpi^{(\lambda_1,\cdots,\lambda_{n_1})}   & &  \\
      & 0_{n_2}&\\ 
      && 0_{n_3}
 \end{array}\right) \right).
\end{align*}

We write a $ \lambda\in \BZ^{n,+}$ of class $(n_1,n_2,n_3)$ to the concatenation $(\xi_{1},\xi_2=\lambda^\flat,\xi_3)$ where  $\xi_1=(\lambda_1,\cdots,\lambda_{n_1})\in \BZ^{n_1,+}_{[-N,-1]}$,  $\lambda^\flat=(\lambda_{n_1+1},\cdots,\lambda_{n_1+n_2})\in\BZ^{n_2,+}_{[0,v(t)]}$ and $\xi_3=(\lambda_{n-n_3+1},\cdots,\lambda_{n})\in \BZ^{n_3,+}_{[1+v(t),N+v(t)]}$. Then by \eqref{eq:vol KaK} we have
$$
\vol(K_n\varpi^ \lambda K_n )=\frac{\prod_{i=1}^3\mu(\GL_{n_i})}{\mu(\GL_n)} q^{-(n_2+n_3)|\xi_1|+(n_1-n_3)|\xi_2|+(n_1+n_2)|\xi_3|}\prod_{i=1}^3\vol(K_{n_i}\varpi^ {\xi_i} K_{n_i} ).
$$
For any integer $m$ we write $m+\xi_3$ for the tuple $(m+\lambda_{n-n_3+1},\cdots,m+\lambda_{n})$. Then for $\xi_3 \in \BZ^{n_3,+}_{[1+v(t),N+v(t)]}$  we may write it as $v(t)+\xi'_3$ for a unique $\xi_3'\in \BZ^{n_3,+}_{[1,N]}$.  We then have
$$
\vol(K_n\varpi^ \lambda K_n )=\frac{\prod_{i=1}^3\mu(\GL_{n_i})}{\mu(\GL_n)} q^{-(n_2+n_3)|\xi_1|+(n_1+n_2)|\xi_3'|}q^{(n_1-n_3)|\xi_2|+n_3(n_1+n_2)v(t)}\prod_{i=1}^3\vol(K_{n_i}\varpi^ {\xi_i} K_{n_i} ).
$$
We denote  $$c_{\xi_1,\xi'_3} =\frac{\prod_{i=1}^3\mu(\GL_{n_i})}{\mu(\GL_n)} q^{-(n_2+n_3)|\xi_1|+(n_1+n_2)|\xi_3'|} \vol(K_{n_1}\varpi^ {\xi_1} K_{n_1} )\vol(K_{n_3}\varpi^ {\xi'_3} K_{n_3} ) $$ which depends only on $\xi_1,\xi_3'$ but not on $\xi_2=\lambda^\flat$. Note that $\vol(K_{n_i}\varpi^ {\xi_3} K_{n_i} )=\vol(K_{n_i}\varpi^ {\xi_3'} K_{n_i} )$. Then we have
\begin{align*}
 \wt\phi(t)=\sum_{ (n_1,n_2,n_3)}\sum_{\xi_1\in \BZ^{n_1,+}_{[-N,-1]},\atop \xi'_3\in \BZ^{n_3,+}_{[1,N]}} & c_{\xi_1,\xi'_3} 
 \phi\left(\left(\begin{array}{ccc}
 0_{n_1}    & &  \\
      &  0_{n_2} &\\ 
      &&\varpi^{-\xi'_3} 
 \end{array}\right) ,\left(\begin{array}{ccc}
 \varpi^{\xi_1}   & &  \\
      & 0_{n_2}&\\ 
      && 0_{n_3}
 \end{array}\right) \right)
\\&
\sum_{\lambda^\flat\in\BZ^{n_2,+}_{[0,v(t)]} }q^{(n_1-n_3)|\lambda^\flat|+n_3(n_1+n_2)v(t)}\vol(K_{n_2}\varpi^ {\lambda^\flat} K_{n_2} ) .
\end{align*}

By  the definition of the volume \eqref{eq:vol n a}, we have
\begin{align}\label{eq:Orb germ}
 \wt\phi(t)
 =
 \sum_{ (n_1,n_2,n_3)} \sum_{\xi_1\in \BZ^{n_1,+}_{[-N,-1]},\atop \xi'_3\in \BZ^{n_3,+}_{[1,N]}}  & c_{\xi_1,\xi'_3} 
 \phi\left(\left(\begin{array}{ccc}
 0_{n_1}    & &  \\
      &  0_{n_2} &\\ 
      &&  \varpi^{-\xi_3'} 
 \end{array}\right) ,\left(\begin{array}{ccc}
 \varpi^{\xi_1}   & &  \\
      & 0_{n_2}&\\ 
      && 0_{n_3}
 \end{array}\right) \right)\\
&q^{n_3(n_1+n_2)v(t)}\vol_{n_2,n_1-n_3} (v(t))\notag.
\end{align} 
This shows that as a function of $t\in F^\times$ near $0$
$$\wt\phi(t)=\sum_{ (n_1,n_2,n_3)} c_{(n_1,n_2,n_3)}(\phi) q^{n_3(n_1+n_2)v(t)}\vol_{n_2,n_1-n_3} (v(t))
$$
for constants $c_{(n_1,n_2,n_3)}(\phi)$. 

\end{proof}

Finally, we return to the proof of Theorem \ref{thm: germ1}. 
Let $x=v(t)$ and we view both sides of \eqref{eq: germ} in Lemma \ref{lem: germ} as functions in $x\in\BZ$ near $\infty$. Note that when $n_1-n_3\leq 0$ we can apply the functional equation \eqref{eq:vol n a FE} to obtain
\begin{equation}\label{eq: FE q}
q^{n_3(n_1+n_2)x}\vol_{n_2,n_1-n_3} (x)=q^{n_1(n_3+n_2)x}\vol_{n_2,n_3-n_1} (x).
\end{equation}
Therefore by Lemma \ref{lem vol alpha}
we have expansions of every $q^{n_3(n_1+n_2)x}\vol_{n_2,n_1-n_3} (x)$ into the form \eqref{eq:asymp n a}. 
We {\em claim} that the linear term (to clarify, by ``linear term" we mean only the degree-one-in-$x$ term, not including the constant term) in the expansion of $\wt\phi(\varpi^x)$ is exactly 
$\phi(0)B_n x$ (coming from the summand $\phi(0)\vol_{n,0} (x)$, corresponding to $n_1=n_3=0, n_2=n$ in \eqref{eq: germ}), where $B_{n}$ is the nonzero constant given by \eqref{eq: cont Bn} in Lemma \ref{lem: c(0,1)}. 

To show the claim, it suffices to show that none of the other triple $(n_1,n_2,n_3)$ has any non-zero linear term. By Lemma \ref{lem vol alpha} (2) and \eqref{eq: FE q}, we know that this is true if $|n_1-n_3|\geq 1$. If $n_1=n_3>0$, this is also true, due to the factor $q^{n_3(n_1+n_2)x}=q^{n_1(n_3+n_2)x}$ which has the exponent $n_1(n_3+n_2)>0$. Therefore, the only remaining triple is $(n_1=0,n_2=n,n_3=0)$, as claimed.   

Hence, the assertion in Theorem \ref{thm: germ1} follows immediately by the linear independence of the terms $q^{a_ix}x^{b_i}$ in \eqref{eq:asymp n a} and the non-vanishing of $B_n$. 

This completes the proof of Theorem \ref{thm: germ1}, and hence Theorem \ref{Thm: elliptic non-vanishing Lie} and Theorem \ref{Thm: FJ elliptic}.

\section{An ellipticity result at archimedean places}\label{Sec compact}

Let $\BE/\BF$
denote the quadratic extension $\BC/\BR$. Let $\eta_{\BE/\BF}:\BF^\times\to\{\pm 1\}$ be the quadratic character by class field theory; in this case it is simply the sign character $\eta_{\BE/\BF}=\sgn$.

 Let $\BV$ be a positive definite $\BE/\BF$-Hermitian space of $\dim \BV=2n$, equipped with an orthogonal decomposition $\BV=\BW_1\oplus \BW_2$ with $\dim \BW_1=\dim \BW_2=n$. Let $\BG(\BF)=\U(\BV)$ and $\BH=\U(\BW_1)\times \U(\BW_2)$.
\begin{Lem}\label{lem: branc compact}
Let $\tau$ be an irreducible representation of the compact group $\BG(\BF)$. 
Assume that it is distinguished by $\BH(\BF)$, namely $\Hom_{\BH(\BF)}(\tau,\BC)\neq0$. Then we necessarily have 
$$
\dim\Hom_{\BH(\BF)}(\tau,\BC)=1.
$$and there is a unique irreducible tempered unitary representation $\tau_{0}$ of $\GL_{2n}(\BR)$ such that  $BC(\tau_{0})=BC(\tau)$. In particular we have $\tau_{0}\simeq  \tau_{0}\otimes\eta_{\BE/\BF}$.
 
\end{Lem}
\begin{proof}
Let $\BG(\BF)_\cc\simeq \GL_{2n}(\cc)$ denote the complexification $\BG(\BF)$, and similarly $\BH(\BF)_\cc\simeq \GL_n(\cc)\times \GL_n(\cc)$. Then $\tau$ extends canonically to a representation of $\BG(\BF)_\cc$, denoted by $\tau_{\cc}$. As complexification induces an equivalence of symmetric monoidal categories of finite-dimensional continuous complex $\BG(\BF)$-representations and the category of rational representations of $\BG(\BF)_\cc$, we see that
\begin{equation}\label{eqn: Homspace}
    \Hom_{\BH(\BF)}(\tau,\BC)\simeq\Hom_{\BH(\BF)_{\cc}}(\tau_{\cc},\BC).
\end{equation}
The multiplicity one statement is well known for the symmetric pair $(\BG(\BF)_{\cc},\BH(\BF)_{\cc})$ (see for example \cite{AizGourdescent}), but also follows from a Gelfand trick on $\BG(\BF)$ itself. 

Suppose now that $\tau$ has highest-weight $\lam = (\lam_1,\ldots, \lam_{2n})$, where $\lam_i\geq \lam_{i+1}$. Then the same is true of $\tau_{\cc}$, and we claim that 
\[
\Hom_{\BH(\BF)_{\cc}}(\tau_{\cc},\BC)\neq0
\]
if and only if $\lam$ is self-associate in the sense that $\lam_i + \lam_{2n+1-i}=0$. Indeed, this may be derived from the general branching law for $(\BG(\BF)_{\cc},\BH(\BF)_{\cc})$ in \cite[2.2.1]{HoweTanWillenbring}, along with basic properties of Littlewood--Richardson coefficients. More precisely, let $\lam^+ = (\lam_1^+,\ldots,\lam_p^+)$ and $\lam^- = (\lam_1^-,\ldots,\lam_q^-)$ be \emph{positive} partitions of lengths $p$ and $q$ respectively, where $p+q\leq 2n$, and where
\[
\lam = (\lam_1^+,\ldots,\lam_p^+,0,\ldots,0, -\lam_q^-,\ldots, -\lam_q^-).
\] By an abuse of notation, let $0$ denote the zero partition of any length. In the notation of \cite{HoweTanWillenbring}, $F^{(\lam^+,\lam^-)}_{(2n)}$ denotes the representation of $\GL_{2n}(\cc)$ of highest weight $\lam$, and the restriction multiplicity we desire is
\begin{align*}
    [F^{(\lam^+,\lam^-)}_{(2n)},F_{(n)}^{(0,0)}\otimes F_{(n)}^{(0,0)}] &= \sum_{\ga^+,\ga^-,\ga^0} c^{\ga^+}_{00}c^{\ga^-}_{00}c^{\lam^+}_{0\ga^0}c^{\lam^-}_{0\ga^0}\\
    &= \sum_{\ga^0} c^{\lam^+}_{0\ga^0}c^{\lam^-}_{0\ga^0}=\de_{\lam^+,\ga^0}\de_{\lam^-,\ga^0},
\end{align*}
where the sum ranges over all positive partitions $\ga^+$, $\ga^-$, and $\ga^0$ and where $c^\lam_{\mu\nu}$ denotes the Littlewood--Richardson coefficient. This last expression is a product of Dirac delta functions, and the claim follows.

In particular, if $\tau$ is distinguished by $\BH(\BF)$, it has self-associate highest weight $\lam= (\lam_1,\ldots,\lam_n,-\lam_n,\ldots, -\lam_1)$. It follows that its L-parameter is of the form $$\bigoplus_{1\leq i\leq n}\chi_{i}\oplus\,^c\chi_i$$
where each of $\chi_i$ is of a character of $\BC^\times$ of conjugate-symplectic type (namely $\chi_i=\,^c\chi_i^{-1}, \chi_i|_{\BR^\times}=\sgn$), $\,^c\chi_i$ denotes the conjugate character, and $\chi_i\neq\chi_j^{\pm 1}$ for $i\neq j$.
 Then there is a unique irreducible tempered unitary representation $\pi_{0,v}$ of $\GL_{2n}(\BR)$ such that  $BC(\pi_{0,v})=BC(\pi_v)$. In particular we have $\pi_{0,v}\simeq  \pi_{0,v}\otimes\eta_v$.
 \end{proof}
\quash{\begin{proof}
The multiplicity one follows from a Gelfand trick.
We {\em claim} that its L-parameter is of the form $$\bigoplus_{1\leq i\leq n}\chi_{i}\oplus\,^c\chi_i$$
where each of $\chi_i$ is of a character of $\BC^\times$ of conjugate-symplectic type (namely $\chi_i=\,^c\chi_i^{-1}, \chi_i|_{\BR^\times}=\sgn$), $\,^c\chi_i$ denotes the conjugate character, and $\chi_i\neq\chi_j^{\pm 1}$ for $i\neq j$.
 Then there is a unique irreducible tempered unitary representation $\pi_{0,v}$ of $\GL_{2n}(\BR)$ such that  $BC(\pi_{0,v})=BC(\pi_v)$. In particular we have $\pi_{0,v}\simeq  \pi_{0,v}\otimes\eta_v$.
 
 To show the claim, we {\bf TBA}
 \end{proof}}
 

Now let $J_{\tau}$ be 
the relative local character. Since the group $\BG(\BF)$ is compact and the local multiplicity one holds, it is well-defined up to a non-zero constant multiple and it represented by the orbital integral of some matrix coefficient of $\tau$ (cf. Lemma \ref{Lem: matrix coeff}).

\begin{Thm}\label{thm:arch nonzero lin char} There exists a non-zero constant $c_\tau$ such that for any $f_0 \in C_c^\infty(\BG(\BF))$  and $\wt{f}_0=f'_0\otimes \Phi_0\in C_c^\infty(\G'(\BF)\times \BF_{n})$  matching $f_0$ (in the sense of Definition \ref{Def: transfer si}), we have 
   \[
I^{(\eta,\eta)}_{\tau_{0}}(\wt{f}_0) \neq c_\tau J_{\tau}(f_0).
   \]
where the LHS is as in \eqref{eqn: linear local rel char}.

 In particular, since there exists $f_0 \in C_c^\infty(\BG(\BF))$ with support in the regular semisimple locus  such that $J_{\tau}(f_0)\neq 0$ and the smooth transfer of such $f_0$ exists (Proposition \ref{prop:tr infty}), there exists $\wt{f}_{0}\in C_c^\infty(\G'(\BF)\times \BF_{n})$  with regular elliptic support such that  $I^{(\eta,\eta)}_{\tau_{0}}(\wt{f}_{0})\neq 0$.
   
\end{Thm}

  We first show that one can globalize $\tau$ to a cuspidal automorphic representation $\pi$ with non-vanishing $\rH$-period.

\begin{Prop}\label{Prop:global} 
Let $E/F$ be a  quadratic extension of number fields with a non-split archimedean place $v_0$ of $F$, so that $E_v/F_v\simeq \BE/\BF$ and we fix such an isomorphism. 
Let $V=W_1\oplus W_2$ be an $E/F$-Hermitian space equipped with an orthogonal decomposition with $\dim W_1=\dim W_2=n$ such that it is positive definite at all non-split archimedean places $v_0$, so that  the base change to $E_v/F_v$ of the decomposition of $E/F$-Hermitian space space $V=W_1\oplus W_2$ is isomorphic to  $\BV=\BW_1\oplus \BW_2$. Fix an isomorphism to define isomorphisms of the groups $\G(F_v)\simeq \BG(\BF)$ etc.. Let $S$ be a finite subset of places containing all archimedean places and two non-archimedean split places $v_1,v_2$.  Let  $K=\prod_{v\notin S}K_v\subset \G(\BA^S)$ a compact open subgroup. 

Then there exists a cuspidal automorphic representation $\pi$ of $\G=\U(V)$, satisfying the following conditions
\begin{enumerate}
\item $\pi_{v_0}=\tau$.
\item $\pi_{v_1}$ is a supercuspidal representation.
\item $\pi_v^{K_v}\neq 0$ for every $v\notin S$.
\item $\pi$ is distinguished by $\rH=\U(W_1)\times\U(W_2)$.
\end{enumerate}

\end{Prop}

\begin{proof}
We use a simple relative trace formula for the triple $(\G,\rH,\rH)$.  Fix a supercuspidal representation $\pi_{v_1}$ of $\G(F_{v_1})$. Let $f=\otimes f_v$ be an elliptic nice test function (cf. Definition \ref{elliptic test functions}) such that 
\begin{itemize}
\item $f_{v_0}$ is a matrix coefficient of $\pi_{v_0}=\tau$ such that the local character $J_{\pi_{v_0}}$ is represented by the orbital integral of $f_{v_0}$.
\item
 $f_{v_1}$ (upon integrating over the center of $\G(F_{v_1})$) is a matrix coefficient of $\pi_{v_1}$,
 \item $f_{v_2}={\bf 1}_{\Xi}$ for some (non-empty) compact open subset $\Xi$ on the regular elliptic locus of $\G(F_{v_2})$.
 \end{itemize}
There exists a regular elliptic element $\gamma_{v_0}\in \G(F_{v_0})$ such that
$$
\Orb^{\rH\times\rH}( f_{v_0},\gamma_{v_0})\neq 0.
$$
Similarly, by the ellipticity of the local character for $\pi_{v_1}$ and the local character being represented by the orbital integral of a matrix coefficient,  we can further assume that  there exists a regular elliptic element $\gamma_{v_1}\in \G(F_{v_1})$ such that
$$
\Orb^{\rH\times\rH}( f_{v_1},\gamma_{v_1})\neq 0.
$$
We also have
$$
\Orb^{\rH\times\rH}( f_{v_2},\gamma_{v_2})\neq 0
$$
for some (necessarily regular elliptic) $\gamma_{v_2}\in\Xi\subset \G(F_{v_2})$.

We now apply the simple relative trace formula, Proposition \ref{Prop: Simple RTF unitary}, to obtain an identity
     \begin{equation}\label{eq: simple RTF}
\sum_{\pi} J_\pi(f)=\sum_{\gamma}\vol([\rH_\gamma]) \Orb^{\rH\times\rH}( f,\gamma),
\end{equation}
where the LHS runs over all cuspidal automorphic representations $\pi$ with the given local components at the places $v_0,v_1$. By weak approximation, there exists a $\gamma^\circ\in \G(F)$ simultaneously approximating $\gamma_{v_i}$ at $v_i$ for $i=0,1,2$. We thus could assume such a $\gamma^\circ\in \G(F)$  satisfies $$
\Orb^{\rH\times\rH}( f_{v_1},\gamma^\circ)\neq 0, \quad i=0,1,2.
$$
Now for $v\in S\setminus\{v_0,v_1,v_2\}$, we choose  $f_v$ such that $\Orb^{\rH\times\rH}(f_{v},\gamma^\circ)\neq 0$. At $v\notin S$, we choose 
bi-$K_v$-invariant function $f_v$ such that $\Orb^{\rH\times\rH}(f_{v},\gamma^\circ)\neq 0$; this is possible, for example, by choosing ${\bf 1}_{K_v\gamma^\circ K_v}$.
Now the function $f$ satisfies $$
\Orb^{\rH\times\rH}(f,\gamma^\circ)\neq 0.
$$
We now modify $f_{v_2}$. We can shrink $\Xi$ so that $\gamma^\circ\in \Xi$  (hence $\Orb(\gamma^\circ, f)\neq 0$ remains true)  but  $$
\Orb^{\rH\times\rH}(f,\gamma)= 0.
$$
for every $\gamma\in \G(F)$ that is not in the same orbit of $\gamma^\circ$ (cf. proof of \cite[Theorem 7.4]{LXZfund}).  The RHS of \eqref{eq: simple RTF}
does not vanish for such $f$ and hence the LHS does not vanish. Therefore there exists $\pi$ such that $J_{\pi}(f)\neq 0$ together with other desired properties.  The proof is complete.
\end{proof}

We now prove Theorem \ref{thm:arch nonzero lin char}. We take a quadratic number field extension $E/F$  such that it has at least one non-split archimedean place,  it is unramified at all non-archimedean places, and all places with small residue characteristic are all split (namely, those $v$ of $F$ such that $p\leq \max\{e(v/p)+1,2\}$). Obviously such a quadratic number field extension $E/F$ exists. Let $W$ be an $E/F$-Hermitian space of $\dim W=n$ such that it is positive definite at all non-split archimedean places, and split at all non-split non-archimedean places. Such  $W$ exists by Hasse principle and the fact that the number of non-split archimedean places is even. Let $V=W\oplus W$ and let the decomposition $W_1\oplus W_2$ be $W\oplus W$. Let $S$ be the union of all archimedean places and exactly two split non-archimedean places $v_1,v_2$. Let $v_0$ be a fixed  non-split archimedean place. At a non-split $v\notin S$, we take a self-dual lattice $L$ in $W_v$ and take $L\oplus L\subset V_v$, and let $K_v\subset \G(F_v)$ be its stabilizer (a hyperspecial subgroup).  We then apply Proposition \ref{Prop:global} to obtain a cuspidal automorphic representation $\pi$ such that $\pi_{v_0}=\tau$. 
Then there exists an elliptic nice test function $f=\bigotimes_v f_v$  such that 
$f_v={\bf 1}_{K_v}$ for  non-split $v\notin S$,  $f_{v_1}$ has regular semisimple support and 
\begin{equation}\label{eq: J v0}
J_{\pi}(f) \neq 0.
\end{equation}

Now for any smooth transfer $\wt f'=\bigotimes_v\wt f'_v$ of $f$, we can apply Corollary \ref{Cor: si comparison of global rel char} to obtain
$$
2I_{\pi_{0}}^{(\eta,\eta)}(\wt{f},0)+2I_{\pi_{0}\otimes \eta}^{(\eta,\eta)}(\wt{f},0)=\sum_{\pi'} J_{\pi'}(f),
$$
where the RHS runs over all representation $\pi'$ in the near equivalence class of  $\pi$. Note that now  at every  non-split $v\notin S$, $f_v={\bf 1}_{K_v}$ for the hyperspecial compact open $K_v$, hence $\pi_v'\simeq \pi_v$. At a non-split archimedean place $v$, since $\G(F_v)$ is a compact unitary group, the L-packet of $\pi_v$ is a singleton  and hence we also have $\pi_v'\simeq \pi_v$. The remaining places are split and we have $\pi_v'\simeq \pi_v$ too. Therefore the RHS consists of only one term, namely the one for $\pi$:
$$
2I_{\pi_{0}}^{(\eta,\eta)}(\wt{f},0)+2I_{\pi_{0}\otimes \eta}^{(\eta,\eta)}(\wt{f},0)=J_{\pi}(f).
$$
By the factorization property the LHS equals
$$
\mathcal{L}^{(\eta,1,\eta)}_{\pi_0}\left(\prod_v I^{(\eta,\eta),\natural}_{\pi_{0,v}}(\wt{f}_v)+\prod_vI^{(\eta,\eta),\natural}_{\pi_{0,v}\otimes\eta_{v}}(\wt{f}_v)\right).
$$
By Lemma \ref{lem: branc compact}, we have
$$
 I^{(\eta,\eta),\natural}_{\pi_{0,v_0}}(\wt{f}_{v_0})=I^{(\eta,\eta),\natural}_{\pi_{0,v_0}\otimes\eta_{v_0}}(\wt{f}_{v_0}).
$$ 
We rewrite the LHS as 
\begin{equation}\label{eq: LHS}
\mathcal{L}^{(\eta,1,\eta)}_{\pi_0}  I^{(\eta,\eta),\natural}_{\pi_{0,v_0}}(\wt{f}_{v_0}) \left(\prod_{v\neq v_0} I^{(\eta,\eta),\natural}_{\pi_{0,v}}(\wt{f}_v)+\prod_{v\neq v_0}I^{(\eta,\eta),\natural}_{\pi_{0,v}\otimes\eta_{v}}(\wt{f}_v)\right).
\end{equation}
The RHS can also be partially factorized, using the local multiplicity one at $v_0$ of Lemma \ref{lem: branc compact}. We write the (non-zero) global period integral as $ \sP_{\rH}=\ell_{v_0}\otimes \ell^{v_0}\in \Hom_{\rH(\BA)}(\pi,\BC)\simeq \Hom_{\rH(F_{v_0})}(\pi_{v_0},\BC)\otimes\Hom_{\rH(\BA^{v_0})}(\pi^{v_0},\BC)$, and obtain
\begin{equation}\label{eq: RHS1}
J_\pi(f)=J_{\pi_{v_0}}(f_{v_0}) J_{\pi^{v_0}}(f^{v_0}).
\end{equation}
Comparing both sides, there exist (non-unique) constants $c_1,c_2=J_{\pi^{v_0}}(f^{v_0})$ such that 
$$
c_1I^{(\eta,\eta),\natural}_{\pi_{0,v_0}}(\wt{f}_{v_0}) 
=c_2J_{\pi_{v_0}}(f_{v_0})
$$holds for all matching test functions $f_{v_0}, \wt{f}_{v_0}$.

We need to show that $c_1,c_2$ can be chosen to be non-zero. Since $J_{\pi_{v_0}}(f_{v_0})\neq0$ for some choice of the test function $f_{v_0}$ with regular semisimple and such test function admits smooth transfer $\wt{f}_{v_0}$ by Proposition \ref{prop:tr infty}, 
it suffices to show that the constant $c_2=J_{\pi^{v_0}}(f^{v_0})$ is non-zero for some $f=\bigotimes_{v\neq v_0} f_v$ that admits a smooth transfer. To show this, by the construction of $\pi$ we know that there exists $f=\bigotimes_v f_v$ with $f_v={\bf 1}_{K_v}$ at each non-split non-archimedean place $v$ such that
$$ J_\pi(f)\neq 0.
$$
Similarly to \eqref{eq: RHS1} using the local multiplicity one at non-split $v\mid\infty$, we obtain
\begin{equation*}
J_\pi(f)= J_{\pi^{T}}(f^{T}) \prod_{v\in T}J_{\pi_{v}}(f_{v}).
\end{equation*}
where $T$ is the finite set of all non-split archimedean places.

Since we do not know the existence of smooth transfer at  non-split archimedean places For the moment, we modify the local components at every non-split archimedean place $v\neq v_0$: we choose
$f_{v}$ to have regular semi-simple (necessarily elliptic) support and such that $J_{\pi_v}(f_v)\neq 0$. As before, this is possible due to the regularity of the local character $J_{\pi_v}$ at $v\in T$. Applying Proposition \ref{prop:tr infty} again, such test functions admit smooth transfers and therefore the non-vanishing  $ J_\pi(f)\neq 0 $ remains true for our modified test function. This shows that $c_2\neq0$
and completes the proof of Theorem \ref{thm:arch nonzero lin char}.\qed

\bibliographystyle{alpha}

\bibliography{bibs}

\end{document}